\documentclass[12pt,twoside]{article}
\usepackage{inputenc}
\usepackage[english]{babel} 
\usepackage{amsfonts}
\usepackage{amsmath}
\usepackage{wrapfig}
\usepackage{amssymb}
\usepackage{amsthm}
\usepackage{graphicx}          
\usepackage{hyperref}
\usepackage{epsfig, epsf,epstopdf}
\usepackage{tkz-graph}
\usepackage{xypic}
\usepackage[matrix,arrow,curve,arc,all]{xy}
\usepackage{tikz}

\newcommand{\lbr}[1]
{
\left[ #1  \right]_{\mathcal{L}}
}

\newcommand{\lpkbr}[1]
{
\left[ #1  \right]_{\mathcal{L}}'
}

\newcommand{\lskbr}[1]
{
{\left[ #1  \right]_{\mathcal{L}}}''
}

\newcommand{\kbr}[1]
{
\left\langle #1  \right\rangle
}

\newcommand{\jbr}[1]
{
\mathcal{X}\left( #1  \right)}

\newcommand{\kupbr}[1]
{
{\left|#1 \right\rangle_{A_2}}
}

\newcommand{\moybr}[1]
{
\left\langle #1  \right\rangle_{n}
}

\newcommand{\moypbr}[1]
{
P_n\left( #1  \right)
}

\newcommand{\arbr}[1]
{
\mathcal{A}\left[ #1  \right]
}

\newcommand{\narbr}[1]
{
\mathcal{W}\left[ #1  \right]
}

\newcommand{\unknot}
{
\mathcal{O}}

\usepackage{tikz}
\usetikzlibrary{arrows,shapes,positioning}
\usetikzlibrary{decorations.markings}
\tikzstyle arrowstyle=[scale=1]
\tikzstyle directed=[postaction={decorate,decoration={markings,
    mark=at position .6 with {\arrow[arrowstyle]{angle 90}}}}]
\tikzstyle reverse directed=[postaction={decorate,decoration={markings,
    mark=at position .4 with {\arrowreversed[arrowstyle]{angle 90};}}}]

    \tikzstyle final directed=[postaction={decorate,decoration={markings,
    mark=at position .98 with {\arrow[arrowstyle]{angle 90}}}}]
\tikzstyle reverse final directed=[postaction={decorate,decoration={markings,
    mark=at position .98 with {\arrowreversed[arrowstyle]{angle 90};}}}]

    \tikzstyle double directed=[postaction={decorate,decoration={
    markings,
    mark=at position .60 with {\arrow[arrowstyle]{angle 90}};,
    mark=at position 0.95 with {\arrow[arrowstyle]{angle 90}};}}]

\tikzstyle double reverse directed=[postaction={decorate,decoration={markings,
    mark=at position .60 with {\arrowreversed[arrowstyle]{angle 90}};,
    mark=at position .95 with {\arrowreversed[arrowstyle]{angle 90}};}}]

\newcommand{\Vone}[1]
{
\begin{tikzpicture}[scale=#1, baseline=0pt]
 \coordinate (Right) at (1.414,1.414) ;
    \coordinate (Left) at (-1.414,1.414) ;
    \coordinate (Bottom) at (0,-2) ;
    \coordinate (BCCenter) at (0,0) ;
\draw [densely dashed, ultra thick] (BCCenter) circle [radius=2];
\draw[ultra thick, reverse directed]  (Left) to   (BCCenter);
\draw[ultra thick,  directed]     (Right) to (BCCenter);
\draw[thin] (Bottom) --  (BCCenter) node[midway,right]{\mbox{\textbf{\emph{~}}}};
\draw [fill=white, ultra thick] (BCCenter) circle [radius=0.25];
\end{tikzpicture}
}

\newcommand{\Vtwo}[1]
{
\begin{tikzpicture}[scale=#1, baseline=0pt]
 \coordinate (Right) at (1.414,1.414) ;
    \coordinate (Left) at (-1.414,1.414) ;
    \coordinate (Bottom) at (0,-2) ;
    \coordinate (BCCenter) at (0,0) ;
\draw [densely dashed, ultra thick] (BCCenter) circle [radius=2];
\draw[ultra thick, reverse directed]  (Left) to   (BCCenter);
\draw[ultra thick,  directed]     (Right) to (BCCenter);
\draw[thin] (Bottom) --  (BCCenter) node[midway,right]{\mbox{\textbf{\emph{~}}}};
\draw [fill=black, ultra thick] (BCCenter) circle [radius=0.25];
 \end{tikzpicture}
}
\newcommand{\Vthree}[1]
{
\begin{tikzpicture}[scale=#1, baseline=0pt]
 \coordinate (Right) at (1.414,1.414) ;
    \coordinate (Left) at (-1.414,1.414) ;
    \coordinate (Bottom) at (0,-2) ;
    \coordinate (BCCenter) at (0,0) ;
\draw [densely dashed, ultra thick] (BCCenter) circle [radius=2];
\draw[ultra thick,  directed]  (Left) to   (BCCenter);
\draw[ultra thick, reverse directed]     (Right) to (BCCenter);
\draw[thin] (Bottom) --  (BCCenter) node[midway,right]{\mbox{\textbf{\emph{~}}}};
\draw [fill=white, ultra thick] (BCCenter) circle [radius=0.25];
\end{tikzpicture}
}

\newcommand{\Vfour}[1]
{
\begin{tikzpicture}[scale=#1, baseline=0pt]
 \coordinate (Right) at (1.414,1.414) ;
    \coordinate (Left) at (-1.414,1.414) ;
    \coordinate (Bottom) at (0,-2) ;
    \coordinate (BCCenter) at (0,0) ;
\draw [densely dashed, ultra thick] (BCCenter) circle [radius=2];
\draw[ultra thick,  directed]  (Left) to   (BCCenter);
\draw[ultra thick, reverse directed]     (Right) to (BCCenter);
\draw[thin] (Bottom) --  (BCCenter) node[midway,right]{\mbox{\textbf{\emph{~}}}};
\draw [fill=black, ultra thick] (BCCenter) circle [radius=0.25];
\end{tikzpicture}
}

\newcommand{\Vfive}[1]
{
\begin{tikzpicture}[scale=#1, baseline=0pt]
 \coordinate (Right) at (1.414,1.414) ;
    \coordinate (Left) at (-1.414,1.414) ;
    \coordinate (Bottom) at (0,-2) ;
    \coordinate (BCCenter) at (0,0) ;
\draw [densely dashed, ultra thick] (BCCenter) circle [radius=2];
\draw[ultra thick, reverse directed]  (Left) to   (BCCenter);
\draw[ultra thick, reverse directed]     (Right) to (BCCenter);
\draw[thin, reverse directed]  (Bottom) --  (BCCenter) node[midway,right]{\mbox{\textbf{\emph{~}}}};
\draw [fill=white, ultra thick] (BCCenter) circle [radius=0.25];
\end{tikzpicture}
}

\newcommand{\Vsix}[1]
{
\begin{tikzpicture}[scale=#1, baseline=0pt]
 \coordinate (Right) at (1.414,1.414) ;
    \coordinate (Left) at (-1.414,1.414) ;
    \coordinate (Bottom) at (0,-2) ;
    \coordinate (BCCenter) at (0,0) ;
\draw [densely dashed, ultra thick] (BCCenter) circle [radius=2];
\draw[ultra thick, reverse directed]  (Left) to   (BCCenter);
\draw[ultra thick, reverse directed]     (Right) to (BCCenter);
\draw[thin, reverse directed]  (Bottom) --  (BCCenter) node[midway,right]{\mbox{\textbf{\emph{~}}}};
\draw [fill=black, ultra thick] (BCCenter) circle [radius=0.25];
\end{tikzpicture}
}

\newcommand{\Vseven}[1]
{
\begin{tikzpicture}[scale=#1, baseline=0pt]
 \coordinate (Right) at (1.414,1.414) ;
    \coordinate (Left) at (-1.414,1.414) ;
    \coordinate (Bottom) at (0,-2) ;
    \coordinate (BCCenter) at (0,0) ;
\draw [densely dashed, ultra thick] (BCCenter) circle [radius=2];
\draw[ultra thick,  directed]  (Left) to   (BCCenter);
\draw[ultra thick,  directed]     (Right) to (BCCenter);
\draw[thin, directed]  (Bottom) --  (BCCenter) node[midway,right]{\mbox{\textbf{\emph{~}}}};
\draw [fill=white, ultra thick] (BCCenter) circle [radius=0.25];
\end{tikzpicture}
}

\newcommand{\Veight}[1]
{
\begin{tikzpicture}[scale=#1, baseline=0pt]
 \coordinate (Right) at (1.414,1.414) ;
    \coordinate (Left) at (-1.414,1.414) ;
    \coordinate (Bottom) at (0,-2) ;
    \coordinate (BCCenter) at (0,0) ;
\draw [densely dashed, ultra thick] (BCCenter) circle [radius=2];
\draw[ultra thick,  directed]  (Left) to   (BCCenter);
\draw[ultra thick,  directed]     (Right) to (BCCenter);
\draw[thin, directed]  (Bottom) --  (BCCenter) node[midway,right]{\mbox{\textbf{\emph{~}}}};
\draw [fill=black, ultra thick] (BCCenter) circle [radius=0.25];
\end{tikzpicture}
}

\newcommand{\RFApos}[1]
{
\begin{tikzpicture} [scale=#1, baseline=0pt]
 \coordinate (Right) at (1.414,-1.414) ;
    \coordinate (Left) at (-1.414,-1.414) ;
    \coordinate (Top) at (0,0.5) ;
    \coordinate (Bottom) at (0,-0.5) ;
    \coordinate (BCCenter) at (0,0) ;
    \coordinate (SCCenter) at (0,1) ;
\draw [densely dashed, ultra thick] (BCCenter) circle [radius=2];
\draw [ultra thick] (SCCenter) circle [radius=0.5];
\draw[ultra thick,  directed]  (Left) to [out=45,in=180]  (Bottom);
\draw[ultra thick,  directed]  (Bottom) to [out=0,in=135]  (Right);
\draw[thin] (Bottom) --  (Top) node[midway,right]{\mbox{\textbf{\emph{~}}}};
\draw[ultra thick, reverse final directed]    (Top) to [out=180,in=270] (-0.5,1);
\draw[ultra thick, reverse final directed]    (0,1.5) to [out=0,in=90] (0.5,1);
\draw [fill=white, ultra thick] (Top) circle [radius=0.15];
\draw [fill=white, ultra thick] (Bottom) circle [radius=0.15];
\end{tikzpicture}
}

\newcommand{\RFBpos}[1]
{
\begin{tikzpicture} [scale=#1, baseline=0pt]
 \coordinate (RightBottom) at (1.414,-1.414) ;
    \coordinate (LeftBottom) at (-1.414,-1.414) ;
     \coordinate (RightTop) at (0.5,0) ;
    \coordinate (LeftTop) at (-.5,0) ;
    \coordinate (Top) at (0,1.5) ;
        \coordinate (BCCenter) at (0,0) ;
\draw [densely dashed, ultra thick] (BCCenter) circle [radius=2];
\draw[ultra thick,  directed]  (LeftBottom) to [out=45,in=270]  (LeftTop);
\draw[ultra thick, reverse directed]  (RightBottom) to [out=135,in=270]  (RightTop);
\draw[ultra thick,  directed]  (RightTop) to [out=90,in=0]  (Top);
\draw[ultra thick, reverse directed]  (LeftTop) to [out=90,in=180]  (Top);
\draw[thin, reverse directed]  (LeftTop) --  (RightTop) node[midway,above=.1]{\mbox{\textbf{\emph{~}}}};
\draw [fill=black, ultra thick] (LeftTop) circle [radius=0.15];
\draw [fill=black, ultra thick] (RightTop) circle [radius=0.15];
\end{tikzpicture}
}

\newcommand{\RFAneg}[1]
{
\begin{tikzpicture} [scale=#1, baseline=0pt]
 \coordinate (RightBottom) at (1.414,-1.414) ;
    \coordinate (LeftBottom) at (-1.414,-1.414) ;
     \coordinate (RightTop) at (0.5,0) ;
    \coordinate (LeftTop) at (-.5,0) ;
    \coordinate (Top) at (0,1.5) ;
        \coordinate (BCCenter) at (0,0) ;
\draw [densely dashed, ultra thick] (BCCenter) circle [radius=2];
\draw[ultra thick,  directed]  (LeftBottom) to [out=45,in=270]  (LeftTop);
\draw[ultra thick, reverse directed]  (RightBottom) to [out=135,in=270]  (RightTop);
\draw[ultra thick,  directed]  (RightTop) to [out=90,in=0]  (Top);
\draw[ultra thick, reverse directed]  (LeftTop) to [out=90,in=180]  (Top);
\draw[thin, reverse directed]  (LeftTop) --  (RightTop) node[midway,above=.1]{\mbox{\textbf{\emph{~}}}};
\draw [fill=white, ultra thick] (LeftTop) circle [radius=0.15];
\draw [fill=white, ultra thick] (RightTop) circle [radius=0.15];
\end{tikzpicture}
}

\newcommand{\RFBneg}[1]
{
\begin{tikzpicture} [scale=#1, baseline=0pt]
 \coordinate (Right) at (1.414,-1.414) ;
    \coordinate (Left) at (-1.414,-1.414) ;
    \coordinate (Top) at (0,0.5) ;
    \coordinate (Bottom) at (0,-0.5) ;
    \coordinate (BCCenter) at (0,0) ;
    \coordinate (SCCenter) at (0,1) ;
\draw [densely dashed, ultra thick] (BCCenter) circle [radius=2];
\draw [ultra thick] (SCCenter) circle [radius=0.5];
\draw[ultra thick,  directed]  (Left) to [out=45,in=180]  (Bottom);
\draw[ultra thick,  directed]  (Bottom) to [out=0,in=135]  (Right);
\draw[thin] (Bottom) --  (Top) node[midway,right]{\mbox{\textbf{\emph{~}}}};
\draw[ultra thick, reverse final directed]    (Top) to [out=180,in=270] (-0.5,1);
\draw[ultra thick, reverse final directed]    (0,1.5) to [out=0,in=90] (0.5,1);
\draw [fill=black, ultra thick] (Top) circle [radius=0.15];
\draw [fill=black, ultra thick] (Bottom) circle [radius=0.15];
\end{tikzpicture}
}
\newcommand{\RFempty}[1]
{
\begin{tikzpicture} [scale=#1, baseline=0pt]
 \coordinate (RightBottom) at (1.414,-1.414) ;
    \coordinate (LeftBottom) at (-1.414,-1.414) ;
     \coordinate (RightTop) at (0.5,0) ;
    \coordinate (LeftTop) at (-.5,0) ;
    \coordinate (Top) at (0,1.5) ;
        \coordinate (BCCenter) at (0,0) ;
        \coordinate (Label) at (0,0.5) ;
\draw [densely dashed, ultra thick] (BCCenter) circle [radius=2];
\draw[ultra thick,  directed]  (LeftBottom) to [out=45,in=270]  (LeftTop);
\draw[ultra thick]  (RightBottom) to [out=135,in=270]  (RightTop);
\draw[ultra thick]  (RightTop) to [out=90,in=0]  (Top);
\draw[ultra thick]  (LeftTop) to [out=90,in=180]  (Top);
\end{tikzpicture}
}

\newcommand{\RSBAcoh}[1]
{
\begin{tikzpicture} [scale=#1, baseline=0pt]
 \coordinate (RightBottom) at (1.414,-1.414) ;
    \coordinate (LeftBottom) at (-1.414,-1.414) ;
     \coordinate (RightCenterBottom) at (.5,-1) ;
    \coordinate (LeftCenterBottom) at (-.5,-1) ;
 \coordinate (RightTop) at (1.414,1.414) ;
    \coordinate (LeftTop) at (-1.414,1.414) ;
     \coordinate (RightCenterTop) at (.5,1) ;
    \coordinate (LeftCenterTop) at (-.5,1) ;
    \coordinate (BCCenter) at (0,0) ;
\draw [densely dashed, ultra thick] (BCCenter) circle [radius=2];
\draw[ultra thick,  directed]  (LeftTop) to [out=315,in=180]  (LeftCenterTop);
\draw[ultra thick, reverse directed]  (RightCenterTop) to [out=210,in=330]   (LeftCenterTop);
\draw[ultra thick,  directed]  (RightCenterTop) to [out=0,in=225]  (RightTop);
\draw[ultra thick,  directed]  (LeftBottom) to [out=45,in=180]  (LeftCenterBottom);
\draw[ultra thick, reverse directed]  (RightCenterBottom) to [out=150,in=30]   (LeftCenterBottom);
\draw[ultra thick,  directed]  (RightCenterBottom) to [out=0,in=135]  (RightBottom);
\draw[thin] (LeftCenterBottom) --  (LeftCenterTop) node[midway,right]{\mbox{\textbf{\emph{~}}}};
\draw[thin] (RightCenterBottom) --  (RightCenterTop) node[midway,right]{\mbox{\textbf{\emph{~}}}};
\draw [fill=black, ultra thick] (LeftCenterBottom) circle [radius=0.15];
\draw [fill=black, ultra thick] (LeftCenterTop) circle [radius=0.15];
\draw [fill=white, ultra thick] (RightCenterBottom) circle [radius=0.15];
\draw [fill=white, ultra thick] (RightCenterTop) circle [radius=0.15];
\end{tikzpicture}
}

\newcommand{\RSABcoh}[1]
{
\begin{tikzpicture} [scale=#1, baseline=0pt]
 \coordinate (RightBottom) at (1.414,-1.414) ;
    \coordinate (LeftBottom) at (-1.414,-1.414) ;
     \coordinate (RightTop) at (1.414,1.414) ;
    \coordinate (LeftTop) at (-1.414,1.414) ;
    \coordinate (RightCenter) at (1.414,0) ;
    \coordinate (LeftCenter) at (-1.414,0) ;
    \coordinate (BCCenter) at (0,0) ;
    \coordinate (LCC) at (-0.5,0) ;
     \coordinate (RCC) at (0.5,0) ;
     \coordinate (BCC) at (0,-0.5) ;
      \coordinate (TCC) at (0,0.5) ;
\draw [densely dashed, ultra thick] (BCCenter) circle [radius=2];
\draw [ultra thick] (BCCenter) circle [radius=0.5];
\draw[ultra thick, reverse final directed]    (RCC) to [out=90,in=0] (TCC);
\draw[ultra thick, reverse final directed]    (RCC) to [out=270,in=0] (BCC);

\draw[ultra thick,  directed]  (LeftBottom) to  (LeftCenter);
\draw[ultra thick, reverse directed]  (LeftCenter) to   (LeftTop);

\draw[ultra thick, reverse directed]  (RightBottom) to   (RightCenter);
\draw[ultra thick,  directed]  (RightCenter) to  (RightTop);

\draw[thin, directed]  (LCC) -- (LeftCenter)  node[midway,above=.1]{\mbox{\textbf{\emph{~}}}};
\draw[thin, directed]  (RightCenter) --  (RCC) node[midway,above=.1]{\mbox{\textbf{\emph{~}}}};

\draw [fill=black, ultra thick] (RightCenter) circle [radius=0.15];
\draw [fill=black, ultra thick] (RCC) circle [radius=0.15];
\draw [fill=white, ultra thick] (LeftCenter) circle [radius=0.15];
\draw [fill=white, ultra thick] (LCC) circle [radius=0.15];
\end{tikzpicture}
}

\newcommand{\RSBBcoh}[1]
{
\begin{tikzpicture} [scale=#1, baseline=0pt]
 \coordinate (RightBottom) at (1.414,-1.414) ;
    \coordinate (LeftBottom) at (-1.414,-1.414) ;
     \coordinate (RightTop) at (1.414,1.414) ;
    \coordinate (LeftTop) at (-1.414,1.414) ;
    \coordinate (RightCenter) at (1.414,0) ;
    \coordinate (LeftCenter) at (-1.414,0) ;
    \coordinate (BCCenter) at (0,0) ;
    \coordinate (LCC) at (-0.5,0) ;
     \coordinate (RCC) at (0.5,0) ;
     \coordinate (BCC) at (-.5,-1) ;
      \coordinate (TCC) at (-.5,1) ;
\draw [densely dashed, ultra thick] (BCCenter) circle [radius=2];

\draw[ultra thick, reverse  directed]    (RCC) to [out=90,in=0] (TCC);
\draw[ultra thick, reverse  directed]    (RCC) to [out=270,in=0] (BCC);

\draw[ultra thick,  directed]  (LeftBottom) to [out=30,in=210] (BCC);
\draw[ultra thick, reverse directed]  (TCC) to  [out=150,in=330] (LeftTop);

\draw[ultra thick, reverse directed]  (RightBottom) to   (RightCenter);
\draw[ultra thick,  directed]  (RightCenter) to  (RightTop);

\draw[thin] (TCC) -- (BCC)  node[midway,right]{\mbox{\textbf{\emph{~}}}};
\draw[thin, directed]  (RightCenter) --  (RCC) node[midway,above=.1]{\mbox{\textbf{\emph{~}}}};

\draw [fill=black, ultra thick] (BCC) circle [radius=0.15];
\draw [fill=black, ultra thick] (RCC) circle [radius=0.15];
\draw [fill=black, ultra thick] (TCC) circle [radius=0.15];
\draw [fill=black, ultra thick] (RightCenter) circle [radius=0.15];
\end{tikzpicture}
}

\newcommand{\RSAAcoh}[1]
{
\begin{tikzpicture} [scale=#1, baseline=0pt]
 \coordinate (RightBottom) at (-1.414,-1.414) ;
    \coordinate (LeftBottom) at (1.414,-1.414) ;
     \coordinate (RightTop) at (-1.414,1.414) ;
    \coordinate (LeftTop) at (1.414,1.414) ;
    \coordinate (RightCenter) at (-1.414,0) ;
    \coordinate (LeftCenter) at (1.414,0) ;
    \coordinate (BCCenter) at (0,0) ;
    \coordinate (LCC) at (0.5,0) ;
     \coordinate (RCC) at (-0.5,0) ;
     \coordinate (BCC) at (.5,-1) ;
      \coordinate (TCC) at (.5,1) ;
\draw [densely dashed, ultra thick] (BCCenter) circle [radius=2];

\draw[ultra thick, reverse  directed]    (TCC) to [out=180,in=90] (RCC);
\draw[ultra thick, reverse  directed]    (BCC) to [out=180,in=270] (RCC);

\draw[ultra thick, reverse directed]  (LeftBottom) to [out=150,in=330] (BCC);
\draw[ultra thick,  directed]  (TCC) to  [out=30,in=210] (LeftTop);

\draw[ultra thick,  directed]  (RightBottom) to   (RightCenter);
\draw[ultra thick, reverse directed]  (RightCenter) to  (RightTop);

\draw[thin] (TCC) -- (BCC)  node[midway,right]{\mbox{\textbf{\emph{~}}}};
\draw[thin, reverse directed]  (RightCenter) --  (RCC) node[midway,above=.1]{\mbox{\textbf{\emph{~}}}};

\draw [fill=white, ultra thick] (BCC) circle [radius=0.15];
\draw [fill=white, ultra thick] (RCC) circle [radius=0.15];
\draw [fill=white, ultra thick] (TCC) circle [radius=0.15];
\draw [fill=white, ultra thick] (RightCenter) circle [radius=0.15];
\end{tikzpicture}
}

\newcommand{\RSemptycoh}[1]
{
\begin{tikzpicture} [scale=#1, baseline=0pt]
 \coordinate (RightBottom) at (1.414,-1.414) ;
    \coordinate (LeftBottom) at (-1.414,-1.414) ;
     \coordinate (RightCenterBottom) at (.5,-1) ;
    \coordinate (LeftCenterBottom) at (-.5,-1) ;
 \coordinate (RightTop) at (1.414,1.414) ;
    \coordinate (LeftTop) at (-1.414,1.414) ;
     \coordinate (RightCenterTop) at (.5,1) ;
    \coordinate (LeftCenterTop) at (-.5,1) ;
    \coordinate (BCCenter) at (0,0) ;
\draw [densely dashed, ultra thick] (BCCenter) circle [radius=2];
\draw[ultra thick,  directed]  (LeftTop) to [out=315,in=235]  (RightTop);
\draw[ultra thick,  directed]  (LeftBottom) to [out=45,in=135]  (RightBottom);
\end{tikzpicture}
}

\newcommand{\RSABncoh}[1]
{
\begin{tikzpicture} [scale=#1, baseline=0pt]
 \coordinate (RightBottom) at (1.414,-1.414) ;
    \coordinate (LeftBottom) at (-1.414,-1.414) ;
     \coordinate (RightCenterBottom) at (.5,-1) ;
    \coordinate (LeftCenterBottom) at (-.5,-1) ;
 \coordinate (RightTop) at (1.414,1.414) ;
    \coordinate (LeftTop) at (-1.414,1.414) ;
     \coordinate (RightCenterTop) at (.5,1) ;
    \coordinate (LeftCenterTop) at (-.5,1) ;
    \coordinate (BCCenter) at (0,0) ;
\draw [densely dashed, ultra thick] (BCCenter) circle [radius=2];
\draw[ultra thick,  directed]  (LeftTop) to [out=315,in=180]  (LeftCenterTop);
\draw[ultra thick, reverse directed]  (RightCenterTop) to [out=210,in=330]   (LeftCenterTop);
\draw[ultra thick,  directed]  (RightCenterTop) to [out=0,in=225]  (RightTop);
\draw[ultra thick,  directed]  (LeftBottom) to [out=45,in=180]  (LeftCenterBottom);
\draw[ultra thick, reverse directed]  (RightCenterBottom) to [out=150,in=30]   (LeftCenterBottom);
\draw[ultra thick,  directed]  (RightCenterBottom) to [out=0,in=135]  (RightBottom);
\draw[thin] (LeftCenterBottom) --  (LeftCenterTop) node[midway,right]{\mbox{\textbf{\emph{~}}}};
\draw[thin] (RightCenterBottom) --  (RightCenterTop) node[midway,right]{\mbox{\textbf{\emph{~}}}};
\draw [fill=white, ultra thick] (LeftCenterBottom) circle [radius=0.15];
\draw [fill=white, ultra thick] (LeftCenterTop) circle [radius=0.15];
\draw [fill=black, ultra thick] (RightCenterBottom) circle [radius=0.15];
\draw [fill=black, ultra thick] (RightCenterTop) circle [radius=0.15];
\end{tikzpicture}
}

\newcommand{\RSBAncoh}[1]
{
\begin{tikzpicture} [scale=#1, baseline=0pt]
 \coordinate (RightBottom) at (1.414,-1.414) ;
    \coordinate (LeftBottom) at (-1.414,-1.414) ;
     \coordinate (RightTop) at (1.414,1.414) ;
    \coordinate (LeftTop) at (-1.414,1.414) ;
    \coordinate (RightCenter) at (1.414,0) ;
    \coordinate (LeftCenter) at (-1.414,0) ;
    \coordinate (BCCenter) at (0,0) ;
    \coordinate (LCC) at (-0.5,0) ;
     \coordinate (RCC) at (0.5,0) ;
     \coordinate (BCC) at (0,-0.5) ;
      \coordinate (TCC) at (0,0.5) ;
\draw [densely dashed, ultra thick] (BCCenter) circle [radius=2];
\draw [ultra thick] (BCCenter) circle [radius=0.5];
\draw[ultra thick, reverse final directed]    (RCC) to [out=90,in=0] (TCC);
\draw[ultra thick, reverse final directed]    (RCC) to [out=270,in=0] (BCC);

\draw[ultra thick,  directed]  (LeftBottom) to  (LeftCenter);
\draw[ultra thick, reverse directed]  (LeftCenter) to   (LeftTop);

\draw[ultra thick, reverse directed]  (RightBottom) to   (RightCenter);
\draw[ultra thick,  directed]  (RightCenter) to  (RightTop);

\draw[thin, directed]  (LCC) -- (LeftCenter)  node[midway,above=.1]{\mbox{\textbf{\emph{~}}}};
\draw[thin, directed]  (RightCenter) --  (RCC) node[midway,above=.1]{\mbox{\textbf{\emph{~}}}};

\draw [fill=white, ultra thick] (RightCenter) circle [radius=0.15];
\draw [fill=white, ultra thick] (RCC) circle [radius=0.15];
\draw [fill=black, ultra thick] (LeftCenter) circle [radius=0.15];
\draw [fill=black, ultra thick] (LCC) circle [radius=0.15];
\end{tikzpicture}
}

\newcommand{\RSAAncoh}[1]
{
\begin{tikzpicture} [scale=#1, baseline=0pt]
 \coordinate (RightBottom) at (1.414,-1.414) ;
    \coordinate (LeftBottom) at (-1.414,-1.414) ;
     \coordinate (RightTop) at (1.414,1.414) ;
    \coordinate (LeftTop) at (-1.414,1.414) ;
    \coordinate (RightCenter) at (1.414,0) ;
    \coordinate (LeftCenter) at (-1.414,0) ;
    \coordinate (BCCenter) at (0,0) ;
    \coordinate (LCC) at (-0.5,0) ;
     \coordinate (RCC) at (0.5,0) ;
     \coordinate (BCC) at (-.5,-1) ;
      \coordinate (TCC) at (-.5,1) ;
\draw [densely dashed, ultra thick] (BCCenter) circle [radius=2];

\draw[ultra thick, reverse  directed]    (RCC) to [out=90,in=0] (TCC);
\draw[ultra thick, reverse  directed]    (RCC) to [out=270,in=0] (BCC);

\draw[ultra thick,  directed]  (LeftBottom) to [out=30,in=210] (BCC);
\draw[ultra thick, reverse directed]  (TCC) to  [out=150,in=330] (LeftTop);

\draw[ultra thick, reverse directed]  (RightBottom) to   (RightCenter);
\draw[ultra thick,  directed]  (RightCenter) to  (RightTop);

\draw[thin] (TCC) -- (BCC)  node[midway,right]{\mbox{\textbf{\emph{~}}}};
\draw[thin, directed]  (RightCenter) --  (RCC) node[midway,above=.1]{\mbox{\textbf{\emph{~}}}};

\draw [fill=white, ultra thick] (BCC) circle [radius=0.15];
\draw [fill=white, ultra thick] (RCC) circle [radius=0.15];
\draw [fill=white, ultra thick] (TCC) circle [radius=0.15];
\draw [fill=white, ultra thick] (RightCenter) circle [radius=0.15];
\end{tikzpicture}
}

\newcommand{\RSBBncoh}[1]
{
\begin{tikzpicture} [scale=#1, baseline=0pt]
 \coordinate (RightBottom) at (-1.414,-1.414) ;
    \coordinate (LeftBottom) at (1.414,-1.414) ;
     \coordinate (RightTop) at (-1.414,1.414) ;
    \coordinate (LeftTop) at (1.414,1.414) ;
    \coordinate (RightCenter) at (-1.414,0) ;
    \coordinate (LeftCenter) at (1.414,0) ;
    \coordinate (BCCenter) at (0,0) ;
    \coordinate (LCC) at (0.5,0) ;
     \coordinate (RCC) at (-0.5,0) ;
     \coordinate (BCC) at (.5,-1) ;
      \coordinate (TCC) at (.5,1) ;
\draw [densely dashed, ultra thick] (BCCenter) circle [radius=2];

\draw[ultra thick, reverse  directed]    (TCC) to [out=180,in=90] (RCC);
\draw[ultra thick, reverse  directed]    (BCC) to [out=180,in=270] (RCC);

\draw[ultra thick, reverse directed]  (LeftBottom) to [out=150,in=330] (BCC);
\draw[ultra thick,  directed]  (TCC) to  [out=30,in=210] (LeftTop);

\draw[ultra thick,  directed]  (RightBottom) to   (RightCenter);
\draw[ultra thick, reverse directed]  (RightCenter) to  (RightTop);

\draw[thin] (TCC) -- (BCC)  node[midway,right]{\mbox{\textbf{\emph{~}}}};
\draw[thin, reverse directed]  (RightCenter) --  (RCC) node[midway,above=.1]{\mbox{\textbf{\emph{~}}}};

\draw [fill=black, ultra thick] (BCC) circle [radius=0.15];
\draw [fill=black, ultra thick] (RCC) circle [radius=0.15];
\draw [fill=black, ultra thick] (TCC) circle [radius=0.15];
\draw [fill=black, ultra thick] (RightCenter) circle [radius=0.15];
\end{tikzpicture}
}

\newcommand{\RSBArever}[1]
{
\begin{tikzpicture} [scale=#1, baseline=0pt]
 \coordinate (RightBottom) at (1.414,-1.414) ;
    \coordinate (LeftBottom) at (-1.414,-1.414) ;
     \coordinate (RightCenterBottom) at (.5,-1) ;
    \coordinate (LeftCenterBottom) at (-.5,-1) ;
 \coordinate (RightTop) at (1.414,1.414) ;
    \coordinate (LeftTop) at (-1.414,1.414) ;
     \coordinate (RightCenterTop) at (.5,1) ;
    \coordinate (LeftCenterTop) at (-.5,1) ;
    \coordinate (BCCenter) at (0,0) ;
\draw [densely dashed, ultra thick] (BCCenter) circle [radius=2];
\draw[ultra thick,  directed]  (LeftTop) to [out=315,in=180]  (LeftCenterTop);
\draw[ultra thick, reverse directed]  (LeftCenterTop) to [out=330,in=210]   (RightCenterTop);
\draw[ultra thick,  directed]  (RightCenterTop) to [out=0,in=225]  (RightTop);
\draw[ultra thick, reverse directed]  (LeftBottom) to [out=45,in=180]  (LeftCenterBottom);
\draw[ultra thick, reverse directed]  (RightCenterBottom) to [out=150,in=30]   (LeftCenterBottom);
\draw[ultra thick, reverse  directed]  (RightCenterBottom) to [out=0,in=135]  (RightBottom);
\draw[thin, directed]  (LeftCenterBottom) --  (LeftCenterTop) node[midway,right]{\mbox{\textbf{\emph{~}}}};
\draw[thin, reverse directed]  (RightCenterBottom) --  (RightCenterTop) node[midway,right]{\mbox{\textbf{\emph{~}}}};
\draw [fill=black, ultra thick] (LeftCenterBottom) circle [radius=0.15];
\draw [fill=black, ultra thick] (LeftCenterTop) circle [radius=0.15];
\draw [fill=white, ultra thick] (RightCenterBottom) circle [radius=0.15];
\draw [fill=white, ultra thick] (RightCenterTop) circle [radius=0.15];
\end{tikzpicture}
}

\newcommand{\RSABrever}[1]
{
\begin{tikzpicture} [scale=#1, baseline=0pt]
 \coordinate (RightBottom) at (1.414,-1.414) ;
    \coordinate (LeftBottom) at (-1.414,-1.414) ;
     \coordinate (RightTop) at (1.414,1.414) ;
    \coordinate (LeftTop) at (-1.414,1.414) ;
    \coordinate (RightCenter) at (1.414,0) ;
    \coordinate (LeftCenter) at (-1.414,0) ;
    \coordinate (BCCenter) at (0,0) ;
    \coordinate (LCC) at (-0.5,0) ;
     \coordinate (RCC) at (0.5,0) ;
     \coordinate (BCC) at (0,-0.5) ;
      \coordinate (TCC) at (0,0.5) ;
\draw [densely dashed, ultra thick] (BCCenter) circle [radius=2];
\draw [ultra thick] (BCCenter) circle [radius=0.5];
\draw[ultra thick, reverse final directed]    (LCC) to [out=90,in=180] (TCC);
\draw[ultra thick, reverse final directed]    (RCC) to [out=270,in=0] (BCC);

\draw[ultra thick, reverse directed]  (LeftBottom) to  (LeftCenter);
\draw[ultra thick, reverse directed]  (LeftCenter) to   (LeftTop);

\draw[ultra thick,  directed]  (RightBottom) to   (RightCenter);
\draw[ultra thick,  directed]  (RightCenter) to  (RightTop);

\draw[thin] (LCC) -- (LeftCenter)  node[midway,above=.1]{\mbox{\textbf{\emph{~}}}};
\draw[thin] (RightCenter) --  (RCC) node[midway,above=.1]{\mbox{\textbf{\emph{~}}}};

\draw [fill=black, ultra thick] (RightCenter) circle [radius=0.15];
\draw [fill=black, ultra thick] (RCC) circle [radius=0.15];
\draw [fill=white, ultra thick] (LeftCenter) circle [radius=0.15];
\draw [fill=white, ultra thick] (LCC) circle [radius=0.15];
\end{tikzpicture}
}

\newcommand{\RSBBrever}[1]
{
\begin{tikzpicture} [scale=#1, baseline=0pt]
 \coordinate (RightBottom) at (1.414,-1.414) ;
    \coordinate (LeftBottom) at (-1.414,-1.414) ;
     \coordinate (RightTop) at (1.414,1.414) ;
    \coordinate (LeftTop) at (-1.414,1.414) ;
    \coordinate (RightCenter) at (1.414,0) ;
    \coordinate (LeftCenter) at (-1.414,0) ;
    \coordinate (BCCenter) at (0,0) ;
    \coordinate (LCC) at (-0.5,0) ;
     \coordinate (RCC) at (0.5,0) ;
     \coordinate (BCC) at (-.5,-1) ;
      \coordinate (TCC) at (-.5,1) ;
\draw [densely dashed, ultra thick] (BCCenter) circle [radius=2];

\draw[ultra thick,   directed]    (RCC) to [out=90,in=0] (TCC);
\draw[ultra thick, reverse  directed]    (RCC) to [out=270,in=0] (BCC);

\draw[ultra thick, reverse directed]  (LeftBottom) to [out=30,in=210] (BCC);
\draw[ultra thick, reverse directed]  (TCC) to  [out=150,in=330] (LeftTop);

\draw[ultra thick,  directed]  (RightBottom) to   (RightCenter);
\draw[ultra thick,  directed]  (RightCenter) to  (RightTop);

\draw[thin]  (TCC) -- (BCC)  node[midway,right]{\mbox{\textbf{\emph{~}}}};
\draw[thin] (RightCenter) --  (RCC) node[midway,above=.1]{\mbox{\textbf{\emph{~}}}};

\draw [fill=black, ultra thick] (BCC) circle [radius=0.15];
\draw [fill=black, ultra thick] (RCC) circle [radius=0.15];
\draw [fill=black, ultra thick] (TCC) circle [radius=0.15];
\draw [fill=black, ultra thick] (RightCenter) circle [radius=0.15];
\end{tikzpicture}
}

\newcommand{\RSAArever}[1]
{
\begin{tikzpicture} [scale=#1, baseline=0pt]
 \coordinate (RightBottom) at (-1.414,-1.414) ;
    \coordinate (LeftBottom) at (1.414,-1.414) ;
     \coordinate (RightTop) at (-1.414,1.414) ;
    \coordinate (LeftTop) at (1.414,1.414) ;
    \coordinate (RightCenter) at (-1.414,0) ;
    \coordinate (LeftCenter) at (1.414,0) ;
    \coordinate (BCCenter) at (0,0) ;
    \coordinate (LCC) at (0.5,0) ;
     \coordinate (RCC) at (-0.5,0) ;
     \coordinate (BCC) at (.5,-1) ;
      \coordinate (TCC) at (.5,1) ;
\draw [densely dashed, ultra thick] (BCCenter) circle [radius=2];

\draw[ultra thick,   directed]    (TCC) to [out=180,in=90] (RCC);
\draw[ultra thick, reverse  directed]    (BCC) to [out=180,in=270] (RCC);

\draw[ultra thick,  directed]  (LeftBottom) to [out=150,in=330] (BCC);
\draw[ultra thick,  directed]  (TCC) to  [out=30,in=210] (LeftTop);

\draw[ultra thick, reverse directed]  (RightBottom) to   (RightCenter);
\draw[ultra thick, reverse directed]  (RightCenter) to  (RightTop);

\draw[thin, directed]  (TCC) -- (BCC)  node[midway,right]{\mbox{\textbf{\emph{~}}}};
\draw[thin] (RightCenter) --  (RCC) node[midway,above=.1]{\mbox{\textbf{\emph{~}}}};

\draw [fill=white, ultra thick] (BCC) circle [radius=0.15];
\draw [fill=white, ultra thick] (RCC) circle [radius=0.15];
\draw [fill=white, ultra thick] (TCC) circle [radius=0.15];
\draw [fill=white, ultra thick] (RightCenter) circle [radius=0.15];
\end{tikzpicture}
}

\newcommand{\RSemptyrev}[1]
{
\begin{tikzpicture} [scale=#1, baseline=0pt]
 \coordinate (RightBottom) at (1.414,-1.414) ;
    \coordinate (LeftBottom) at (-1.414,-1.414) ;
     \coordinate (RightCenterBottom) at (.5,-1) ;
    \coordinate (LeftCenterBottom) at (-.5,-1) ;
 \coordinate (RightTop) at (1.414,1.414) ;
    \coordinate (LeftTop) at (-1.414,1.414) ;
     \coordinate (RightCenterTop) at (.5,1) ;
    \coordinate (LeftCenterTop) at (-.5,1) ;
    \coordinate (BCCenter) at (0,0) ;
\draw [densely dashed, ultra thick] (BCCenter) circle [radius=2];
\draw[ultra thick,  directed]  (LeftTop) to [out=315,in=235]  (RightTop);
\draw[ultra thick, reverse directed]  (LeftBottom) to [out=45,in=135]  (RightBottom);
\end{tikzpicture}
}

\newcommand{\RSABrev}[1]
{
\begin{tikzpicture} [scale=#1, baseline=0pt]
 \coordinate (RightBottom) at (1.414,-1.414) ;
    \coordinate (LeftBottom) at (-1.414,-1.414) ;
     \coordinate (RightCenterBottom) at (.5,-1) ;
    \coordinate (LeftCenterBottom) at (-.5,-1) ;
 \coordinate (RightTop) at (1.414,1.414) ;
    \coordinate (LeftTop) at (-1.414,1.414) ;
     \coordinate (RightCenterTop) at (.5,1) ;
    \coordinate (LeftCenterTop) at (-.5,1) ;
    \coordinate (BCCenter) at (0,0) ;
\draw [densely dashed, ultra thick] (BCCenter) circle [radius=2];
\draw[ultra thick,  directed]  (LeftTop) to [out=315,in=180]  (LeftCenterTop);
\draw[ultra thick, reverse directed]  (LeftCenterTop) to [out=330,in=210]   (RightCenterTop);
\draw[ultra thick,  directed]  (RightCenterTop) to [out=0,in=225]  (RightTop);
\draw[ultra thick, reverse directed]  (LeftBottom) to [out=45,in=180]  (LeftCenterBottom);
\draw[ultra thick, reverse directed]  (RightCenterBottom) to [out=150,in=30]   (LeftCenterBottom);
\draw[ultra thick, reverse  directed]  (RightCenterBottom) to [out=0,in=135]  (RightBottom);
\draw[thin, directed]  (LeftCenterBottom) --  (LeftCenterTop) node[midway,right]{\mbox{\textbf{\emph{~}}}};
\draw[thin, reverse directed]  (RightCenterBottom) --  (RightCenterTop) node[midway,right]{\mbox{\textbf{\emph{~}}}};
\draw [fill=white, ultra thick] (LeftCenterBottom) circle [radius=0.15];
\draw [fill=white, ultra thick] (LeftCenterTop) circle [radius=0.15];
\draw [fill=black, ultra thick] (RightCenterBottom) circle [radius=0.15];
\draw [fill=black, ultra thick] (RightCenterTop) circle [radius=0.15];
\end{tikzpicture}
}

\newcommand{\RSBArev}[1]
{
\begin{tikzpicture} [scale=#1, baseline=0pt]
 \coordinate (RightBottom) at (1.414,-1.414) ;
    \coordinate (LeftBottom) at (-1.414,-1.414) ;
     \coordinate (RightTop) at (1.414,1.414) ;
    \coordinate (LeftTop) at (-1.414,1.414) ;
    \coordinate (RightCenter) at (1.414,0) ;
    \coordinate (LeftCenter) at (-1.414,0) ;
    \coordinate (BCCenter) at (0,0) ;
    \coordinate (LCC) at (-0.5,0) ;
     \coordinate (RCC) at (0.5,0) ;
     \coordinate (BCC) at (0,-0.5) ;
      \coordinate (TCC) at (0,0.5) ;
\draw [densely dashed, ultra thick] (BCCenter) circle [radius=2];
\draw [ultra thick] (BCCenter) circle [radius=0.5];
\draw[ultra thick, reverse final directed]    (LCC) to [out=90,in=180] (TCC);
\draw[ultra thick, reverse final directed]    (RCC) to [out=270,in=0] (BCC);

\draw[ultra thick, reverse directed]  (LeftBottom) to  (LeftCenter);
\draw[ultra thick, reverse directed]  (LeftCenter) to   (LeftTop);

\draw[ultra thick,  directed]  (RightBottom) to   (RightCenter);
\draw[ultra thick,  directed]  (RightCenter) to  (RightTop);

\draw[thin] (LCC) -- (LeftCenter)  node[midway,above=.1]{\mbox{\textbf{\emph{~}}}};
\draw[thin] (RightCenter) --  (RCC) node[midway,above=.1]{\mbox{\textbf{\emph{~}}}};

\draw [fill=white, ultra thick] (RightCenter) circle [radius=0.15];
\draw [fill=white, ultra thick] (RCC) circle [radius=0.15];
\draw [fill=black, ultra thick] (LeftCenter) circle [radius=0.15];
\draw [fill=black, ultra thick] (LCC) circle [radius=0.15];
\end{tikzpicture}
}

\newcommand{\RSAArev}[1]
{
\begin{tikzpicture} [scale=#1, baseline=0pt]
 \coordinate (RightBottom) at (1.414,-1.414) ;
    \coordinate (LeftBottom) at (-1.414,-1.414) ;
     \coordinate (RightTop) at (1.414,1.414) ;
    \coordinate (LeftTop) at (-1.414,1.414) ;
    \coordinate (RightCenter) at (1.414,0) ;
    \coordinate (LeftCenter) at (-1.414,0) ;
    \coordinate (BCCenter) at (0,0) ;
    \coordinate (LCC) at (-0.5,0) ;
     \coordinate (RCC) at (0.5,0) ;
     \coordinate (BCC) at (-.5,-1) ;
      \coordinate (TCC) at (-.5,1) ;
\draw [densely dashed, ultra thick] (BCCenter) circle [radius=2];

\draw[ultra thick,   directed]    (RCC) to [out=90,in=0] (TCC);
\draw[ultra thick, reverse  directed]    (RCC) to [out=270,in=0] (BCC);

\draw[ultra thick, reverse directed]  (LeftBottom) to [out=30,in=210] (BCC);
\draw[ultra thick, reverse directed]  (TCC) to  [out=150,in=330] (LeftTop);

\draw[ultra thick,  directed]  (RightBottom) to   (RightCenter);
\draw[ultra thick,  directed]  (RightCenter) to  (RightTop);

\draw[thin, reverse directed]  (TCC) -- (BCC)  node[midway,right]{\mbox{\textbf{\emph{~}}}};
\draw[thin] (RightCenter) --  (RCC) node[midway,above=.1]{\mbox{\textbf{\emph{~}}}};

\draw [fill=white, ultra thick] (BCC) circle [radius=0.15];
\draw [fill=white, ultra thick] (RCC) circle [radius=0.15];
\draw [fill=white, ultra thick] (TCC) circle [radius=0.15];
\draw [fill=white, ultra thick] (RightCenter) circle [radius=0.15];
\end{tikzpicture}
}

\newcommand{\RSBBrev}[1]
{
\begin{tikzpicture} [scale=#1, baseline=0pt]
 \coordinate (RightBottom) at (-1.414,-1.414) ;
    \coordinate (LeftBottom) at (1.414,-1.414) ;
     \coordinate (RightTop) at (-1.414,1.414) ;
    \coordinate (LeftTop) at (1.414,1.414) ;
    \coordinate (RightCenter) at (-1.414,0) ;
    \coordinate (LeftCenter) at (1.414,0) ;
    \coordinate (BCCenter) at (0,0) ;
    \coordinate (LCC) at (0.5,0) ;
     \coordinate (RCC) at (-0.5,0) ;
     \coordinate (BCC) at (.5,-1) ;
      \coordinate (TCC) at (.5,1) ;
\draw [densely dashed, ultra thick] (BCCenter) circle [radius=2];

\draw[ultra thick,   directed]    (TCC) to [out=180,in=90] (RCC);
\draw[ultra thick, reverse  directed]    (BCC) to [out=180,in=270] (RCC);

\draw[ultra thick,  directed]  (LeftBottom) to [out=150,in=330] (BCC);
\draw[ultra thick,  directed]  (TCC) to  [out=30,in=210] (LeftTop);

\draw[ultra thick, reverse directed]  (RightBottom) to   (RightCenter);
\draw[ultra thick, reverse directed]  (RightCenter) to  (RightTop);

\draw[thin, directed]  (TCC) -- (BCC)  node[midway,right]{\mbox{\textbf{\emph{~}}}};
\draw[thin] (RightCenter) --  (RCC) node[midway,above=.1]{\mbox{\textbf{\emph{~}}}};

\draw [fill=black, ultra thick] (BCC) circle [radius=0.15];
\draw [fill=black, ultra thick] (RCC) circle [radius=0.15];
\draw [fill=black, ultra thick] (TCC) circle [radius=0.15];
\draw [fill=black, ultra thick] (RightCenter) circle [radius=0.15];
\end{tikzpicture}
}

\newcommand{\RTAABleft}[1]
{
\begin{tikzpicture} [scale=#1, baseline=0pt]
  \coordinate (Left) at (-4,0) ;
  \coordinate (Right) at (4,0) ;
 \coordinate (RightBottom) at (2.828,-2.828) ;
    \coordinate (LeftBottom) at (-2.828,-2.828) ;
 \coordinate (RightTop) at (2.828,2.828) ;
    \coordinate (LeftTop) at (-2.828,2.828) ;
           \coordinate (Righti) at (-1,0) ;
    \coordinate (Lefti) at (-2,-2) ;
          \coordinate (Rightj) at (2,0) ;
    \coordinate (Leftj) at (1,-2) ;
    \coordinate (Rightk) at (1,1.5) ;
    \coordinate (Leftk) at (-1,1.5) ;
    \coordinate (BCCenter) at (0,0) ;
\draw [densely dashed, ultra thick] (BCCenter) circle [radius=4];
\draw[ultra thick,  directed]  (LeftBottom) to [out=15,in=270]  (Lefti);
\draw[ultra thick,  directed]  (Left) to [out=0,in=90]  (Lefti);
\draw[ultra thick,  directed]  (RightBottom) to [out=180,in=320]  (Leftj);
\draw[ultra thick,   directed]  (Righti) to [out=270,in=150] (Leftj)  ;
\draw[ultra thick,   directed]  (Righti) to (Leftk)  ;
\draw[ultra thick,   directed]  (Leftk) to [out=90,in=0] (LeftTop)  ;
\draw[ultra thick,  reverse directed]  (Right) to [out=225,in=315] (Rightj)  ;
\draw[ultra thick,   directed]  (Rightk) to [out=90,in=180] (RightTop)  ;
\draw[ultra thick,   directed]  (Rightj) to [out=135,in=270] (Rightk)  ;

\draw[thin, reverse directed]  (Lefti) --  (Righti) node[midway,right=.2,below]{\mbox{\textbf{\emph{~}}}};
\draw[thin, directed] (Rightj) -- (Leftj)  node[midway,right=.1]{\mbox{\textbf{\emph{~}}}};
\draw[thin] (Leftk) --  (Rightk) node[midway,above=-.1]{\mbox{\textbf{\emph{~}}}};
\draw [fill=black, ultra thick] (Leftk) circle [radius=0.15];
\draw [fill=black, ultra thick] (Rightk) circle [radius=0.15];
\draw [fill=white, ultra thick] (Lefti) circle [radius=0.15];
\draw [fill=white, ultra thick] (Righti) circle [radius=0.15];
\draw [fill=white, ultra thick] (Leftj) circle [radius=0.15];
\draw [fill=white, ultra thick] (Rightj) circle [radius=0.15];
\end{tikzpicture}
}

\newcommand{\RTBAAleft}[1]
{
\begin{tikzpicture} [scale=#1, baseline=0pt]
  \coordinate (Left) at (-4,0) ;
  \coordinate (Right) at (4,0) ;
 \coordinate (RightBottom) at (2.828,-2.828) ;
    \coordinate (LeftBottom) at (-2.828,-2.828) ;
 \coordinate (RightTop) at (2.828,2.828) ;
    \coordinate (LeftTop) at (-2.828,2.828) ;
         \coordinate (Righti) at (-1,-2) ;
    \coordinate (Lefti) at (-2,0) ;
          \coordinate (Rightj) at (2,0) ;
    \coordinate (Leftj) at (1,-2) ;
    \coordinate (Rightk) at (0,2) ;
    \coordinate (Leftk) at (0,.5) ;
    \coordinate (BCCenter) at (0,0) ;
\draw [densely dashed, ultra thick] (BCCenter) circle [radius=4];
\draw[ultra thick,  directed]  (LeftBottom) to [out=0,in=225]  (Righti);
\draw[ultra thick,   directed]  (Righti) to [out=45,in=135] (Leftj)  ;
\draw[ultra thick,  directed]  (RightBottom) to [out=180,in=315]  (Leftj);
\draw[ultra thick,  directed]  (Left) to [out=315,in=225]  (Lefti);
\draw[ultra thick,   directed]  (Lefti) to [out=45,in=180] (Leftk)  ;
\draw[ultra thick, reverse  directed]  (Leftk) to [out=0,in=135] (Rightj)  ;
\draw[ultra thick,  reverse directed]  (Right) to [out=225,in=315] (Rightj)  ;
\draw[ultra thick,   directed]  (Rightk) to [out=0,in=225] (RightTop)  ;
\draw[ultra thick,   directed]  (Rightk) to [out=180,in=315] (LeftTop)  ;

\draw[thin] (Lefti) --  (Righti) node[midway,right=.2,above]{\mbox{\textbf{\emph{~}}}};
\draw[thin, directed] (Rightj) -- (Leftj)  node[midway,right=.1]{\mbox{\textbf{\emph{~}}}};
\draw[thin, reverse directed]  (Leftk) --  (Rightk) node[midway,right=.1]{\mbox{\textbf{\emph{~}}}};
\draw [fill=white, ultra thick] (Leftk) circle [radius=0.15];
\draw [fill=white, ultra thick] (Rightk) circle [radius=0.15];
\draw [fill=black, ultra thick] (Lefti) circle [radius=0.15];
\draw [fill=black, ultra thick] (Righti) circle [radius=0.15];
\draw [fill=white, ultra thick] (Leftj) circle [radius=0.15];
\draw [fill=white, ultra thick] (Rightj) circle [radius=0.15];
\end{tikzpicture}
}

\newcommand{\RTBBBleft}[1]
{
\begin{tikzpicture} [scale=#1, baseline=0pt]
  \coordinate (Left) at (-4,0) ;
  \coordinate (Right) at (4,0) ;
 \coordinate (RightBottom) at (2.828,-2.828) ;
    \coordinate (LeftBottom) at (-2.828,-2.828) ;
 \coordinate (RightTop) at (2.828,2.828) ;
    \coordinate (LeftTop) at (-2.828,2.828) ;
         \coordinate (Righti) at (-1,-2) ;
    \coordinate (Lefti) at (-2,0) ;
          \coordinate (Rightj)  at (2,-2) ;
    \coordinate (Leftj) at (1,0) ;
       \coordinate (Rightk) at (1,1.5) ;
    \coordinate (Leftk) at (-1,1.5) ;
    \coordinate (BCCenter) at (0,0) ;
\draw [densely dashed, ultra thick] (BCCenter) circle [radius=4];
\draw[ultra thick,  directed]  (LeftBottom) to [out=0,in=225]  (Righti);
\draw[ultra thick,   directed]  (Leftj) to  (Rightk)  ;
\draw[ultra thick,   directed]  (Righti) to [out=45,in=270]   (Leftj)  ;
\draw[ultra thick,  directed]  (RightBottom) to [out=180,in=270]  (Rightj);
\draw[ultra thick,  directed]  (Left) to [out=315,in=225]  (Lefti);
\draw[ultra thick,   directed]  (Lefti) to [out=45,in=270] (Leftk)  ;
\draw[ultra thick,  reverse directed]  (Right) to [out=180,in=90] (Rightj)  ;
\draw[ultra thick,   directed]  (Leftk) to [out=90,in=0] (LeftTop)  ;
\draw[ultra thick,   directed]  (Rightk) to [out=90,in=180] (RightTop)  ;

\draw[thin] (Lefti) --  (Righti) node[midway,right=.2,above]{\mbox{\textbf{\emph{~}}}};
\draw[thin] (Rightj) -- (Leftj)  node[midway,right=.1,above]{\mbox{\textbf{\emph{~}}}};
\draw[thin] (Leftk) --  (Rightk) node[midway,above=.1]{\mbox{\textbf{\emph{~}}}};
\draw [fill=black, ultra thick] (Leftk) circle [radius=0.15];
\draw [fill=black, ultra thick] (Rightk) circle [radius=0.15];
\draw [fill=black, ultra thick] (Lefti) circle [radius=0.15];
\draw [fill=black, ultra thick] (Righti) circle [radius=0.15];
\draw [fill=black, ultra thick] (Leftj) circle [radius=0.15];
\draw [fill=black, ultra thick] (Rightj) circle [radius=0.15];
\end{tikzpicture}
}

\newcommand{\RTBABleft}[1]
{
\begin{tikzpicture} [scale=#1, baseline=0pt]
  \coordinate (Left) at (-4,0) ;
  \coordinate (Right) at (4,0) ;
 \coordinate (RightBottom) at (2.828,-2.828) ;
    \coordinate (LeftBottom) at (-2.828,-2.828) ;
 \coordinate (RightTop) at (2.828,2.828) ;
    \coordinate (LeftTop) at (-2.828,2.828) ;
         \coordinate (Righti) at (-1,-2) ;
    \coordinate (Lefti) at (-2,0) ;
          \coordinate (Rightj) at (2,0) ;
    \coordinate (Leftj) at (1,-2) ;
       \coordinate (Rightk) at (1,1.5) ;
    \coordinate (Leftk) at (-1,1.5) ;
    \coordinate (BCCenter) at (0,0) ;
\draw [densely dashed, ultra thick] (BCCenter) circle [radius=4];
\draw[ultra thick,  directed]  (LeftBottom) to [out=0,in=225]  (Righti);
\draw[ultra thick,   directed]  (Righti) to [out=45,in=135] (Leftj)  ;
\draw[ultra thick,  directed]  (RightBottom) to [out=180,in=315]  (Leftj);
\draw[ultra thick,  directed]  (Left) to [out=315,in=225]  (Lefti);
\draw[ultra thick,   directed]  (Lefti) to [out=45,in=270] (Leftk)  ;
\draw[ultra thick,  reverse directed]  (Right) to [out=225,in=315] (Rightj)  ;
\draw[ultra thick,   directed]  (Leftk) to [out=90,in=0] (LeftTop)  ;
\draw[ultra thick,   directed]  (Rightk) to [out=90,in=180] (RightTop)  ;
\draw[ultra thick,   directed]  (Rightj) to [out=135,in=270] (Rightk)  ;

\draw[thin] (Lefti) --  (Righti) node[midway,right=.2,above]{\mbox{\textbf{\emph{~}}}};
\draw[thin, directed] (Rightj) -- (Leftj)  node[midway,right=.1]{\mbox{\textbf{\emph{~}}}};
\draw[thin] (Leftk) --  (Rightk) node[midway,above=.1]{\mbox{\textbf{\emph{~}}}};
\draw [fill=black, ultra thick] (Leftk) circle [radius=0.15];
\draw [fill=black, ultra thick] (Rightk) circle [radius=0.15];
\draw [fill=black, ultra thick] (Lefti) circle [radius=0.15];
\draw [fill=black, ultra thick] (Righti) circle [radius=0.15];
\draw [fill=white, ultra thick] (Leftj) circle [radius=0.15];
\draw [fill=white, ultra thick] (Rightj) circle [radius=0.15];
\end{tikzpicture}
}

\newcommand{\RTABAleft}[1]
{
\begin{tikzpicture} [scale=#1, baseline=0pt]
  \coordinate (Left) at (-4,0) ;
  \coordinate (Right) at (4,0) ;
 \coordinate (RightBottom) at (2.828,-2.828) ;
    \coordinate (LeftBottom) at (-2.828,-2.828) ;
 \coordinate (RightTop) at (2.828,2.828) ;
    \coordinate (LeftTop) at (-2.828,2.828) ;
       \coordinate (Righti) at (-1,0) ;
    \coordinate (Lefti) at (-2,-2) ;
            \coordinate (Rightj)  at (2,-2) ;
    \coordinate (Leftj) at (1,0) ;
    \coordinate (Rightk) at (0,2) ;
    \coordinate (Leftk) at (0,.5) ;
    \coordinate (BCCenter) at (0,0) ;
\draw [densely dashed, ultra thick] (BCCenter) circle [radius=4];
\draw[ultra thick,  directed]  (LeftBottom) to [out=15,in=270]  (Lefti);
\draw[ultra thick,  directed]  (Left) to [out=0,in=90]  (Lefti);
\draw[ultra thick,   directed]  (Righti) to [out=270,in=270] (Leftj)  ;
\draw[ultra thick,  reverse directed]  (Leftk) to [out=180,in=90] (Righti)   ;
\draw[ultra thick, reverse   directed]  (Leftk) to [out=0,in=90] (Leftj)  ;
\draw[ultra thick,  directed]  (RightBottom) to [out=180,in=270]  (Rightj);
\draw[ultra thick,  reverse directed]  (Right) to [out=180,in=90] (Rightj)  ;
\draw[ultra thick,   directed]  (Rightk) to [out=0,in=225] (RightTop)  ;
\draw[ultra thick,   directed]  (Rightk) to [out=180,in=315] (LeftTop)  ;

\draw[thin, reverse directed]  (Lefti) --  (Righti) node[midway,right=.2,below]{\mbox{\textbf{\emph{~}}}};
\draw[thin] (Rightj) -- (Leftj)  node[midway,right=.1,above]{\mbox{\textbf{\emph{~}}}};
\draw[thin, reverse directed]  (Leftk) --  (Rightk) node[midway,right=.1]{\mbox{\textbf{\emph{~}}}};
\draw [fill=white, ultra thick] (Leftk) circle [radius=0.15];
\draw [fill=white, ultra thick] (Rightk) circle [radius=0.15];
\draw [fill=white, ultra thick] (Lefti) circle [radius=0.15];
\draw [fill=white, ultra thick] (Righti) circle [radius=0.15];
\draw [fill=black, ultra thick] (Leftj) circle [radius=0.15];
\draw [fill=black, ultra thick] (Rightj) circle [radius=0.15];
\end{tikzpicture}
}

\newcommand{\RTBBAleft}[1]
{
\begin{tikzpicture} [scale=#1, baseline=0pt]
  \coordinate (Left) at (-4,0) ;
  \coordinate (Right) at (4,0) ;
 \coordinate (RightBottom) at (2.828,-2.828) ;
    \coordinate (LeftBottom) at (-2.828,-2.828) ;
 \coordinate (RightTop) at (2.828,2.828) ;
    \coordinate (LeftTop) at (-2.828,2.828) ;
         \coordinate (Righti) at (-1,-2) ;
    \coordinate (Lefti) at (-2,0) ;
          \coordinate (Rightj)  at (2,-2) ;
    \coordinate (Leftj) at (1,0) ;
      \coordinate (Rightk) at (0,2) ;
    \coordinate (Leftk) at (0,.5) ;
    \coordinate (BCCenter) at (0,0) ;
\draw [densely dashed, ultra thick] (BCCenter) circle [radius=4];
\draw[ultra thick,  directed]  (LeftBottom) to [out=0,in=225]  (Righti);
\draw[ultra thick,   directed]  (Righti) to [out=45,in=270]  (Leftj)  ;
\draw[ultra thick,  directed]  (RightBottom) to [out=180,in=270]  (Rightj);
\draw[ultra thick,  directed]  (Left) to [out=315,in=225]  (Lefti);
\draw[ultra thick,   directed]  (Lefti) to [out=45,in=180] (Leftk)  ;
\draw[ultra thick, reverse   directed]  (Leftk) to [out=0,in=90] (Leftj)  ;
\draw[ultra thick,  reverse directed]  (Right) to [out=180,in=90] (Rightj)  ;
\draw[ultra thick,   directed]  (Rightk) to [out=0,in=225] (RightTop)  ;
\draw[ultra thick,   directed]  (Rightk) to [out=180,in=315] (LeftTop)  ;

\draw[thin] (Lefti) --  (Righti) node[midway,right=.2,above]{\mbox{\textbf{\emph{~}}}};
\draw[thin] (Rightj) -- (Leftj)  node[midway,right=.1,above]{\mbox{\textbf{\emph{~}}}};
\draw[thin, reverse directed]  (Leftk) --  (Rightk) node[midway,right=.1]{\mbox{\textbf{\emph{~}}}};
\draw [fill=white, ultra thick] (Leftk) circle [radius=0.15];
\draw [fill=white, ultra thick] (Rightk) circle [radius=0.15];
\draw [fill=black, ultra thick] (Lefti) circle [radius=0.15];
\draw [fill=black, ultra thick] (Righti) circle [radius=0.15];
\draw [fill=black, ultra thick] (Leftj) circle [radius=0.15];
\draw [fill=black, ultra thick] (Rightj) circle [radius=0.15];
\end{tikzpicture}
}

\newcommand{\RTABBleft}[1]
{
\begin{tikzpicture} [scale=#1, baseline=0pt]
  \coordinate (Left) at (-4,0) ;
  \coordinate (Right) at (4,0) ;
 \coordinate (RightBottom) at (2.828,-2.828) ;
    \coordinate (LeftBottom) at (-2.828,-2.828) ;
 \coordinate (RightTop) at (2.828,2.828) ;
    \coordinate (LeftTop) at (-2.828,2.828) ;
        \coordinate (Righti) at (-1,0) ;
    \coordinate (Lefti) at (-2,-2) ;
          \coordinate (Rightj)  at (2,-2) ;
    \coordinate (Leftj) at (1,0) ;
       \coordinate (Rightk) at (1,1.5) ;
    \coordinate (Leftk) at (-1,1.5) ;
    \coordinate (BCCenter) at (0,0) ;
\draw [densely dashed, ultra thick] (BCCenter) circle [radius=4];
\draw[ultra thick,  directed]  (LeftBottom) to [out=15,in=270]  (Lefti);
\draw[ultra thick,  directed]  (Left) to [out=0,in=90]  (Lefti);
\draw[ultra thick,   directed]  (Leftj) to  (Rightk)  ;
\draw[ultra thick,   directed]  (Righti) to [out=270,in=270] (Leftj)  ;
\draw[ultra thick,  directed]  (RightBottom) to [out=180,in=270]  (Rightj);
\draw[ultra thick,   directed]  (Righti) to  (Leftk)  ;
\draw[ultra thick,  reverse directed]  (Right) to [out=180,in=90] (Rightj)  ;
\draw[ultra thick,   directed]  (Leftk) to [out=90,in=0] (LeftTop)  ;
\draw[ultra thick,   directed]  (Rightk) to [out=90,in=180] (RightTop)  ;

\draw[thin, reverse directed]  (Lefti) --  (Righti) node[midway,right=.2,below]{\mbox{\textbf{\emph{~}}}};
\draw[thin] (Rightj) -- (Leftj)  node[midway,right=.1,above]{\mbox{\textbf{\emph{~}}}};
\draw[thin]  (Leftk) --  (Rightk) node[midway,above=.1]{\mbox{\textbf{\emph{~}}}};
\draw [fill=black, ultra thick] (Leftk) circle [radius=0.15];
\draw [fill=black, ultra thick] (Rightk) circle [radius=0.15];
\draw [fill=white, ultra thick] (Lefti) circle [radius=0.15];
\draw [fill=white, ultra thick] (Righti) circle [radius=0.15];
\draw [fill=black, ultra thick] (Leftj) circle [radius=0.15];
\draw [fill=black, ultra thick] (Rightj) circle [radius=0.15];
\end{tikzpicture}
}

\newcommand{\RTAAAleft}[1]
{
\begin{tikzpicture} [scale=#1, baseline=0pt]
  \coordinate (Left) at (-4,0) ;
  \coordinate (Right) at (4,0) ;
 \coordinate (RightBottom) at (2.828,-2.828) ;
    \coordinate (LeftBottom) at (-2.828,-2.828) ;
 \coordinate (RightTop) at (2.828,2.828) ;
    \coordinate (LeftTop) at (-2.828,2.828) ;
           \coordinate (Righti) at (-1,0) ;
    \coordinate (Lefti) at (-2,-2) ;
          \coordinate (Rightj) at (2,0) ;
    \coordinate (Leftj) at (1,-2) ;
\coordinate (Rightk) at (0,2) ;
    \coordinate (Leftk) at (0,.5) ;
    \coordinate (BCCenter) at (0,0) ;
\draw [densely dashed, ultra thick] (BCCenter) circle [radius=4];
\draw[ultra thick,  directed]  (LeftBottom) to [out=15,in=270]  (Lefti);
\draw[ultra thick,  directed]  (Left) to [out=0,in=90]  (Lefti);
\draw[ultra thick,  directed]  (RightBottom) to [out=180,in=320]  (Leftj);
\draw[ultra thick,   directed]  (Righti) to [out=270,in=150] (Leftj)  ;
\draw[ultra thick,  reverse directed]  (Leftk) to [out=180,in=90] (Righti)   ;
\draw[ultra thick,  reverse directed]  (Right) to [out=225,in=315] (Rightj)  ;
\draw[ultra thick,   directed]  (Rightk) to [out=0,in=225] (RightTop)  ;
\draw[ultra thick,   directed]  (Rightk) to [out=180,in=315] (LeftTop)  ;
\draw[ultra thick,   directed]  (Rightj) to [out=135,in=0] (Leftk)  ;

\draw[thin, reverse directed]  (Lefti) --  (Righti) node[midway,right=.2,below]{\mbox{\textbf{\emph{~}}}};
\draw[thin, directed] (Rightj) -- (Leftj)  node[midway,right=.1]{\mbox{\textbf{\emph{~}}}};
\draw[thin, reverse directed]  (Leftk) --  (Rightk) node[midway,right=.1]{\mbox{\textbf{\emph{~}}}};
\draw [fill=white, ultra thick] (Leftk) circle [radius=0.15];
\draw [fill=white, ultra thick] (Rightk) circle [radius=0.15];
\draw [fill=white, ultra thick] (Lefti) circle [radius=0.15];
\draw [fill=white, ultra thick] (Righti) circle [radius=0.15];
\draw [fill=white, ultra thick] (Leftj) circle [radius=0.15];
\draw [fill=white, ultra thick] (Rightj) circle [radius=0.15];
\end{tikzpicture}
}

\newcommand{\RTAABright}[1]
{
\begin{tikzpicture} [scale=#1, baseline=0pt]
  \coordinate (Left) at (-4,0) ;
  \coordinate (Right) at (4,0) ;
 \coordinate (RightBottom) at (2.828,-2.828) ;
    \coordinate (LeftBottom) at (-2.828,-2.828) ;
 \coordinate (RightTop) at (2.828,2.828) ;
    \coordinate (LeftTop) at (-2.828,2.828) ;
           \coordinate (Righti) at (-1,2) ;
    \coordinate (Lefti) at (-2,0) ;
          \coordinate (Rightj) at (2,2) ;
    \coordinate (Leftj) at (1,0) ;
    \coordinate (Rightk) at (1,-1.5) ;
    \coordinate (Leftk) at (-1,-1.5) ;
    \coordinate (BCCenter) at (0,0) ;
\draw [densely dashed, ultra thick] (BCCenter) circle [radius=4];
\draw[ultra thick,  directed]  (LeftBottom) to [out=0,in=270]  (Leftk);
\draw[ultra thick,  directed]  (Left) to [out=0,in=180]  (Lefti);
\draw[ultra thick,  directed]  (RightBottom) to [out=180,in=270]  (Rightk);
\draw[ultra thick,   directed]  (Righti) to [out=335,in=90] (Leftj)  ;
\draw[ultra thick, reverse  directed]  (Lefti) to [out=335,in=60] (Leftk)  ;
\draw[ultra thick,   directed]  (Righti) to [out=180,in=290] (LeftTop)  ;
\draw[ultra thick,  reverse directed]  (Right) to [out=180,in=300] (Rightj)  ;
\draw[ultra thick,  reverse directed]  (RightTop) to [out=180,in=110] (Rightj)  ;
\draw[ultra thick, reverse  directed]  (Leftj) to [out=270,in=270] (Rightk)  ;

\draw[thin, reverse directed]  (Lefti) --  (Righti) node[midway,right=.2,below]{\mbox{\textbf{\emph{~}}}};
\draw[thin, directed] (Rightj) -- (Leftj)  node[midway,right=.1]{\mbox{\textbf{\emph{~}}}};
\draw[thin]  (Leftk) --  (Rightk) node[midway,above=-.1]{\mbox{\textbf{\emph{~}}}};
\draw [fill=black, ultra thick] (Leftk) circle [radius=0.15];
\draw [fill=black, ultra thick] (Rightk) circle [radius=0.15];
\draw [fill=white, ultra thick] (Lefti) circle [radius=0.15];
\draw [fill=white, ultra thick] (Righti) circle [radius=0.15];
\draw [fill=white, ultra thick] (Leftj) circle [radius=0.15];
\draw [fill=white, ultra thick] (Rightj) circle [radius=0.15];
\end{tikzpicture}
}

\newcommand{\RTABAright}[1]
{
\begin{tikzpicture} [scale=#1, baseline=0pt]
  \coordinate (Left) at (-4,0) ;
  \coordinate (Right) at (4,0) ;
 \coordinate (RightBottom) at (2.828,-2.828) ;
    \coordinate (LeftBottom) at (-2.828,-2.828) ;
 \coordinate (RightTop) at (2.828,2.828) ;
    \coordinate (LeftTop) at (-2.828,2.828) ;
           \coordinate (Righti) at (-1,2) ;
    \coordinate (Lefti) at (-2,0) ;
          \coordinate (Rightj)  at (2,0) ;
    \coordinate (Leftj) at (1,2) ;
\coordinate (Leftk)  at (0,-2) ;
    \coordinate (Rightk) at (0,-.5) ;
    \coordinate (BCCenter) at (0,0) ;
\draw [densely dashed, ultra thick] (BCCenter) circle [radius=4];
\draw[ultra thick,  directed]  (LeftBottom) to [out=45,in=180]  (Leftk);
\draw[ultra thick,  directed]  (Left) to [out=0,in=180]  (Lefti);
\draw[ultra thick,  directed]  (RightBottom) to [out=135,in=0]  (Leftk);
\draw[ultra thick,  reverse directed] (Leftj)  to [out=180,in=0] (Righti)   ;
\draw[ultra thick,   directed] (Rightk)  to [out=180,in=0] (Lefti)   ;
\draw[ultra thick,   directed]  (Righti) to [out=180,in=290] (LeftTop)  ;
\draw[ultra thick,  reverse directed]  (Right) to [out=180,in=0] (Rightj)  ;
\draw[ultra thick,   directed]  (Rightk) to [out=20,in=180] (Rightj)  ;
\draw[ultra thick,  reverse directed]  (RightTop) to [out=270,in=0] (Leftj)  ;

\draw[thin, reverse directed]  (Lefti) --  (Righti) node[midway,right=.2,below]{\mbox{\textbf{\emph{~}}}};
\draw[thin] (Rightj) -- (Leftj)  node[midway,right=.1,above]{\mbox{\textbf{\emph{~}}}};
\draw[thin, reverse directed]  (Leftk) --  (Rightk) node[midway,right=.1]{\mbox{\textbf{\emph{~}}}};
\draw [fill=white, ultra thick] (Leftk) circle [radius=0.15];
\draw [fill=white, ultra thick] (Rightk) circle [radius=0.15];
\draw [fill=white, ultra thick] (Lefti) circle [radius=0.15];
\draw [fill=white, ultra thick] (Righti) circle [radius=0.15];
\draw [fill=black, ultra thick] (Leftj) circle [radius=0.15];
\draw [fill=black, ultra thick] (Rightj) circle [radius=0.15];
\end{tikzpicture}
}

\newcommand{\RTBBBright}[1]
{
\begin{tikzpicture} [scale=#1, baseline=0pt]
  \coordinate (Left) at (-4,0) ;
  \coordinate (Right) at (4,0) ;
 \coordinate (RightBottom) at (2.828,-2.828) ;
    \coordinate (LeftBottom) at (-2.828,-2.828) ;
 \coordinate (RightTop) at (2.828,2.828) ;
    \coordinate (LeftTop) at (-2.828,2.828) ;
    \coordinate (Lefti) at (-2,2) ;
      \coordinate (Righti) at (-1,0) ;
          \coordinate (Rightj)  at (2,0) ;
    \coordinate (Leftj) at (1,2) ;
 \coordinate (Rightk) at (1,-1.5) ;
    \coordinate (Leftk) at (-1,-1.5) ;
    \coordinate (BCCenter) at (0,0) ;
\draw [densely dashed, ultra thick] (BCCenter) circle [radius=4];

\draw[ultra thick,  directed]  (Left) to [out=0,in=270]  (Lefti);
\draw[ultra thick,  reverse directed] (Leftj)  to [out=180,in=90] (Righti)   ;
\draw[ultra thick,  directed]  (LeftBottom) to [out=0,in=270]  (Leftk);
        \draw[ultra thick,  directed]  (RightBottom) to [out=180,in=270]  (Rightk);
\draw[ultra thick,   directed] (Leftk)  to (Righti)   ;
\draw[ultra thick,  reverse directed] (LeftTop)  to [out=0,in=90] (Lefti)  ;
\draw[ultra thick,  reverse directed]  (Right) to [out=180,in=0] (Rightj)  ;
\draw[ultra thick,   directed]  (Rightk) to [out=100,in=180] (Rightj)  ;
\draw[ultra thick,  reverse directed]  (RightTop) to [out=270,in=0] (Leftj)  ;

\draw[thin]  (Lefti) --  (Righti) node[midway,right=.2,above]{\mbox{\textbf{\emph{~}}}};
\draw[thin] (Rightj) -- (Leftj)  node[midway,right=.1,above]{\mbox{\textbf{\emph{~}}}};
 \draw[thin]  (Leftk) --  (Rightk) node[midway,above=-.1]{\mbox{\textbf{\emph{~}}}};
\draw [fill=black, ultra thick] (Leftk) circle [radius=0.15];
\draw [fill=black, ultra thick] (Rightk) circle [radius=0.15];
\draw [fill=black, ultra thick] (Lefti) circle [radius=0.15];
\draw [fill=black, ultra thick] (Righti) circle [radius=0.15];
\draw [fill=black, ultra thick] (Leftj) circle [radius=0.15];
\draw [fill=black, ultra thick] (Rightj) circle [radius=0.15];
\end{tikzpicture}
}

\newcommand{\RTABBright}[1]
{
\begin{tikzpicture} [scale=#1, baseline=0pt]
  \coordinate (Left) at (-4,0) ;
  \coordinate (Right) at (4,0) ;
 \coordinate (RightBottom) at (2.828,-2.828) ;
    \coordinate (LeftBottom) at (-2.828,-2.828) ;
 \coordinate (RightTop) at (2.828,2.828) ;
    \coordinate (LeftTop) at (-2.828,2.828) ;
    \coordinate (Righti) at (-1,2) ;
    \coordinate (Lefti) at (-2,0) ;
          \coordinate (Rightj)  at (2,0) ;
    \coordinate (Leftj) at (1,2) ;
 \coordinate (Rightk) at (1,-1.5) ;
    \coordinate (Leftk) at (-1,-1.5) ;
    \coordinate (BCCenter) at (0,0) ;
\draw [densely dashed, ultra thick] (BCCenter) circle [radius=4];

\draw[ultra thick,  directed]  (LeftBottom) to [out=0,in=270]  (Leftk);
        \draw[ultra thick,  directed]  (RightBottom) to [out=180,in=270]  (Rightk);
\draw[ultra thick,   directed] (Leftk)  to [out=45,in=0] (Lefti)   ;
    \draw[ultra thick,  directed]  (Left) to [out=0,in=180]  (Lefti);
    \draw[ultra thick,   directed]  (Righti) to [out=180,in=290] (LeftTop)  ;
\draw[ultra thick,  reverse directed]  (Right) to [out=180,in=0] (Rightj)  ;
\draw[ultra thick,   directed]  (Rightk) to [out=100,in=180] (Rightj)  ;
\draw[ultra thick,  reverse directed]  (RightTop) to [out=270,in=0] (Leftj)  ;
\draw[ultra thick,  reverse directed] (Leftj)  to [out=180,in=0] (Righti)   ;

 \draw[thin, reverse directed]  (Lefti) --  (Righti) node[midway,right=.2,below]{\mbox{\textbf{\emph{~}}}};
\draw[thin] (Rightj) -- (Leftj)  node[midway,right=.1,above]{\mbox{\textbf{\emph{~}}}};
 \draw[thin]  (Leftk) --  (Rightk) node[midway,above=-.1]{\mbox{\textbf{\emph{~}}}};
\draw [fill=black, ultra thick] (Leftk) circle [radius=0.15];
\draw [fill=black, ultra thick] (Rightk) circle [radius=0.15];
\draw [fill=white, ultra thick] (Lefti) circle [radius=0.15];
\draw [fill=white, ultra thick] (Righti) circle [radius=0.15];
\draw [fill=black, ultra thick] (Leftj) circle [radius=0.15];
\draw [fill=black, ultra thick] (Rightj) circle [radius=0.15];
\end{tikzpicture}
}

\newcommand{\RTBAAright}[1]
{
\begin{tikzpicture} [scale=#1, baseline=0pt]
  \coordinate (Left) at (-4,0) ;
  \coordinate (Right) at (4,0) ;
 \coordinate (RightBottom) at (2.828,-2.828) ;
    \coordinate (LeftBottom) at (-2.828,-2.828) ;
 \coordinate (RightTop) at (2.828,2.828) ;
    \coordinate (LeftTop) at (-2.828,2.828) ;
          \coordinate (Lefti) at (-2,2) ;
      \coordinate (Righti) at (-1,0) ;
          \coordinate (Rightj) at (2,2) ;
    \coordinate (Leftj) at (1,0) ;
\coordinate (Leftk)  at (0,-2) ;
    \coordinate (Rightk) at (0,-.5) ;
    \coordinate (BCCenter) at (0,0) ;
\draw [densely dashed, ultra thick] (BCCenter) circle [radius=4];
 \draw[ultra thick,  directed]  (LeftBottom) to [out=45,in=180]  (Leftk);
      \draw[ultra thick,  directed]  (RightBottom) to [out=135,in=0]  (Leftk);
\draw[ultra thick,   directed]  (Righti) to [out=90,in=90] (Leftj)  ;
      \draw[ultra thick,  directed]  (Left) to [out=0,in=270]  (Lefti);
      \draw[ultra thick,  reverse directed] (LeftTop)  to [out=0,in=90] (Lefti)  ;
\draw[ultra thick,  reverse directed]  (Right) to [out=180,in=300] (Rightj)  ;
\draw[ultra thick,   directed]  (Rightk) to [out=180,in=270] (Righti)  ;
\draw[ultra thick,  reverse directed]  (RightTop) to [out=180,in=110] (Rightj)  ;
\draw[ultra thick, reverse  directed]  (Leftj) to [out=270,in=0] (Rightk)  ;

\draw[thin]  (Lefti) --  (Righti) node[midway,right=.2,above]{\mbox{\textbf{\emph{~}}}};
\draw[thin, directed] (Rightj) -- (Leftj)  node[midway,right=.1]{\mbox{\textbf{\emph{~}}}};
\draw[thin, reverse directed]  (Leftk) --  (Rightk) node[midway,right=.1]{\mbox{\textbf{\emph{~}}}};
\draw [fill=white, ultra thick] (Leftk) circle [radius=0.15];
\draw [fill=white, ultra thick] (Rightk) circle [radius=0.15];
\draw [fill=black, ultra thick] (Lefti) circle [radius=0.15];
\draw [fill=black, ultra thick] (Righti) circle [radius=0.15];
\draw [fill=white, ultra thick] (Leftj) circle [radius=0.15];
\draw [fill=white, ultra thick] (Rightj) circle [radius=0.15];
\end{tikzpicture}
}

\newcommand{\RTAAAright}[1]
{
\begin{tikzpicture} [scale=#1, baseline=0pt]
  \coordinate (Left) at (-4,0) ;
  \coordinate (Right) at (4,0) ;
 \coordinate (RightBottom) at (2.828,-2.828) ;
    \coordinate (LeftBottom) at (-2.828,-2.828) ;
 \coordinate (RightTop) at (2.828,2.828) ;
    \coordinate (LeftTop) at (-2.828,2.828) ;
         \coordinate (Righti) at (-1,2) ;
    \coordinate (Lefti) at (-2,0) ;
          \coordinate (Rightj) at (2,2) ;
    \coordinate (Leftj) at (1,0) ;
\coordinate (Leftk)  at (0,-2) ;
    \coordinate (Rightk) at (0,-.5) ;
    \coordinate (BCCenter) at (0,0) ;
\draw [densely dashed, ultra thick] (BCCenter) circle [radius=4];
 \draw[ultra thick,  directed]  (LeftBottom) to [out=45,in=180]  (Leftk);
      \draw[ultra thick,  directed]  (RightBottom) to [out=135,in=0]  (Leftk);
\draw[ultra thick,   directed]  (Righti) to [out=0,in=90] (Leftj)  ;
\draw[ultra thick,  directed]  (Left) to [out=0,in=180]  (Lefti);
 \draw[ultra thick,   directed]  (Righti) to [out=180,in=290] (LeftTop)  ;
\draw[ultra thick,  reverse directed]  (Right) to [out=180,in=300] (Rightj)  ;
\draw[ultra thick,   directed]  (Rightk) to [out=180,in=0] (Lefti)  ;
\draw[ultra thick,  reverse directed]  (RightTop) to [out=180,in=110] (Rightj)  ;
\draw[ultra thick, reverse  directed]  (Leftj) to [out=270,in=0] (Rightk)  ;

 \draw[thin, reverse directed]  (Lefti) --  (Righti) node[midway,right=.2,below]{\mbox{\textbf{\emph{~}}}};
\draw[thin, directed] (Rightj) -- (Leftj)  node[midway,right=.1]{\mbox{\textbf{\emph{~}}}};
\draw[thin, reverse directed]  (Leftk) --  (Rightk) node[midway,right=.1]{\mbox{\textbf{\emph{~}}}};
\draw [fill=white, ultra thick] (Leftk) circle [radius=0.15];
\draw [fill=white, ultra thick] (Rightk) circle [radius=0.15];
  \draw [fill=white, ultra thick] (Lefti) circle [radius=0.15];
\draw [fill=white, ultra thick] (Righti) circle [radius=0.15];
\draw [fill=white, ultra thick] (Leftj) circle [radius=0.15];
\draw [fill=white, ultra thick] (Rightj) circle [radius=0.15];
\end{tikzpicture}
}

\newcommand{\RTBABright}[1]
{
\begin{tikzpicture} [scale=#1, baseline=0pt]
  \coordinate (Left) at (-4,0) ;
  \coordinate (Right) at (4,0) ;
 \coordinate (RightBottom) at (2.828,-2.828) ;
    \coordinate (LeftBottom) at (-2.828,-2.828) ;
 \coordinate (RightTop) at (2.828,2.828) ;
    \coordinate (LeftTop) at (-2.828,2.828) ;
    \coordinate (Lefti) at (-2,2) ;
      \coordinate (Righti) at (-1,0) ;
         \coordinate (Rightj) at (2,2) ;
    \coordinate (Leftj) at (1,0) ;
 \coordinate (Rightk) at (1,-1.5) ;
    \coordinate (Leftk) at (-1,-1.5) ;
    \coordinate (BCCenter) at (0,0) ;
\draw [densely dashed, ultra thick] (BCCenter) circle [radius=4];

\draw[ultra thick,  directed]  (Left) to [out=0,in=270]  (Lefti);
\draw[ultra thick,  reverse directed] (Leftj)  to [out=90,in=90] (Righti)   ;
\draw[ultra thick,  directed]  (LeftBottom) to [out=0,in=270]  (Leftk);
        \draw[ultra thick,  directed]  (RightBottom) to [out=180,in=270]  (Rightk);
\draw[ultra thick,   directed] (Leftk)  to (Righti)   ;
\draw[ultra thick,  reverse directed] (LeftTop)  to [out=0,in=90] (Lefti)  ;
\draw[ultra thick,   directed]  (Rightk) to (Leftj)  ;
\draw[ultra thick,  reverse directed]  (RightTop) to [out=180,in=110] (Rightj)  ;
    \draw[ultra thick,  reverse directed]  (Right) to [out=180,in=300] (Rightj)  ;

\draw[thin]  (Lefti) --  (Righti) node[midway,right=.2,above]{\mbox{\textbf{\emph{~}}}};
   \draw[thin, directed] (Rightj) -- (Leftj)  node[midway,right=.1]{\mbox{\textbf{\emph{~}}}};
 \draw[thin]  (Leftk) --  (Rightk) node[midway,above=-.1]{\mbox{\textbf{\emph{~}}}};
\draw [fill=black, ultra thick] (Leftk) circle [radius=0.15];
\draw [fill=black, ultra thick] (Rightk) circle [radius=0.15];
\draw [fill=black, ultra thick] (Lefti) circle [radius=0.15];
\draw [fill=black, ultra thick] (Righti) circle [radius=0.15];
   \draw [fill=white, ultra thick] (Leftj) circle [radius=0.15];
\draw [fill=white, ultra thick] (Rightj) circle [radius=0.15];
\end{tikzpicture}
}

\newcommand{\RTBBAright}[1]
{
\begin{tikzpicture} [scale=#1, baseline=0pt]
  \coordinate (Left) at (-4,0) ;
  \coordinate (Right) at (4,0) ;
 \coordinate (RightBottom) at (2.828,-2.828) ;
    \coordinate (LeftBottom) at (-2.828,-2.828) ;
 \coordinate (RightTop) at (2.828,2.828) ;
    \coordinate (LeftTop) at (-2.828,2.828) ;
    \coordinate (Lefti) at (-2,2) ;
      \coordinate (Righti) at (-1,0) ;
          \coordinate (Rightj)  at (2,0) ;
    \coordinate (Leftj) at (1,2) ;
\coordinate (Leftk)  at (0,-2) ;
    \coordinate (Rightk) at (0,-.5) ;
    \coordinate (BCCenter) at (0,0) ;
\draw [densely dashed, ultra thick] (BCCenter) circle [radius=4];

\draw[ultra thick,  directed]  (Left) to [out=0,in=270]  (Lefti);
\draw[ultra thick,  reverse directed] (Leftj)  to [out=180,in=90] (Righti)   ;
 \draw[ultra thick,  directed]  (LeftBottom) to [out=45,in=180]  (Leftk);
      \draw[ultra thick,  directed]  (RightBottom) to [out=135,in=0]  (Leftk);
\draw[ultra thick,   directed] (Rightk)  to [out=180,in=270]  (Righti)   ;
\draw[ultra thick,  reverse directed] (LeftTop)  to [out=0,in=90] (Lefti)  ;
\draw[ultra thick,  reverse directed]  (Right) to [out=180,in=0] (Rightj)  ;
\draw[ultra thick,   directed]  (Rightk) to [out=0,in=180] (Rightj)  ;
\draw[ultra thick,  reverse directed]  (RightTop) to [out=270,in=0] (Leftj)  ;

\draw[thin]  (Lefti) --  (Righti) node[midway,right=.2,above]{\mbox{\textbf{\emph{~}}}};
\draw[thin] (Rightj) -- (Leftj)  node[midway,right=.1,above]{\mbox{\textbf{\emph{~}}}};
\draw[thin, reverse directed]  (Leftk) --  (Rightk) node[midway,right=.1]{\mbox{\textbf{\emph{~}}}};
\draw [fill=white, ultra thick] (Leftk) circle [radius=0.15];
\draw [fill=white, ultra thick] (Rightk) circle [radius=0.15];
\draw [fill=black, ultra thick] (Lefti) circle [radius=0.15];
\draw [fill=black, ultra thick] (Righti) circle [radius=0.15];
\draw [fill=black, ultra thick] (Leftj) circle [radius=0.15];
\draw [fill=black, ultra thick] (Rightj) circle [radius=0.15];
\end{tikzpicture}
}

\newcommand{\RScrpos}[1]
{
\begin{tikzpicture} [scale=#1, baseline=0pt]
  \coordinate (RightT) at (1.414,1.414) ;
    \coordinate (LeftT) at (-1.414,1.414) ;
 \coordinate (RightBottom) at (1.414,-1.414) ;
    \coordinate (LeftBottom) at (-1.414,-1.414) ;
        \coordinate (BCCenter) at (0,0) ;
        \coordinate (LTCenter) at (-.25,.25) ;
         \coordinate (RBCenter) at (.25,-.25) ;
         \coordinate (i) at (.75,0) ;
\coordinate (LA) at  (0,-1.4) ;
\coordinate (RA) at (0,1.4) ;
\coordinate (RB) at  (1.4,0) ;
\coordinate (LB) at (-1.4,0) ;
\draw [densely dashed, ultra thick] (BCCenter) circle [radius=2];
\draw[ultra thick, final directed]  (LeftBottom) to  (RightT);

\draw[ultra thick]  (RightBottom) to   (RBCenter);
\draw[ultra thick, final directed]  (LTCenter) to (LeftT);
\node at (i) {\mbox{\textbf{\emph{~}}}};
\node at (RA) {\mbox{\textbf{\emph{A}}}};
\node at (LA) {\mbox{\textbf{\emph{A}}}};
\node at (RB) {\mbox{\textbf{\emph{B}}}};
\node at (LB) {\mbox{\textbf{\emph{B}}}};
\end{tikzpicture}
}

\newcommand{\RSBpos}[1]
{
\begin{tikzpicture} [scale=#1, baseline=0pt]
   \coordinate (RightT) at (1.414,1.414) ;
    \coordinate (LeftT) at (-1.414,1.414) ;
 \coordinate (Right) at (1.414,-1.414) ;
    \coordinate (Left) at (-1.414,-1.414) ;
    \coordinate (Top) at (0,0.5) ;
    \coordinate (Bottom) at (0,-0.5) ;
    \coordinate (BCCenter) at (0,0) ;
\draw [densely dashed, ultra thick] (BCCenter) circle [radius=2];
\draw[ultra thick,  directed]  (Left) to [out=45,in=180]  (Bottom);
\draw[ultra thick,  reverse directed]  (Bottom) to [out=0,in=135]  (Right);
\draw[ultra thick, reverse directed]  (LeftT) to [out=315,in=180]  (Top);
\draw[ultra thick,  directed]  (Top) to [out=0,in=225]  (RightT);
\draw[thin, directed]  (Top) -- (Bottom)  node[midway,right]{\mbox{\textbf{\emph{~}}}};
\draw [fill=black, ultra thick] (Top) circle [radius=0.15];
\draw [fill=black, ultra thick] (Bottom) circle [radius=0.15];
\end{tikzpicture}
}

\newcommand{\RSApos}[1]
{
\begin{tikzpicture} [scale=#1, baseline=0pt]
  \coordinate (RightT) at (1.414,1.414) ;
    \coordinate (LeftT) at (-1.414,1.414) ;
 \coordinate (RightBottom) at (1.414,-1.414) ;
    \coordinate (LeftBottom) at (-1.414,-1.414) ;
     \coordinate (RightTop) at (0.5,0) ;
    \coordinate (LeftTop) at (-.5,0) ;
    \coordinate (Top) at (0,1.5) ;
        \coordinate (BCCenter) at (0,0) ;
\draw [densely dashed, ultra thick] (BCCenter) circle [radius=2];
\draw[ultra thick,  directed]  (LeftBottom) to [out=45,in=270]  (LeftTop);
\draw[ultra thick,  directed]  (RightBottom) to [out=135,in=270]  (RightTop);
\draw[ultra thick,  directed]  (RightTop) to [out=90,in=225]  (RightT);
\draw[ultra thick,  directed]  (LeftTop) to [out=90,in=315]  (LeftT);
\draw[thin]  (LeftTop) --  (RightTop) node[midway,above=.1]{\mbox{\textbf{\emph{~}}}};
\draw [fill=white, ultra thick] (LeftTop) circle [radius=0.15];
\draw [fill=white, ultra thick] (RightTop) circle [radius=0.15];
\end{tikzpicture}
}

\newcommand{\RScrneg}[1]
{
\begin{tikzpicture} [scale=#1, baseline=0pt]
  \coordinate (RightT) at (1.414,1.414) ;
    \coordinate (LeftT) at (-1.414,1.414) ;
 \coordinate (RightBottom) at (1.414,-1.414) ;
    \coordinate (LeftBottom) at (-1.414,-1.414) ;
        \coordinate (BCCenter) at (0,0) ;
        \coordinate (LBCenter) at (-.25,-.25) ;
         \coordinate (RTCenter) at (.25,.25) ;
         \coordinate (i) at (.75,0) ;
\coordinate (LA) at (-1.4,0) ;
\coordinate (RA) at (1.4,0) ;
\coordinate (RB) at (0,1.4) ;
\coordinate (LB) at (0,-1.4) ;
\draw [densely dashed, ultra thick] (BCCenter) circle [radius=2];
\draw[ultra thick, final directed]  (RightBottom) to  (LeftT);
\draw[ultra thick]  (LeftBottom) to   (LBCenter);
\draw[ultra thick, final directed]  (RTCenter) to (RightT);
\node at (i) {\mbox{\textbf{\emph{~}}}};
\node at (RA) {\mbox{\textbf{\emph{A}}}};
\node at (LA) {\mbox{\textbf{\emph{A}}}};
\node at (RB) {\mbox{\textbf{\emph{B}}}};
\node at (LB) {\mbox{\textbf{\emph{B}}}};
\end{tikzpicture}
}

\newcommand{\RSAneg}[1]
{
\begin{tikzpicture} [scale=#1, baseline=0pt]
   \coordinate (RightT) at (1.414,1.414) ;
    \coordinate (LeftT) at (-1.414,1.414) ;
 \coordinate (Right) at (1.414,-1.414) ;
    \coordinate (Left) at (-1.414,-1.414) ;
    \coordinate (Top) at (0,0.5) ;
    \coordinate (Bottom) at (0,-0.5) ;
    \coordinate (BCCenter) at (0,0) ;
\draw [densely dashed, ultra thick] (BCCenter) circle [radius=2];
\draw[ultra thick,  directed]  (Left) to [out=45,in=180]  (Bottom);
\draw[ultra thick,  reverse directed]  (Bottom) to [out=0,in=135]  (Right);
\draw[ultra thick, reverse directed]  (LeftT) to [out=315,in=180]  (Top);
\draw[ultra thick,  directed]  (Top) to [out=0,in=225]  (RightT);
\draw[thin, directed]  (Top) -- (Bottom)  node[midway,right]{\mbox{\textbf{\emph{~}}}};
\draw [fill=white, ultra thick] (Top) circle [radius=0.15];
\draw [fill=white, ultra thick] (Bottom) circle [radius=0.15];
\end{tikzpicture}
}

\newcommand{\RSBneg}[1]
{
\begin{tikzpicture} [scale=#1, baseline=0pt]
  \coordinate (RightT) at (1.414,1.414) ;
    \coordinate (LeftT) at (-1.414,1.414) ;
 \coordinate (RightBottom) at (1.414,-1.414) ;
    \coordinate (LeftBottom) at (-1.414,-1.414) ;
     \coordinate (RightTop) at (0.5,0) ;
    \coordinate (LeftTop) at (-.5,0) ;
    \coordinate (Top) at (0,1.5) ;
        \coordinate (BCCenter) at (0,0) ;
\draw [densely dashed, ultra thick] (BCCenter) circle [radius=2];
\draw[ultra thick,  directed]  (LeftBottom) to [out=45,in=270]  (LeftTop);
\draw[ultra thick,  directed]  (RightBottom) to [out=135,in=270]  (RightTop);
\draw[ultra thick,  directed]  (RightTop) to [out=90,in=225]  (RightT);
\draw[ultra thick,  directed]  (LeftTop) to [out=90,in=315]  (LeftT);
\draw[thin]  (LeftTop) --  (RightTop) node[midway,above=.1]{\mbox{\textbf{\emph{~}}}};
\draw [fill=black, ultra thick] (LeftTop) circle [radius=0.15];
\draw [fill=black, ultra thick] (RightTop) circle [radius=0.15];
\end{tikzpicture}
}

\newcommand{\ROFneg}[1]
{
\begin{tikzpicture} [scale=#1, baseline=0pt]
 \coordinate (RightTop) at (0.5,1) ;
    \coordinate (LeftTop) at (-.5,1) ;
    \coordinate (Top) at (0,1.5) ;
  \coordinate (RightT) at (1,0) ;
    \coordinate (LeftT) at (-1,0) ;
 \coordinate (RightBottom) at (1.414,-1.414) ;
    \coordinate (LeftBottom) at (-1.414,-1.414) ;
        \coordinate (BCCenter) at (0,0) ;
       \coordinate (LTCenter) at (-.25,.25) ;
         \coordinate (RBCenter) at (.25,-.25) ;
  \coordinate (i) at (.75,0) ;
\coordinate (LA) at (-1.4,0) ;
\coordinate (RA) at (1.4,0) ;
\coordinate (RB) at (0,1.4) ;
\coordinate (LB) at (0,-1.4) ;
\draw [densely dashed, ultra thick] (BCCenter) circle [radius=2];
\draw[ultra thick]  (RightTop) to [out=90,in=0]  (Top);
\draw[ultra thick]  (LeftTop) to [out=90,in=180]  (Top);
\draw[ultra thick]  (LeftBottom) to  (0,0);
\draw[ultra thick, final directed]  (0,0) to [out=45,in=270] (RightTop);
\draw[ultra thick]  (LeftTop) to [out=270,in=135] (LTCenter);
\draw[ultra thick, reverse directed]  (RightBottom) to   (RBCenter);
\node at (i) {\mbox{\textbf{\emph{~}}}};
\end{tikzpicture}
}

\newcommand{\ROFpos}[1]
{
\begin{tikzpicture} [scale=#1, baseline=0pt]
 \coordinate (RightTop) at (0.5,1) ;
    \coordinate (LeftTop) at (-.5,1) ;
    \coordinate (Top) at (0,1.5) ;
  \coordinate (RightT) at (1,0) ;
    \coordinate (LeftT) at (-1,0) ;
 \coordinate (RightBottom) at (1.414,-1.414) ;
    \coordinate (LeftBottom) at (-1.414,-1.414) ;
        \coordinate (BCCenter) at (0,0) ;
         \coordinate (LBCenter) at (-.25,-.25) ;
         \coordinate (RTCenter) at (.25,.25) ;
  \coordinate (i) at (.75,0) ;
\coordinate (LA) at (-1.4,0) ;
\coordinate (RA) at (1.4,0) ;
\coordinate (RB) at (0,1.4) ;
\coordinate (LB) at (0,-1.4) ;
\draw [densely dashed, ultra thick] (BCCenter) circle [radius=2];
\draw[ultra thick]  (RightTop) to [out=90,in=0]  (Top);
\draw[ultra thick]  (LeftTop) to [out=90,in=180]  (Top);
\draw[ultra thick,   directed] (0,0) to  (RightBottom) ;
\draw[ultra thick]  (LeftTop)  to [out=270,in=135] (0,0);
\draw[ultra thick, final directed]  (RTCenter)  to [out=45,in=270] (RightTop);
\draw[ultra thick]  (LeftBottom) to   (LBCenter);

\node at (i) {\mbox{\textbf{\emph{~}}}};
\end{tikzpicture}
}

\newcommand{\ROSsadcoh}[1]
{
\begin{tikzpicture} [scale=#1, baseline=0pt]

    \coordinate (Top) at (0,.75) ;
     \coordinate (Bottom) at (0,-.75) ;
     \coordinate (RightTop) at (1.414,1.414) ;
    \coordinate (LeftTop) at (-1.414,1.414) ;
 \coordinate (RightBottom) at (1.414,-1.414) ;
    \coordinate (LeftBottom) at (-1.414,-1.414) ;
        \coordinate (BCCenter) at (0,0) ;
         \coordinate (LTCenter) at (-1.25,.25) ;
         \coordinate (RBCenter) at (1,-.25) ;
         \coordinate (RTCenter) at (1.25,.25) ;
         \coordinate (LBCenter) at (-1,-.25) ;
  \coordinate (i) at (-.5,0) ;
    \coordinate (j) at (1.75,0) ;
\draw [densely dashed, ultra thick] (BCCenter) circle [radius=2];
\draw[ultra thick, reverse directed]  (RightTop) to  (RTCenter);
\draw[ultra thick,  directed]  (LeftTop) to   (LTCenter);
\draw[ultra thick]  (RBCenter)  to [out=225,in=0] (Bottom);
\draw[ultra thick, final directed]  (LBCenter)  to [out=315,in=180] (Bottom);
\draw[ultra thick, final directed]  (LeftBottom) to [out=90,in=180] (Top)  ;
\draw[ultra thick]    (Top) to [out=0,in=90] (RightBottom);

\node at (i) {\mbox{\textbf{\emph{~}}}};
\node at (j) {\mbox{\textbf{\emph{~}}}};
\end{tikzpicture}
}

\newcommand{\ROShapcoh}[1]
{
\begin{tikzpicture} [scale=#1, baseline=0pt]

    \coordinate (Top) at (0,.75) ;
     \coordinate (Bottom) at (0,-.75) ;
     \coordinate (RightTop) at (1.414,1.414) ;
    \coordinate (LeftTop) at (-1.414,1.414) ;
 \coordinate (RightBottom) at (1.414,-1.414) ;
    \coordinate (LeftBottom) at (-1.414,-1.414) ;
        \coordinate (BCCenter) at (0,0) ;
         \coordinate (LTCenter) at (-1,.25) ;
         \coordinate (RBCenter) at (1.25,-.25) ;
         \coordinate (RTCenter) at (1,.25) ;
         \coordinate (LBCenter) at (-1.25,-.25) ;
  \coordinate (i) at (-.5,0) ;
    \coordinate (j) at (1.75,0) ;
\draw [densely dashed, ultra thick] (BCCenter) circle [radius=2];
\draw[ultra thick, reverse directed]  (RightBottom) to  (RBCenter);
\draw[ultra thick,  directed]  (LeftBottom) to   (LBCenter);
\draw[ultra thick]  (RTCenter)  to [out=135,in=0] (Top);
\draw[ultra thick, final directed]  (LTCenter)  to [out=45,in=180] (Top);
\draw[ultra thick]  (LeftTop) to [out=270,in=180] (Bottom)  ;
\draw[ultra thick, reverse final directed]   (RightTop)  to [out=270,in=0] (Bottom) ;

\node at (i) {\mbox{\textbf{\emph{~}}}};
\node at (j) {\mbox{\textbf{\emph{~}}}};
\end{tikzpicture}
}

\newcommand{\ROSsadrev}[1]
{
\begin{tikzpicture} [scale=#1, baseline=0pt]

    \coordinate (Top) at (0,.75) ;
     \coordinate (Bottom) at (0,-.75) ;
     \coordinate (RightTop) at (1.414,1.414) ;
    \coordinate (LeftTop) at (-1.414,1.414) ;
 \coordinate (RightBottom) at (1.414,-1.414) ;
    \coordinate (LeftBottom) at (-1.414,-1.414) ;
        \coordinate (BCCenter) at (0,0) ;
         \coordinate (LTCenter) at (-1.25,.25) ;
         \coordinate (RBCenter) at (1,-.25) ;
         \coordinate (RTCenter) at (1.25,.25) ;
         \coordinate (LBCenter) at (-1,-.25) ;
  \coordinate (i) at (-.5,0) ;
    \coordinate (j) at (1.75,0) ;
\draw [densely dashed, ultra thick] (BCCenter) circle [radius=2];
\draw[ultra thick, reverse directed]  (RightTop) to  (RTCenter);
\draw[ultra thick,  directed]  (LeftTop) to   (LTCenter);
\draw[ultra thick]  (RBCenter)  to [out=225,in=0] (Bottom);
\draw[ultra thick, final directed]  (LBCenter)  to [out=315,in=180] (Bottom);
\draw[ultra thick] (Top)  to [out=180,in=90] (LeftBottom)   ;
\draw[ultra thick, final directed]   (RightBottom)  to [out=90,in=0] (Top);

\node at (i) {\mbox{\textbf{\emph{~}}}};
\node at (j) {\mbox{\textbf{\emph{~}}}};
\end{tikzpicture}
}

\newcommand{\ROShaprev}[1]
{
\begin{tikzpicture} [scale=#1, baseline=0pt]

    \coordinate (Top) at (0,.75) ;
     \coordinate (Bottom) at (0,-.75) ;
     \coordinate (RightTop) at (1.414,1.414) ;
    \coordinate (LeftTop) at (-1.414,1.414) ;
 \coordinate (RightBottom) at (1.414,-1.414) ;
    \coordinate (LeftBottom) at (-1.414,-1.414) ;
        \coordinate (BCCenter) at (0,0) ;
         \coordinate (LTCenter) at (-1,.25) ;
         \coordinate (RBCenter) at (1.25,-.25) ;
         \coordinate (RTCenter) at (1,.25) ;
         \coordinate (LBCenter) at (-1.25,-.25) ;
  \coordinate (i) at (-.5,0) ;
    \coordinate (j) at (1.75,0) ;
\draw [densely dashed, ultra thick] (BCCenter) circle [radius=2];
\draw[ultra thick,  directed]  (RightBottom) to  (RBCenter);
\draw[ultra thick, reverse directed]  (LeftBottom) to   (LBCenter);
\draw[ultra thick, final directed]  (RTCenter)  to [out=135,in=0] (Top);
\draw[ultra thick]  (LTCenter)  to [out=45,in=180] (Top);
\draw[ultra thick]  (Bottom) to [out=180,in=270] (LeftTop)  ;
\draw[ultra thick, reverse final directed]   (RightTop)  to [out=270,in=0] (Bottom) ;

\node at (i) {\mbox{\textbf{\emph{~}}}};
\node at (j) {\mbox{\textbf{\emph{~}}}};
\end{tikzpicture}
}

\newcommand{\ROTleft}[1]
{
\begin{tikzpicture} [scale=#1, baseline=0pt]
  \coordinate (Left) at (-4,0) ;
  \coordinate (Right) at (4,0) ;
 \coordinate (RightBottom) at (2.828,-2.828) ;
    \coordinate (LeftBottom) at (-2.828,-2.828) ;
 \coordinate (RightTop) at (2.828,2.828) ;
    \coordinate (LeftTop) at (-2.828,2.828) ;
              \coordinate (BCCenter) at (0,0) ;
\coordinate (Righti) at (-1.75,0) ;
    \coordinate (Lefti) at (-2.75,0) ;
          \coordinate (Rightj) at (2.5,0) ;
    \coordinate (Leftj) at (1.5,0) ;
\coordinate (Rightk) at (0.25,2.5) ;
    \coordinate (Leftk) at (-.5,2) ;
  \coordinate (i) at (-1.5,0.5) ;
    \coordinate (j) at (3,-1) ;
    \coordinate (k) at (1.5,2) ;
    \coordinate (BCCenter) at (0,0) ;
\draw [densely dashed, ultra thick] (BCCenter) circle [radius=4];
\draw[ultra thick, final directed]  (RightBottom) to [out=90,in=0]  (LeftTop);
\draw[ultra thick]  (LeftBottom) to [out=90,in=200]  (Leftk);
\draw[ultra thick, final directed]  (Rightk) to [out=45,in=170]  (RightTop);
\draw[ultra thick,  directed]  (Left) to   (Lefti);
\draw[ultra thick,  directed]  (Righti) to   (Leftj);
\draw[ultra thick, final directed]  (Rightj) to   (Right);

\node at (i) {\mbox{\textbf{\emph{~}}}};
\node at (j) {\mbox{\textbf{\emph{~}}}};
\node at (k) {\mbox{\textbf{\emph{~}}}};
\end{tikzpicture}
}

\newcommand{\ROTright}[1]
{
\begin{tikzpicture} [scale=#1, baseline=0pt]
  \coordinate (Left) at (-4,0) ;
  \coordinate (Right) at (4,0) ;
 \coordinate (RightBottom) at (2.828,-2.828) ;
    \coordinate (LeftBottom) at (-2.828,-2.828) ;
 \coordinate (RightTop) at (2.828,2.828) ;
    \coordinate (LeftTop) at (-2.828,2.828) ;
              \coordinate (BCCenter) at (0,0) ;
\coordinate (Righti) at (-1.75,0) ;
    \coordinate (Lefti) at (-2.75,0) ;
          \coordinate (Rightj) at (2.5,0) ;
    \coordinate (Leftj) at (1.5,0) ;
\coordinate (Rightk) at (0.5,-2) ;
    \coordinate (Leftk) at (-0,-2.5) ;
  \coordinate (i) at (-1.5,0.5) ;
    \coordinate (k) at (2,-2) ;
    \coordinate (j) at (3,.75) ;
    \coordinate (BCCenter) at (0,0) ;
\draw [densely dashed, ultra thick] (BCCenter) circle [radius=4];
\draw[ultra thick, final directed]  (RightBottom) to [out=180,in=270]  (LeftTop);
\draw[ultra thick]  (LeftBottom) to [out=345,in=225]  (Leftk);
\draw[ultra thick, final directed]  (Rightk) to [out=45,in=270]  (RightTop);
\draw[ultra thick,  directed]  (Left) to   (Lefti);
\draw[ultra thick,  directed]  (Righti) to   (Leftj);
\draw[ultra thick, final directed]  (Rightj) to   (Right);

\node at (i) {\mbox{\textbf{\emph{~}}}};
\node at (j) {\mbox{\textbf{\emph{~}}}};
\node at (k) {\mbox{\textbf{\emph{~}}}};
\end{tikzpicture}
}

\newcommand{\Emcr}[1]
{
\begin{tikzpicture} [scale=#1, baseline=0pt]
  \coordinate (RightT) at (1.414,1.414) ;
    \coordinate (LeftT) at (-1.414,1.414) ;
 \coordinate (RightBottom) at (1.414,-1.414) ;
    \coordinate (LeftBottom) at (-1.414,-1.414) ;
        \coordinate (BCCenter) at (0,0) ;
        \coordinate (LTCenter) at (-.25,.25) ;
         \coordinate (RBCenter) at (.25,-.25) ;
         \coordinate (i) at (.75,0) ;
\coordinate (LA) at  (0,-1.4) ;
\coordinate (RA) at (0,1.4) ;
\coordinate (RB) at  (1.4,0) ;
\coordinate (LB) at (-1.4,0) ;
\draw [densely dashed, ultra thick] (BCCenter) circle [radius=2];
\draw[ultra thick]  (LeftBottom) to  (RightT);
\draw[ultra thick]  (RightBottom) to   (RBCenter);
\draw[ultra thick]  (LTCenter) to (LeftT);
\node at (RA) {\mbox{\textbf{\emph{A}}}};
\node at (LA) {\mbox{\textbf{\emph{A}}}};
\node at (RB) {\mbox{\textbf{\emph{B}}}};
\node at (LB) {\mbox{\textbf{\emph{B}}}};
\end{tikzpicture}
}

\newcommand{\RSBKauf}[1]
{
\begin{tikzpicture} [scale=#1, baseline=0pt]
   \coordinate (RightT) at (1.414,1.414) ;
    \coordinate (LeftT) at (-1.414,1.414) ;
 \coordinate (Right) at (1.414,-1.414) ;
    \coordinate (Left) at (-1.414,-1.414) ;
    \coordinate (Top) at (0,0.5) ;
    \coordinate (Bottom) at (0,-0.5) ;
    \coordinate (BCCenter) at (0,0) ;
\draw [densely dashed, ultra thick] (BCCenter) circle [radius=2];
\draw[ultra thick]  (Left) to [out=45,in=180]  (Bottom);
\draw[ultra thick]  (Bottom) to [out=0,in=135]  (Right);
\draw[ultra thick]  (LeftT) to [out=315,in=180]  (Top);
\draw[ultra thick]  (Top) to [out=0,in=225]  (RightT);
\end{tikzpicture}
}

\newcommand{\RSAKauf}[1]
{
\begin{tikzpicture} [scale=#1, baseline=0pt]
  \coordinate (RightT) at (1.414,1.414) ;
    \coordinate (LeftT) at (-1.414,1.414) ;
 \coordinate (RightBottom) at (1.414,-1.414) ;
    \coordinate (LeftBottom) at (-1.414,-1.414) ;
     \coordinate (RightTop) at (0.5,0) ;
    \coordinate (LeftTop) at (-.5,0) ;
    \coordinate (Top) at (0,1.5) ;
        \coordinate (BCCenter) at (0,0) ;
\draw [densely dashed, ultra thick] (BCCenter) circle [radius=2];
\draw[ultra thick]  (LeftBottom) to [out=45,in=270]  (LeftTop);
\draw[ultra thick]  (RightBottom) to [out=135,in=270]  (RightTop);
\draw[ultra thick]  (RightTop) to [out=90,in=225]  (RightT);
\draw[ultra thick]  (LeftTop) to [out=90,in=315]  (LeftT);
\end{tikzpicture}
}

\newcommand{\crneg}[1]
{
\begin{tikzpicture} [scale=#1, baseline=0pt]
  \coordinate (RightT) at (1.414,1.414) ;
    \coordinate (LeftT) at (-1.414,1.414) ;
 \coordinate (RightBottom) at (1.414,-1.414) ;
    \coordinate (LeftBottom) at (-1.414,-1.414) ;
        \coordinate (BCCenter) at (0,0) ;
        \coordinate (LBCenter) at (-.25,-.25) ;
         \coordinate (RTCenter) at (.25,.25) ;
         \coordinate (i) at (.75,0) ;
\coordinate (LA) at (-1.4,0) ;
\coordinate (RA) at (1.4,0) ;
\coordinate (RB) at (0,1.4) ;
\coordinate (LB) at (0,-1.4) ;
\draw [densely dashed, ultra thick] (BCCenter) circle [radius=2];
\draw[ultra thick, final directed]  (RightBottom) to  (LeftT);
\draw[ultra thick]  (LeftBottom) to   (LBCenter);
\draw[ultra thick, final directed]  (RTCenter) to (RightT);
\node at (i) {\mbox{\textbf{\emph{~}}}};
\end{tikzpicture}
}

\newcommand{\crpos}[1]
{
\begin{tikzpicture} [scale=#1, baseline=0pt]
  \coordinate (RightT) at (1.414,1.414) ;
    \coordinate (LeftT) at (-1.414,1.414) ;
 \coordinate (RightBottom) at (1.414,-1.414) ;
    \coordinate (LeftBottom) at (-1.414,-1.414) ;
        \coordinate (BCCenter) at (0,0) ;
        \coordinate (LTCenter) at (-.25,.25) ;
         \coordinate (RBCenter) at (.25,-.25) ;
         \coordinate (i) at (.75,0) ;
\coordinate (LA) at  (0,-1.4) ;
\coordinate (RA) at (0,1.4) ;
\coordinate (RB) at  (1.4,0) ;
\coordinate (LB) at (-1.4,0) ;
\draw [densely dashed, ultra thick] (BCCenter) circle [radius=2];
\draw[ultra thick, final directed]  (LeftBottom) to  (RightT);
\node at (i) {\mbox{\textbf{\emph{~}}}};
\draw[ultra thick]  (RightBottom) to   (RBCenter);
\draw[ultra thick, final directed]  (LTCenter) to (LeftT);
\end{tikzpicture}
}

\newcommand{\crnegem}[1]
{
\begin{tikzpicture} [scale=#1, baseline=0pt]
  \coordinate (RightT) at (1.414,1.414) ;
    \coordinate (LeftT) at (-1.414,1.414) ;
 \coordinate (RightBottom) at (1.414,-1.414) ;
    \coordinate (LeftBottom) at (-1.414,-1.414) ;
        \coordinate (BCCenter) at (0,0) ;
        \coordinate (LBCenter) at (-.25,-.25) ;
         \coordinate (RTCenter) at (.25,.25) ;
         \coordinate (i) at (.75,0) ;
\coordinate (LA) at (-1.4,0) ;
\coordinate (RA) at (1.4,0) ;
\coordinate (RB) at (0,1.4) ;
\coordinate (LB) at (0,-1.4) ;
\draw [densely dashed, ultra thick] (BCCenter) circle [radius=2];
\draw[ultra thick, final directed]  (RightBottom) to  (LeftT);
\draw[ultra thick]  (LeftBottom) to   (LBCenter);
\draw[ultra thick, final directed]  (RTCenter) to (RightT);
\node at (RA) {\mbox{\textbf{\emph{A}}}};
\node at (LA) {\mbox{\textbf{\emph{A}}}};
\node at (RB) {\mbox{\textbf{\emph{B}}}};
\node at (LB) {\mbox{\textbf{\emph{B}}}};
\end{tikzpicture}
}

\newcommand{\crposem}[1]
{
\begin{tikzpicture} [scale=#1, baseline=0pt]
  \coordinate (RightT) at (1.414,1.414) ;
    \coordinate (LeftT) at (-1.414,1.414) ;
 \coordinate (RightBottom) at (1.414,-1.414) ;
    \coordinate (LeftBottom) at (-1.414,-1.414) ;
        \coordinate (BCCenter) at (0,0) ;
        \coordinate (LTCenter) at (-.25,.25) ;
         \coordinate (RBCenter) at (.25,-.25) ;
         \coordinate (i) at (.75,0) ;
\coordinate (LA) at  (0,-1.4) ;
\coordinate (RA) at (0,1.4) ;
\coordinate (RB) at  (1.4,0) ;
\coordinate (LB) at (-1.4,0) ;
\draw [densely dashed, ultra thick] (BCCenter) circle [radius=2];
\draw[ultra thick, final directed]  (LeftBottom) to  (RightT);
\draw[ultra thick]  (RightBottom) to   (RBCenter);
\draw[ultra thick, final directed]  (LTCenter) to (LeftT);
\node at (RA) {\mbox{\textbf{\emph{A}}}};
\node at (LA) {\mbox{\textbf{\emph{A}}}};
\node at (RB) {\mbox{\textbf{\emph{B}}}};
\node at (LB) {\mbox{\textbf{\emph{B}}}};
\end{tikzpicture}
}

\newcommand{\RSBposem}[1]
{
\begin{tikzpicture} [scale=#1, baseline=0pt]
   \coordinate (RightT) at (1.414,1.414) ;
    \coordinate (LeftT) at (-1.414,1.414) ;
 \coordinate (Right) at (1.414,-1.414) ;
    \coordinate (Left) at (-1.414,-1.414) ;
    \coordinate (Top) at (0,0.5) ;
    \coordinate (Bottom) at (0,-0.5) ;
    \coordinate (BCCenter) at (0,0) ;
\draw [densely dashed, ultra thick] (BCCenter) circle [radius=2];
\draw[ultra thick,  directed]  (Left) to [out=45,in=180]  (Bottom);
\draw[ultra thick,  reverse directed]  (Bottom) to [out=0,in=135]  (Right);
\draw[ultra thick, reverse directed]  (LeftT) to [out=315,in=180]  (Top);
\draw[ultra thick,  directed]  (Top) to [out=0,in=225]  (RightT);
\end{tikzpicture}
}

\newcommand{\RSAposem}[1]
{
\begin{tikzpicture} [scale=#1, baseline=0pt]
  \coordinate (RightT) at (1.414,1.414) ;
    \coordinate (LeftT) at (-1.414,1.414) ;
 \coordinate (RightBottom) at (1.414,-1.414) ;
    \coordinate (LeftBottom) at (-1.414,-1.414) ;
     \coordinate (RightTop) at (0.5,0) ;
    \coordinate (LeftTop) at (-.5,0) ;
    \coordinate (Top) at (0,1.5) ;
        \coordinate (BCCenter) at (0,0) ;
\draw [densely dashed, ultra thick] (BCCenter) circle [radius=2];
\draw[ultra thick,  directed]  (LeftBottom) to [out=45,in=270]  (LeftTop);
\draw[ultra thick,  directed]  (RightBottom) to [out=135,in=270]  (RightTop);
\draw[ultra thick,  directed]  (RightTop) to [out=90,in=225]  (RightT);
\draw[ultra thick,  directed]  (LeftTop) to [out=90,in=315]  (LeftT);
\end{tikzpicture}
}

\newcommand{\RSBAcohem}[1]
{
\begin{tikzpicture} [scale=#1, baseline=0pt]
 \coordinate (RightBottom) at (1.414,-1.414) ;
    \coordinate (LeftBottom) at (-1.414,-1.414) ;
     \coordinate (RightCenterBottom) at (.5,-1) ;
    \coordinate (LeftCenterBottom) at (-.5,-1) ;
 \coordinate (RightTop) at (1.414,1.414) ;
    \coordinate (LeftTop) at (-1.414,1.414) ;
     \coordinate (RightCenterTop) at (.5,1) ;
    \coordinate (LeftCenterTop) at (-.5,1) ;
    \coordinate (BCCenter) at (0,0) ;
\draw [densely dashed, ultra thick] (BCCenter) circle [radius=2];
\draw[ultra thick,  directed]  (LeftTop) to [out=315,in=180]  (LeftCenterTop);
\draw[ultra thick, reverse directed]  (RightCenterTop) to [out=0,in=180]   (LeftCenterTop);
\draw[ultra thick,  directed]  (RightCenterTop) to [out=0,in=225]  (RightTop);
\draw[ultra thick,  directed]  (LeftBottom) to [out=45,in=180]  (LeftCenterBottom);
\draw[ultra thick, reverse directed]  (RightCenterBottom) to [out=180,in=0]   (LeftCenterBottom);
\draw[ultra thick,  directed]  (RightCenterBottom) to [out=0,in=135]  (RightBottom);
\end{tikzpicture}
}

\newcommand{\RSABcohem}[1]
{
\begin{tikzpicture} [scale=#1, baseline=0pt]
 \coordinate (RightBottom) at (1.414,-1.414) ;
    \coordinate (LeftBottom) at (-1.414,-1.414) ;
     \coordinate (RightTop) at (1.414,1.414) ;
    \coordinate (LeftTop) at (-1.414,1.414) ;
    \coordinate (RightCenter) at (1.414,0) ;
    \coordinate (LeftCenter) at (-1.414,0) ;
    \coordinate (BCCenter) at (0,0) ;
    \coordinate (LCC) at (-0.5,0) ;
     \coordinate (RCC) at (0.5,0) ;
     \coordinate (BCC) at (0,-0.5) ;
      \coordinate (TCC) at (0,0.5) ;
\draw [densely dashed, ultra thick] (BCCenter) circle [radius=2];
\draw [ultra thick] (BCCenter) circle [radius=0.5];
\draw[ultra thick, reverse final directed]    (RCC) to [out=90,in=0] (TCC);
\draw[ultra thick, reverse final directed]    (RCC) to [out=270,in=0] (BCC);

\draw[ultra thick,  directed]  (LeftBottom) to  (LeftCenter);
\draw[ultra thick, reverse directed]  (LeftCenter) to   (LeftTop);

\draw[ultra thick, reverse directed]  (RightBottom) to   (RightCenter);
\draw[ultra thick,  directed]  (RightCenter) to  (RightTop);
\end{tikzpicture}
}

\newcommand{\RSBBcohem}[1]
{
\begin{tikzpicture} [scale=#1, baseline=0pt]
 \coordinate (RightBottom) at (1.414,-1.414) ;
    \coordinate (LeftBottom) at (-1.414,-1.414) ;
     \coordinate (RightTop) at (1.414,1.414) ;
    \coordinate (LeftTop) at (-1.414,1.414) ;
    \coordinate (RightCenter) at (1.414,0) ;
    \coordinate (LeftCenter) at (-1.414,0) ;
    \coordinate (BCCenter) at (0,0) ;
    \coordinate (LCC) at (-0.5,0) ;
     \coordinate (RCC) at (0.5,0) ;
     \coordinate (BCC) at (-.5,-1) ;
      \coordinate (TCC) at (-.5,1) ;
\draw [densely dashed, ultra thick] (BCCenter) circle [radius=2];

\draw[ultra thick, reverse  directed]    (RCC) to [out=90,in=330] (TCC);
\draw[ultra thick, reverse  directed]    (RCC) to [out=270,in=30] (BCC);

\draw[ultra thick,  directed]  (LeftBottom) to [out=30,in=210] (BCC);
\draw[ultra thick, reverse directed]  (TCC) to  [out=150,in=330] (LeftTop);

\draw[ultra thick, reverse directed]  (RightBottom) to   (RightCenter);
\draw[ultra thick,  directed]  (RightCenter) to  (RightTop);
\end{tikzpicture}
}

\newcommand{\RSAAcohem}[1]
{
\begin{tikzpicture} [scale=#1, baseline=0pt]
 \coordinate (RightBottom) at (-1.414,-1.414) ;
    \coordinate (LeftBottom) at (1.414,-1.414) ;
     \coordinate (RightTop) at (-1.414,1.414) ;
    \coordinate (LeftTop) at (1.414,1.414) ;
    \coordinate (RightCenter) at (-1.414,0) ;
    \coordinate (LeftCenter) at (1.414,0) ;
    \coordinate (BCCenter) at (0,0) ;
    \coordinate (LCC) at (0.5,0) ;
     \coordinate (RCC) at (-0.5,0) ;
     \coordinate (BCC) at (.5,-1) ;
      \coordinate (TCC) at (.5,1) ;
\draw [densely dashed, ultra thick] (BCCenter) circle [radius=2];

\draw[ultra thick, reverse  directed]    (TCC) to [out=210,in=90] (RCC);
\draw[ultra thick, reverse  directed]    (BCC) to [out=150,in=270] (RCC);

\draw[ultra thick, reverse directed]  (LeftBottom) to [out=150,in=330] (BCC);
\draw[ultra thick,  directed]  (TCC) to  [out=30,in=210] (LeftTop);

\draw[ultra thick,  directed]  (RightBottom) to   (RightCenter);
\draw[ultra thick, reverse directed]  (RightCenter) to  (RightTop);
\end{tikzpicture}
}

\newcommand{\RSABcohempty}[1]
{
\begin{tikzpicture} [scale=#1, baseline=0pt]
 \coordinate (RightBottom) at (1.414,-1.414) ;
    \coordinate (LeftBottom) at (-1.414,-1.414) ;
     \coordinate (RightTop) at (1.414,1.414) ;
    \coordinate (LeftTop) at (-1.414,1.414) ;
    \coordinate (RightCenter) at (1.414,0) ;
    \coordinate (LeftCenter) at (-1.414,0) ;
    \coordinate (BCCenter) at (0,0) ;
    \coordinate (LCC) at (-0.5,0) ;
     \coordinate (RCC) at (0.5,0) ;
     \coordinate (BCC) at (0,-0.5) ;
      \coordinate (TCC) at (0,0.5) ;
\draw [densely dashed, ultra thick] (BCCenter) circle [radius=2];
\draw [ultra thick] (BCCenter) circle [radius=0.5];
\draw[ultra thick]    (RCC) to [out=90,in=0] (TCC);
\draw[ultra thick]    (RCC) to [out=270,in=0] (BCC);

\draw[ultra thick]  (LeftBottom) to  (LeftCenter);
\draw[ultra thick]  (LeftCenter) to   (LeftTop);

\draw[ultra thick]  (RightBottom) to   (RightCenter);
\draw[ultra thick]  (RightCenter) to  (RightTop);
\end{tikzpicture}
}

\newcommand{\orcircle}[1]
{
\begin{tikzpicture} [scale=#1, baseline=0pt]
 \coordinate (RightBottom) at (1.414,-1.414) ;
    \coordinate (LeftBottom) at (-1.414,-1.414) ;
     \coordinate (RightTop) at (1.414,1.414) ;
    \coordinate (LeftTop) at (-1.414,1.414) ;
    \coordinate (RightCenter) at (1.414,0) ;
    \coordinate (LeftCenter) at (-1.414,0) ;
    \coordinate (BCCenter) at (0,0) ;
    \coordinate (LCC) at (-0.5,0) ;
     \coordinate (RCC) at (1,0) ;
     \coordinate (BCC) at (0,-1) ;
      \coordinate (TCC) at (0,0.5) ;
\draw [densely dashed, ultra thick] (BCCenter) circle [radius=2];
\draw [ultra thick] (BCCenter) circle [radius=1];
\draw[ultra thick, reverse final directed]    (RCC) to [out=270,in=0] (BCC);

\end{tikzpicture}
}

\newcommand{\emptyd}[1]
{
\begin{tikzpicture} [scale=#1, baseline=0pt]
 \coordinate (RightBottom) at (1.414,-1.414) ;
    \coordinate (LeftBottom) at (-1.414,-1.414) ;
     \coordinate (RightTop) at (1.414,1.414) ;
    \coordinate (LeftTop) at (-1.414,1.414) ;
    \coordinate (RightCenter) at (1.414,0) ;
    \coordinate (LeftCenter) at (-1.414,0) ;
    \coordinate (BCCenter) at (0,0) ;
    \coordinate (LCC) at (-0.5,0) ;
     \coordinate (RCC) at (.5,0) ;
     \coordinate (BCC) at (0,-.5) ;
      \coordinate (TCC) at (0,0.5) ;
\draw [densely dashed, ultra thick] (BCCenter) circle [radius=2];
\end{tikzpicture}
}

\newcommand{\kupscirc}[1]
{
\begin{tikzpicture} [scale=#1, baseline=0pt]
 \coordinate (RightBottom) at (1.414,-1.414) ;
    \coordinate (LeftBottom) at (-1.414,-1.414) ;
     \coordinate (RightTop) at (1.414,1.414) ;
    \coordinate (LeftTop) at (-1.414,1.414) ;
      \coordinate (BCCenter) at (0,0) ;
    \coordinate (LCC) at (-.5,0) ;
     \coordinate (RCC) at (.5,0) ;
     \coordinate (BCC) at (0,-0.5) ;
      \coordinate (TCC) at (0,0.5) ;
\coordinate (Left) at (-2,0) ;
      \coordinate (Right) at (2,0) ;
\draw [densely dashed, ultra thick] (BCCenter) circle [radius=2];
\draw [ultra thick] (BCCenter) circle [radius=0.5];
\draw[ultra thick, reverse final directed]    (LCC) to [out=90,in=180] (TCC);
\draw[ultra thick, reverse final directed]    (LCC) to [out=270,in=180] (BCC);

\draw[ultra thick, reverse directed]  (LCC) -- (Left)  ;
\draw[ultra thick, reverse directed]  (Right) --  (RCC);

\end{tikzpicture}
}

\newcommand{\kupltor}[1]
{
\begin{tikzpicture} [scale=#1, baseline=0pt]
 \coordinate (RightBottom) at (1.414,-1.414) ;
    \coordinate (LeftBottom) at (-1.414,-1.414) ;
     \coordinate (RightTop) at (1.414,1.414) ;
    \coordinate (LeftTop) at (-1.414,1.414) ;
    \coordinate (RightCenter) at (1.414,0) ;
    \coordinate (LeftCenter) at (-1.414,0) ;
    \coordinate (BCCenter) at (0,0) ;
    \coordinate (LCC) at (-2,0) ;
     \coordinate (RCC) at (2,0) ;
     \coordinate (BCC) at (0,-0.5) ;
      \coordinate (TCC) at (0,0.5) ;
\draw [densely dashed, ultra thick] (BCCenter) circle [radius=2];
\draw[ultra thick,  directed]    (LCC) to  (RCC);
\end{tikzpicture}
}

\newcommand{\kupsqv}[1]
{
\begin{tikzpicture} [scale=#1, baseline=0pt]
 \coordinate (RightBottom) at (1.414,-1.414) ;
    \coordinate (LeftBottom) at (-1.414,-1.414) ;
     \coordinate (RightTop) at (1.414,1.414) ;
    \coordinate (LeftTop) at (-1.414,1.414) ;
      \coordinate (BCCenter) at (0,0) ;
    \coordinate (LBC) at (-.75,-.75) ;
     \coordinate (RBC) at (.75,-.75) ;
     \coordinate (LTC) at (-.75,.75) ;
     \coordinate (RTC) at (.75,.75) ;
\draw [densely dashed, ultra thick] (BCCenter) circle [radius=2];
\draw[ultra thick,  directed]    (LBC) to  (RBC);
\draw[ultra thick, reverse  directed]    (LTC) to  (RTC);
\draw[ultra thick,  directed]  (LBC) -- (LTC)  ;
\draw[ultra thick, reverse directed]  (RBC) -- (RTC)  ;

\draw[ultra thick,  directed]    (RightBottom) to  (RBC);
\draw[ultra thick, reverse  directed]    (LTC) to  (LeftTop);
\draw[ultra thick,  reverse directed]   (LeftBottom) -- (LBC) ;
\draw[ultra thick, reverse directed]  (RightTop) -- (RTC)  ;

\end{tikzpicture}
}

\newcommand{\emptyrevvert}[1]
{
\begin{tikzpicture} [scale=#1, baseline=0pt]
 \coordinate (RightBottom) at (1.414,-1.414) ;
    \coordinate (LeftBottom) at (-1.414,-1.414) ;
     \coordinate (RightCenterBottom) at (.5,-1) ;
    \coordinate (LeftCenterBottom) at (-.5,-1) ;
 \coordinate (RightTop) at (1.414,1.414) ;
    \coordinate (LeftTop) at (-1.414,1.414) ;
     \coordinate (RightCenterTop) at (.5,1) ;
    \coordinate (LeftCenterTop) at (-.5,1) ;
    \coordinate (BCCenter) at (0,0) ;
\draw [densely dashed, ultra thick] (BCCenter) circle [radius=2];
\draw[ultra thick,  directed]  (LeftTop) to [out=315,in=45]  (LeftBottom);
\draw[ultra thick, reverse directed]  (RightTop) to [out=235,in=135]  (RightBottom);
\end{tikzpicture}
}

\newcommand{\kupline}[1]
{
\begin{tikzpicture} [scale=#1, baseline=0pt]
 \coordinate (RightBottom) at (1.414,-1.414) ;
    \coordinate (LeftBottom) at (-1.414,-1.414) ;
     \coordinate (RightTop) at (1.414,1.414) ;
    \coordinate (LeftTop) at (-1.414,1.414) ;
      \coordinate (BCCenter) at (0,0) ;
    \coordinate (LBC) at (0,-.75) ;
     \coordinate (RBC) at (.75,-.75) ;
     \coordinate (LTC) at (0,.75) ;
     \coordinate (RTC) at (.75,.75) ;
\draw [densely dashed, ultra thick] (BCCenter) circle [radius=2];

\draw[ultra thick, reverse  directed]  (LBC) -- (LTC)  ;

\draw[ultra thick,  directed]    (RightBottom) to  (LBC);
\draw[ultra thick, reverse  directed]    (LeftTop) to (LTC);
\draw[ultra thick,   directed]   (LeftBottom) -- (LBC) ;
\draw[ultra thick, reverse directed]  (RightTop) -- (LTC)  ;

\end{tikzpicture}
}

\newcommand{\emptycohvert}[1]
{
\begin{tikzpicture} [scale=#1, baseline=0pt]
 \coordinate (RightBottom) at (1.414,-1.414) ;
    \coordinate (LeftBottom) at (-1.414,-1.414) ;
     \coordinate (RightCenterBottom) at (.5,-1) ;
    \coordinate (LeftCenterBottom) at (-.5,-1) ;
 \coordinate (RightTop) at (1.414,1.414) ;
    \coordinate (LeftTop) at (-1.414,1.414) ;
     \coordinate (RightCenterTop) at (.5,1) ;
    \coordinate (LeftCenterTop) at (-.5,1) ;
    \coordinate (BCCenter) at (0,0) ;
\draw [densely dashed, ultra thick] (BCCenter) circle [radius=2];
\draw[ultra thick, reverse directed]  (LeftTop) to [out=315,in=45]  (LeftBottom);
\draw[ultra thick, reverse directed]  (RightTop) to [out=235,in=135]  (RightBottom);
\end{tikzpicture}
}

\newcommand{\RSABcohemi}[1]
{
\begin{tikzpicture} [scale=#1, baseline=0pt]
 \coordinate (RightBottom) at (1.414,-1.414) ;
    \coordinate (LeftBottom) at (-1.414,-1.414) ;
     \coordinate (RightTop) at (1.414,1.414) ;
    \coordinate (LeftTop) at (-1.414,1.414) ;
    \coordinate (RightCenter) at (1.414,0) ;
    \coordinate (LeftCenter) at (-1.414,0) ;
    \coordinate (BCCenter) at (0,0) ;
    \coordinate (LCC) at (-0.5,0) ;
     \coordinate (RCC) at (0.5,0) ;
     \coordinate (BCC) at (0,-0.5) ;
      \coordinate (TCC) at (0,0.5) ;
\draw [densely dashed, ultra thick] (BCCenter) circle [radius=2];
\draw [ultra thick] (BCCenter) circle [radius=0.5];
\draw[ultra thick, reverse final directed]    (RCC) to [out=90,in=0] (TCC);
\draw[ultra thick, reverse final directed]    (RCC) to [out=270,in=0] (BCC);

\draw[ultra thick,  directed]  (LeftBottom) to  (LeftCenter);
\draw[ultra thick, reverse directed]  (LeftCenter) to   (LeftTop);

\draw[ultra thick, reverse directed]  (RightBottom) to   (RightCenter);
\draw[ultra thick,  directed]  (RightCenter) to  (RightTop);
\draw[ultra thick, directed]  (RightCenter) --  (RCC) ;
\draw[ultra thick, reverse directed]  (LeftCenter) -- (LCC)  ;
\end{tikzpicture}
}

\newcommand{\RSBBcohemi}[1]
{
\begin{tikzpicture} [scale=#1, baseline=0pt]
 \coordinate (RightBottom) at (1.414,-1.414) ;
    \coordinate (LeftBottom) at (-1.414,-1.414) ;
     \coordinate (RightTop) at (1.414,1.414) ;
    \coordinate (LeftTop) at (-1.414,1.414) ;
    \coordinate (RightCenter) at (1.414,0) ;
    \coordinate (LeftCenter) at (-1.414,0) ;
    \coordinate (BCCenter) at (0,0) ;
    \coordinate (LCC) at (-0.5,0) ;
     \coordinate (RCC) at (0.5,0) ;
     \coordinate (BCC) at (-.5,-1) ;
      \coordinate (TCC) at (-.5,1) ;
\draw [densely dashed, ultra thick] (BCCenter) circle [radius=2];

\draw[ultra thick, reverse  directed]    (RCC) to [out=90,in=330] (TCC);
\draw[ultra thick, reverse  directed]    (RCC) to [out=270,in=30] (BCC);

\draw[ultra thick,  directed]  (LeftBottom) to [out=30,in=210] (BCC);
\draw[ultra thick, reverse directed]  (TCC) to  [out=150,in=330] (LeftTop);

\draw[ultra thick, reverse directed]  (RightBottom) to   (RightCenter);
\draw[ultra thick,  directed]  (RightCenter) to  (RightTop);
\draw[ultra thick, directed]  (RightCenter) --  (RCC) ;
\end{tikzpicture}
}

\newcommand{\RSAAcohemi}[1]
{
\begin{tikzpicture} [scale=#1, baseline=0pt]
 \coordinate (RightBottom) at (-1.414,-1.414) ;
    \coordinate (LeftBottom) at (1.414,-1.414) ;
     \coordinate (RightTop) at (-1.414,1.414) ;
    \coordinate (LeftTop) at (1.414,1.414) ;
    \coordinate (RightCenter) at (-1.414,0) ;
    \coordinate (LeftCenter) at (1.414,0) ;
    \coordinate (BCCenter) at (0,0) ;
    \coordinate (LCC) at (0.5,0) ;
     \coordinate (RCC) at (-0.5,0) ;
     \coordinate (BCC) at (.5,-1) ;
      \coordinate (TCC) at (.5,1) ;
\draw [densely dashed, ultra thick] (BCCenter) circle [radius=2];

\draw[ultra thick, reverse  directed]    (TCC) to [out=210,in=90] (RCC);
\draw[ultra thick, reverse  directed]    (BCC) to [out=150,in=270] (RCC);

\draw[ultra thick, reverse directed]  (LeftBottom) to [out=150,in=330] (BCC);
\draw[ultra thick,  directed]  (TCC) to  [out=30,in=210] (LeftTop);

\draw[ultra thick,  directed]  (RightBottom) to   (RightCenter);
\draw[ultra thick, reverse directed]  (RightCenter) to  (RightTop);
\draw[ultra thick, reverse directed]  (RightCenter) --  (RCC) ;
\end{tikzpicture}
}

\newcommand{\RSBAreveri}[1]
{
\begin{tikzpicture} [scale=#1, baseline=0pt]
 \coordinate (RightBottom) at (1.414,-1.414) ;
    \coordinate (LeftBottom) at (-1.414,-1.414) ;
     \coordinate (RightCenterBottom) at (.5,-1) ;
    \coordinate (LeftCenterBottom) at (-.5,-1) ;
 \coordinate (RightTop) at (1.414,1.414) ;
    \coordinate (LeftTop) at (-1.414,1.414) ;
     \coordinate (RightCenterTop) at (.5,1) ;
    \coordinate (LeftCenterTop) at (-.5,1) ;
    \coordinate (BCCenter) at (0,0) ;
\draw [densely dashed, ultra thick] (BCCenter) circle [radius=2];
\draw[ultra thick,  directed]  (LeftTop) to [out=315,in=180]  (LeftCenterTop);
\draw[ultra thick, reverse directed]  (LeftCenterTop) to   (RightCenterTop);
\draw[ultra thick,  directed]  (RightCenterTop) to [out=0,in=225]  (RightTop);
\draw[ultra thick, reverse directed]  (LeftBottom) to [out=45,in=180]  (LeftCenterBottom);
\draw[ultra thick, reverse directed]  (RightCenterBottom) to    (LeftCenterBottom);
\draw[ultra thick, reverse  directed]  (RightCenterBottom) to [out=0,in=135]  (RightBottom);
\draw[ultra thick,  directed]  (LeftCenterBottom) --  (LeftCenterTop) ;
\draw[ultra thick, reverse directed]  (RightCenterBottom) --  (RightCenterTop) ;

\end{tikzpicture}
}

\newcommand{\RSABreveri}[1]
{
\begin{tikzpicture} [scale=#1, baseline=0pt]
 \coordinate (RightBottom) at (1.414,-1.414) ;
    \coordinate (LeftBottom) at (-1.414,-1.414) ;
     \coordinate (RightTop) at (1.414,1.414) ;
    \coordinate (LeftTop) at (-1.414,1.414) ;
    \coordinate (RightCenter) at (1.414,0) ;
    \coordinate (LeftCenter) at (-1.414,0) ;
    \coordinate (BCCenter) at (0,0) ;
    \coordinate (LCC) at (-0.5,0) ;
     \coordinate (RCC) at (0.5,0) ;
     \coordinate (BCC) at (0,-0.5) ;
      \coordinate (TCC) at (0,0.5) ;
\draw [densely dashed, ultra thick] (BCCenter) circle [radius=2];
\draw [ultra thick] (BCCenter) circle [radius=0.5];
\draw[ultra thick, reverse final directed]    (LCC) to [out=90,in=180] (TCC);
\draw[ultra thick, reverse final directed]    (RCC) to [out=270,in=0] (BCC);

\draw[ultra thick, reverse directed]  (LeftBottom) to  (LeftCenter);
\draw[ultra thick, reverse directed]  (LeftCenter) to   (LeftTop);

\draw[ultra thick,  directed]  (RightBottom) to   (RightCenter);
\draw[ultra thick,  directed]  (RightCenter) to  (RightTop);

\end{tikzpicture}
}

\newcommand{\RSBBreveri}[1]
{
\begin{tikzpicture} [scale=#1, baseline=0pt]
 \coordinate (RightBottom) at (1.414,-1.414) ;
    \coordinate (LeftBottom) at (-1.414,-1.414) ;
     \coordinate (RightTop) at (1.414,1.414) ;
    \coordinate (LeftTop) at (-1.414,1.414) ;
    \coordinate (RightCenter) at (1.414,0) ;
    \coordinate (LeftCenter) at (-1.414,0) ;
    \coordinate (BCCenter) at (0,0) ;
    \coordinate (LCC) at (-0.5,0) ;
     \coordinate (RCC) at (0.5,0) ;
     \coordinate (BCC) at (-.5,-1) ;
      \coordinate (TCC) at (-.5,1) ;
\draw [densely dashed, ultra thick] (BCCenter) circle [radius=2];

\draw[ultra thick,   directed]    (RCC) to [out=90,in=330] (TCC);
\draw[ultra thick, reverse  directed]    (RCC) to [out=270,in=30] (BCC);

\draw[ultra thick, reverse directed]  (LeftBottom) to [out=30,in=210] (BCC);
\draw[ultra thick, reverse directed]  (TCC) to  [out=150,in=330] (LeftTop);

\draw[ultra thick,  directed]  (RightBottom) to   (RightCenter);
\draw[ultra thick,  directed]  (RightCenter) to  (RightTop);

\draw[ultra thick, reverse directed]  (TCC) -- (BCC)  ;

\end{tikzpicture}
}

\newcommand{\RSAAreveri}[1]
{
\begin{tikzpicture} [scale=#1, baseline=0pt]
 \coordinate (RightBottom) at (-1.414,-1.414) ;
    \coordinate (LeftBottom) at (1.414,-1.414) ;
     \coordinate (RightTop) at (-1.414,1.414) ;
    \coordinate (LeftTop) at (1.414,1.414) ;
    \coordinate (RightCenter) at (-1.414,0) ;
    \coordinate (LeftCenter) at (1.414,0) ;
    \coordinate (BCCenter) at (0,0) ;
    \coordinate (LCC) at (0.5,0) ;
     \coordinate (RCC) at (-0.5,0) ;
     \coordinate (BCC) at (.5,-1) ;
      \coordinate (TCC) at (.5,1) ;
\draw [densely dashed, ultra thick] (BCCenter) circle [radius=2];

\draw[ultra thick,   directed]    (TCC) to [out=210,in=90] (RCC);
\draw[ultra thick, reverse  directed]    (BCC) to [out=150,in=270] (RCC);

\draw[ultra thick,  directed]  (LeftBottom) to [out=150,in=330] (BCC);
\draw[ultra thick,  directed]  (TCC) to  [out=30,in=210] (LeftTop);

\draw[ultra thick, reverse directed]  (RightBottom) to   (RightCenter);
\draw[ultra thick, reverse directed]  (RightCenter) to  (RightTop);

\draw[ultra thick,  directed]  (TCC) -- (BCC)  ;

\end{tikzpicture}
}

\newcommand{\RSBposi}[1]
{
\begin{tikzpicture} [scale=#1, baseline=0pt]
   \coordinate (RightT) at (1.414,1.414) ;
    \coordinate (LeftT) at (-1.414,1.414) ;
 \coordinate (Right) at (1.414,-1.414) ;
    \coordinate (Left) at (-1.414,-1.414) ;
    \coordinate (Top) at (0,0.5) ;
    \coordinate (Bottom) at (0,-0.5) ;
    \coordinate (BCCenter) at (0,0) ;
    \coordinate (i) at (0,1.35) ;
    \coordinate (j) at (0,-1.35) ;
\draw [densely dashed, ultra thick] (BCCenter) circle [radius=2];
\draw[ultra thick,  directed]  (Left) to [out=45,in=270]  (Bottom);
\draw[ultra thick,  reverse directed]  (Bottom) to [out=270,in=135]  (Right);
\draw[ultra thick, reverse directed]  (LeftT) to [out=315,in=90]  (Top);
\draw[ultra thick, reverse directed]  (RightT) to [out=225,in=90] (Top) ;
\draw [fill=black, ultra thick] (Top) circle [radius=0.15];
\draw [fill=black, ultra thick] (Bottom) circle [radius=0.15];
\node at (i) {\mbox{\textbf{\emph{$\#$}}}};
\node at (j) {\mbox{\textbf{\emph{$\#$}}}};
\end{tikzpicture}
}

\newcommand{\kupltori}[1]
{
\begin{tikzpicture} [scale=#1, baseline=0pt]
 \coordinate (RightBottom) at (1.414,-1.414) ;
    \coordinate (LeftBottom) at (-1.414,-1.414) ;
     \coordinate (RightTop) at (1.414,1.414) ;
    \coordinate (LeftTop) at (-1.414,1.414) ;
    \coordinate (RightCenter) at (1.414,0) ;
    \coordinate (LeftCenter) at (-1.414,0) ;
    \coordinate (BCCenter) at (0,0) ;
    \coordinate (LCC) at (-2,0) ;
     \coordinate (RCC) at (2,0) ;
     \coordinate (BCC) at (-.75,1) ;
      \coordinate (TCC) at (.75,1) ;
       \coordinate (i) at (-.75,0) ;
    \coordinate (j) at (.75,0) ;
\draw [densely dashed, ultra thick] (BCCenter) circle [radius=2];
\draw[ultra thick,  directed]    (LCC) to  [out=45,in=240] (BCC);
\draw[ultra thick,  reverse final directed]    (BCC) to  [out=270,in=180] (BCCenter);
\draw[ultra thick]    (TCC) to  [out=270,in=0] (BCCenter);
\draw[ultra thick, reverse directed]    (TCC) to  [out=300,in=135] (RCC);
\draw [fill=black, ultra thick] (TCC) circle [radius=0.15];
\draw [fill=black, ultra thick] (BCC) circle [radius=0.15];
\node at (i) {\mbox{\textbf{\emph{$\#$}}}};
\node at (j) {\mbox{\textbf{\emph{$\#$}}}};
\end{tikzpicture}
}

\newcommand{\RFBposi}[1]
{
\begin{tikzpicture} [scale=#1, baseline=0pt]
 \coordinate (RightBottom) at (1.414,-1.414) ;
    \coordinate (LeftBottom) at (-1.414,-1.414) ;
     \coordinate (RightTop) at (0.5,0) ;
    \coordinate (LeftTop) at (-.5,0) ;
    \coordinate (Top) at (0,1.5) ;
        \coordinate (BCCenter) at (0,0) ;
\draw [densely dashed, ultra thick] (BCCenter) circle [radius=2];
\draw[ultra thick,  directed]  (LeftBottom) to [out=45,in=270]  (LeftTop);
\draw[ultra thick, reverse directed]  (RightBottom) to [out=135,in=270]  (RightTop);
\draw[ultra thick,  directed]  (RightTop) to [out=90,in=0]  (Top);
\draw[ultra thick, reverse directed]  (LeftTop) to [out=90,in=180]  (Top);
\end{tikzpicture}
}

\newcommand{\RFAposi}[1]
{
\begin{tikzpicture} [scale=#1, baseline=0pt]
 \coordinate (Right) at (1.414,-1.414) ;
    \coordinate (Left) at (-1.414,-1.414) ;
    \coordinate (Top) at (0,0.5) ;
    \coordinate (Bottom) at (0,-0.5) ;
    \coordinate (BCCenter) at (0,0) ;
    \coordinate (SCCenter) at (0,1) ;
\draw [densely dashed, ultra thick] (BCCenter) circle [radius=2];
\draw [ultra thick] (SCCenter) circle [radius=0.5];
\draw[ultra thick,  directed]  (Left) to [out=45,in=180]  (Bottom);
\draw[ultra thick,  directed]  (Bottom) to [out=0,in=135]  (Right);
\draw[ultra thick, reverse final directed]    (Top) to [out=180,in=270] (-0.5,1);
\draw[ultra thick, reverse final directed]    (0,1.5) to [out=0,in=90] (0.5,1);
\end{tikzpicture}
}

\newcommand{\RFemptyunor}[1]
{
\begin{tikzpicture} [scale=#1, baseline=0pt]
 \coordinate (RightBottom) at (1.414,-1.414) ;
    \coordinate (LeftBottom) at (-1.414,-1.414) ;
     \coordinate (RightTop) at (0.5,0) ;
    \coordinate (LeftTop) at (-.5,0) ;
    \coordinate (Top) at (0,1.5) ;
        \coordinate (BCCenter) at (0,0) ;
        \coordinate (Label) at (0,0.5) ;
\draw [densely dashed, ultra thick] (BCCenter) circle [radius=2];
\draw[ultra thick]  (LeftBottom) to [out=45,in=270]  (LeftTop);
\draw[ultra thick]  (RightBottom) to [out=135,in=270]  (RightTop);
\draw[ultra thick]  (RightTop) to [out=90,in=0]  (Top);
\draw[ultra thick]  (LeftTop) to [out=90,in=180]  (Top);
\end{tikzpicture}
}

\newcommand{\RSBAreverar}[1]
{
\begin{tikzpicture} [scale=#1, baseline=0pt]
 \coordinate (RightBottom) at (1.414,-1.414) ;
    \coordinate (LeftBottom) at (-1.414,-1.414) ;
     \coordinate (RightCenterBottom) at (.5,-1) ;
    \coordinate (LeftCenterBottom) at (-.5,-1) ;
 \coordinate (RightTop) at (1.414,1.414) ;
    \coordinate (LeftTop) at (-1.414,1.414) ;
     \coordinate (RightCenterTop) at (.5,1) ;
    \coordinate (LeftCenterTop) at (-.5,1) ;
    \coordinate (BCCenter) at (0,0) ;
     \coordinate (i) at (0.5,1.35) ;
    \coordinate (j) at (-.5,-1.35) ;
     \coordinate (ii) at (-.5,1.35) ;
    \coordinate (jj) at (.5,-1.35) ;

\draw [densely dashed, ultra thick] (BCCenter) circle [radius=2];
\draw[ultra thick,  directed]  (LeftTop) to [out=315,in=180]  (LeftCenterTop);
\draw[ultra thick, reverse directed]  (LeftCenterTop) to [out=330,in=210]   (RightCenterTop);
\draw[ultra thick,  directed]  (RightCenterTop) to [out=0,in=225]  (RightTop);
\draw[ultra thick, reverse directed]  (LeftBottom) to [out=45,in=180]  (LeftCenterBottom);
\draw[ultra thick, reverse directed]  (RightCenterBottom) to [out=150,in=30]   (LeftCenterBottom);
\draw[ultra thick, reverse  directed]  (RightCenterBottom) to [out=0,in=135]  (RightBottom);
\draw [fill=black, ultra thick] (LeftCenterBottom) circle [radius=0.15];
\draw [fill=black, ultra thick] (LeftCenterTop) circle [radius=0.15];
\draw [fill=black, ultra thick] (RightCenterBottom) circle [radius=0.15];
\draw [fill=black, ultra thick] (RightCenterTop) circle [radius=0.15];
\node at (ii) {\mbox{\textbf{\emph{$\#$}}}};
\node at (jj) {\mbox{\textbf{\emph{$\#$}}}};
\node at (i) {\mbox{\textbf{\emph{$\#$}}}};
\node at (j) {\mbox{\textbf{\emph{$\#$}}}};
\end{tikzpicture}
}

\newcommand{\RSABreverar}[1]
{
\begin{tikzpicture} [scale=#1, baseline=0pt]
 \coordinate (RightBottom) at (1.414,-1.414) ;
    \coordinate (LeftBottom) at (-1.414,-1.414) ;
     \coordinate (RightTop) at (1.414,1.414) ;
    \coordinate (LeftTop) at (-1.414,1.414) ;
    \coordinate (RightCenter) at (1.414,0) ;
    \coordinate (LeftCenter) at (-1.414,0) ;
    \coordinate (BCCenter) at (0,0) ;
    \coordinate (LCC) at (-0.5,0) ;
     \coordinate (RCC) at (0.5,0) ;
     \coordinate (BCC) at (0,-0.5) ;
      \coordinate (TCC) at (0,0.5) ;
\draw [densely dashed, ultra thick] (BCCenter) circle [radius=2];
\draw [ultra thick] (BCCenter) circle [radius=0.5];
\draw[ultra thick, reverse final directed]    (LCC) to [out=90,in=180] (TCC);
\draw[ultra thick, reverse final directed]    (RCC) to [out=270,in=0] (BCC);

\draw[ultra thick, reverse directed]  (LeftBottom) to  (LeftCenter);
\draw[ultra thick, reverse directed]  (LeftCenter) to   (LeftTop);

\draw[ultra thick,  directed]  (RightBottom) to   (RightCenter);
\draw[ultra thick,  directed]  (RightCenter) to  (RightTop);

\end{tikzpicture}
}

\newcommand{\RSBBreverar}[1]
{
\begin{tikzpicture} [scale=#1, baseline=0pt]
 \coordinate (RightBottom) at (1.414,-1.414) ;
    \coordinate (LeftBottom) at (-1.414,-1.414) ;
     \coordinate (RightTop) at (1.414,1.414) ;
    \coordinate (LeftTop) at (-1.414,1.414) ;
    \coordinate (RightCenter) at (1.414,0) ;
    \coordinate (LeftCenter) at (-1.414,0) ;
    \coordinate (BCCenter) at (0,0) ;
    \coordinate (LCC) at (-0.5,0) ;
     \coordinate (RCC) at (0.5,0) ;
     \coordinate (BCC) at (-.5,-1) ;
      \coordinate (TCC) at (-.5,1) ;
       \coordinate (i) at (0,1.35) ;
    \coordinate (j) at (0,-1.35) ;
\draw [densely dashed, ultra thick] (BCCenter) circle [radius=2];

\draw[ultra thick,   directed]    (RCC) to [out=90,in=0] (TCC);
\draw[ultra thick, reverse  directed]    (RCC) to [out=270,in=0] (BCC);

\draw[ultra thick, reverse directed]  (LeftBottom) to [out=30,in=210] (BCC);
\draw[ultra thick, reverse directed]  (TCC) to  [out=150,in=330] (LeftTop);

\draw[ultra thick,  directed]  (RightBottom) to   (RightCenter);
\draw[ultra thick,  directed]  (RightCenter) to  (RightTop);

\draw [fill=black, ultra thick] (BCC) circle [radius=0.15];

\draw [fill=black, ultra thick] (TCC) circle [radius=0.15];
\node at (i) {\mbox{\textbf{\emph{$\#$}}}};
\node at (j) {\mbox{\textbf{\emph{$\#$}}}};
\end{tikzpicture}
}

\newcommand{\RSAAreverar}[1]
{
\begin{tikzpicture} [scale=#1, baseline=0pt]
 \coordinate (RightBottom) at (-1.414,-1.414) ;
    \coordinate (LeftBottom) at (1.414,-1.414) ;
     \coordinate (RightTop) at (-1.414,1.414) ;
    \coordinate (LeftTop) at (1.414,1.414) ;
    \coordinate (RightCenter) at (-1.414,0) ;
    \coordinate (LeftCenter) at (1.414,0) ;
    \coordinate (BCCenter) at (0,0) ;
    \coordinate (LCC) at (0.5,0) ;
     \coordinate (RCC) at (-0.5,0) ;
     \coordinate (BCC) at (.5,-1) ;
      \coordinate (TCC) at (.5,1) ;
      \coordinate (i) at (0,1.35) ;
    \coordinate (j) at (0,-1.35) ;
\draw [densely dashed, ultra thick] (BCCenter) circle [radius=2];

\draw[ultra thick,   directed]    (TCC) to [out=180,in=90] (RCC);
\draw[ultra thick, reverse  directed]    (BCC) to [out=180,in=270] (RCC);

\draw[ultra thick,  directed]  (LeftBottom) to [out=150,in=330] (BCC);
\draw[ultra thick,  directed]  (TCC) to  [out=30,in=210] (LeftTop);

\draw[ultra thick, reverse directed]  (RightBottom) to   (RightCenter);
\draw[ultra thick, reverse directed]  (RightCenter) to  (RightTop);

\draw [fill=black, ultra thick] (BCC) circle [radius=0.15];
\draw [fill=black, ultra thick] (TCC) circle [radius=0.15];
\node at (i) {\mbox{\textbf{\emph{$\#$}}}};
\node at (j) {\mbox{\textbf{\emph{$\#$}}}};

\end{tikzpicture}
}

\newcommand{\RSBBreverarem}[1]
{
\begin{tikzpicture} [scale=#1, baseline=0pt]
 \coordinate (RightBottom) at (1.414,-1.414) ;
    \coordinate (LeftBottom) at (-1.414,-1.414) ;
     \coordinate (RightTop) at (1.414,1.414) ;
    \coordinate (LeftTop) at (-1.414,1.414) ;
    \coordinate (RightCenter) at (1.414,0) ;
    \coordinate (LeftCenter) at (-1.414,0) ;
    \coordinate (BCCenter) at (0,0) ;
    \coordinate (LCC) at (-0.5,0) ;
     \coordinate (RCC) at (0.5,0) ;
     \coordinate (BCC) at (-.5,-1) ;
      \coordinate (TCC) at (-.5,1) ;
       \coordinate (i) at (0,1.35) ;
    \coordinate (j) at (0,-1.35) ;
\draw [densely dashed, ultra thick] (BCCenter) circle [radius=2];

\draw[ultra thick, reverse directed]  (LeftBottom) to   (LeftCenter);
\draw[ultra thick,  reverse directed]  (LeftCenter) to  (LeftTop);

\draw[ultra thick,  directed]  (RightBottom) to   (RightCenter);
\draw[ultra thick,  directed]  (RightCenter) to  (RightTop);

\end{tikzpicture}
}

\newcommand{\vcr}[1]
{
\begin{tikzpicture} [scale=#1, , baseline=-2pt]
  \coordinate (RightT) at (1.414,1.414) ;
    \coordinate (LeftT) at (-1.414,1.414) ;
 \coordinate (RightBottom) at (1.414,-1.414) ;
    \coordinate (LeftBottom) at (-1.414,-1.414) ;
        \coordinate (BCCenter) at (0,0) ;

\draw[ thick]  (LeftBottom) to  (RightT);
\draw[ thick]  (RightBottom) to    (LeftT);
\draw [ thick] (BCCenter) circle [radius=1];
\end{tikzpicture}
}

\newcommand{\crposit}[1]
{
\begin{tikzpicture} [scale=#1,  baseline=-2pt]
  \coordinate (RightT) at (1.414,1.414) ;
    \coordinate (LeftT) at (-1.414,1.414) ;
 \coordinate (RightBottom) at (1.414,-1.414) ;
    \coordinate (LeftBottom) at (-1.414,-1.414) ;
        \coordinate (BCCenter) at (0,0) ;
        \coordinate (LTCenter) at (-.25,.25) ;
         \coordinate (RBCenter) at (.25,-.25) ;
         \coordinate (i) at (.75,0) ;
\coordinate (LA) at  (0,-1.4) ;
\coordinate (RA) at (0,1.4) ;
\coordinate (RB) at  (1.4,0) ;
\coordinate (LB) at (-1.4,0) ;
\draw[ thick]  (LeftBottom) to  (RightT);
\draw[ thick]  (RightBottom) to   (RBCenter);
\draw[ thick]  (LTCenter) to (LeftT);
\end{tikzpicture}
}

\newcommand{\crnegat}[1]
{
\begin{tikzpicture} [scale=#1, baseline=-2pt]
  \coordinate (RightT) at (1.414,1.414) ;
    \coordinate (LeftT) at (-1.414,1.414) ;
 \coordinate (RightBottom) at (1.414,-1.414) ;
    \coordinate (LeftBottom) at (-1.414,-1.414) ;
        \coordinate (BCCenter) at (0,0) ;
        \coordinate (LBCenter) at (-.25,-.25) ;
         \coordinate (RTCenter) at (.25,.25) ;
         \coordinate (i) at (.75,0) ;
\coordinate (LA) at (-1.4,0) ;
\coordinate (RA) at (1.4,0) ;
\coordinate (RB) at (0,1.4) ;
\coordinate (LB) at (0,-1.4) ;
\draw[ thick]  (RightBottom) to  (LeftT);
\draw[ thick]  (LeftBottom) to   (LBCenter);
\draw[ thick]  (RTCenter) to (RightT);
\end{tikzpicture}
}

\newcommand{\detmovet}[1]
{
\begin{tikzpicture} [scale=#1, baseline=0pt]
 \coordinate (Top) at (0,4) ;
  \coordinate (Bottom) at (0,-4) ;
  \coordinate (Left) at (-4,0) ;
  \coordinate (Right) at (4,0) ;
 \coordinate (RightBottom) at (2,-3.52) ;
    \coordinate (LeftBottom) at (-2,-3.52) ;
 \coordinate (RightTop) at (2,3.52) ;
    \coordinate (LeftTop) at (-2,3.52) ;
    \coordinate (BCCenter) at (0,0) ;
    \coordinate (Lefti) at (-2,2) ;
    \coordinate (Centeri) at (0,2) ;
     \coordinate (Righti) at (2,2) ;
\draw [densely dashed, ultra thick] (BCCenter) circle [radius=4];
 \draw[ultra thick]  (LeftBottom) to   (LeftTop);
  \draw[ultra thick]  (RightBottom) to   (RightTop);
    \draw[ultra thick]  (Bottom) to   (Top);
    \draw[ultra thick,fill=black] (-2.5,-1) rectangle (2.5,1);
\draw[ultra thick]  (Left) to   [out=70,in=180] (Lefti);
\draw[ultra thick]  (Lefti) to   (Centeri);
\draw[ultra thick]  (Centeri) to    (Righti);
\draw[ultra thick]  (Righti) to   [out=0,in=110] (Right);
\draw [ultra thick] (Lefti) circle [radius=.5];
\draw [ultra thick] (Centeri) circle [radius=.5];
\draw [ultra thick] (Righti) circle [radius=.5];
\end{tikzpicture}
}

\newcommand{\detmoveb}[1]
{
\begin{tikzpicture} [scale=#1, baseline=0pt]
 \coordinate (Top) at (0,4) ;
  \coordinate (Bottom) at (0,-4) ;
  \coordinate (Left) at (-4,0) ;
  \coordinate (Right) at (4,0) ;
 \coordinate (RightBottom) at (2,-3.52) ;
    \coordinate (LeftBottom) at (-2,-3.52) ;
 \coordinate (RightTop) at (2,3.52) ;
    \coordinate (LeftTop) at (-2,3.52) ;
    \coordinate (BCCenter) at (0,0) ;
    \coordinate (Lefti) at (-2,-2) ;
    \coordinate (Centeri) at (0,-2) ;
     \coordinate (Righti) at (2,-2) ;
\draw [densely dashed, ultra thick] (BCCenter) circle [radius=4];
 \draw[ultra thick]  (LeftBottom) to   (LeftTop);
  \draw[ultra thick]  (RightBottom) to   (RightTop);
    \draw[ultra thick]  (Bottom) to   (Top);
    \draw[ultra thick,fill=black] (-2.5,-1) rectangle (2.5,1);
\draw[ultra thick]  (Left) to   [out=290,in=180] (Lefti);
\draw[ultra thick]  (Lefti) to   (Centeri);
\draw[ultra thick]  (Centeri) to    (Righti);
\draw[ultra thick]  (Righti) to   [out=0,in=250] (Right);
\draw [ultra thick] (Lefti) circle [radius=.5];
\draw [ultra thick] (Centeri) circle [radius=.5];
\draw [ultra thick] (Righti) circle [radius=.5];
\end{tikzpicture}
}

\newcommand{\detmovebcir}[1]
{
\begin{tikzpicture} [scale=#1, baseline=0pt]
 \coordinate (Top) at (0,4) ;
  \coordinate (Bottom) at (0,-4) ;
  \coordinate (Left) at (-4,0) ;
  \coordinate (Right) at (4,0) ;
 \coordinate (RightBottom) at (2,-3.52) ;
    \coordinate (LeftBottom) at (-2,-3.52) ;
 \coordinate (RightTop) at (2,3.52) ;
    \coordinate (LeftTop) at (-2,3.52) ;
    \coordinate (BCCenter) at (0,0) ;
      \coordinate (SCCenter) at (0,1) ;
    \coordinate (Lefti) at (-2,-2.5) ;
    \coordinate (Centeri) at (0,-2.5) ;
     \coordinate (Righti) at (2,-2.5) ;
\draw [densely dashed, ultra thick] (BCCenter) circle [radius=4];
 \draw[ultra thick]  (LeftBottom) to   (LeftTop);
  \draw[ultra thick]  (RightBottom) to   (RightTop);
    \draw[ultra thick]  (Bottom) to   (Top);
    \draw[ultra thick,fill=black] (SCCenter) circle [radius=2.5]; 
\draw[ultra thick]  (Left) to   [out=0,in=150] (Lefti);
\draw[ultra thick]  (Lefti) to   (Centeri);
\draw[ultra thick]  (Centeri) to    (Righti);
\draw[ultra thick]  (Righti) to   [out=30,in=180] (Right);
\draw [ultra thick] (Lefti) circle [radius=.5];
\draw [ultra thick] (Centeri) circle [radius=.5];
\draw [ultra thick] (Righti) circle [radius=.5];
\end{tikzpicture}
}

\newcommand{\detmovetcir}[1]
{
\begin{tikzpicture} [scale=#1, baseline=0pt]
 \coordinate (Top) at (0,4) ;
  \coordinate (Bottom) at (0,-4) ;
  \coordinate (Left) at (-4,0) ;
  \coordinate (Right) at (4,0) ;
 \coordinate (RightBottom) at (2,-3.52) ;
    \coordinate (LeftBottom) at (-2,-3.52) ;
 \coordinate (RightTop) at (2,3.52) ;
    \coordinate (LeftTop) at (-2,3.52) ;
    \coordinate (BCCenter) at (0,0) ;
      \coordinate (SCCenter) at (0,-1) ;
    \coordinate (Lefti) at (-2,2.5) ;
    \coordinate (Centeri) at (0,2.5) ;
     \coordinate (Righti) at (2,2.5) ;
\draw [densely dashed, ultra thick] (BCCenter) circle [radius=4];
 \draw[ultra thick]  (LeftBottom) to   (LeftTop);
  \draw[ultra thick]  (RightBottom) to   (RightTop);
    \draw[ultra thick]  (Bottom) to   (Top);
    \draw[ultra thick,fill=black] (SCCenter) circle [radius=2.5]; 
\draw[ultra thick]  (Left) to   [out=0,in=210] (Lefti);
\draw[ultra thick]  (Lefti) to   (Centeri);
\draw[ultra thick]  (Centeri) to    (Righti);
\draw[ultra thick]  (Righti) to   [out=330,in=180] (Right);
\draw [ultra thick] (Lefti) circle [radius=.5];
\draw [ultra thick] (Centeri) circle [radius=.5];
\draw [ultra thick] (Righti) circle [radius=.5];
\end{tikzpicture}
}

\newcommand{\detbcir}[1]
{
\begin{tikzpicture} [scale=#1, baseline=0pt]
 \coordinate (Top) at (0,4) ;
  \coordinate (Bottom) at (0,-4) ;
  \coordinate (Left) at (-4,0) ;
  \coordinate (Right) at (4,0) ;
 \coordinate (RightBottom) at (2,-3.52) ;
    \coordinate (LeftBottom) at (-2,-3.52) ;
 \coordinate (RightTop) at (2,3.52) ;
    \coordinate (LeftTop) at (-2,3.52) ;
    \coordinate (BCCenter) at (0,0) ;
      \coordinate (SCCenter) at (0,1) ;
    \coordinate (Lefti) at (-2,-2) ;
    \coordinate (Centeri) at (0,-2) ;
     \coordinate (Righti) at (2,-2) ;
\draw [densely dashed, ultra thick] (BCCenter) circle [radius=4];
 \draw[ultra thick]  (LeftBottom) to   (LeftTop);
  \draw[ultra thick]  (RightBottom) to   (RightTop);
    \draw[ultra thick,fill=black] (SCCenter) circle [radius=2.5]; 
\draw[ultra thick]  (Left) to   [out=0,in=150] (Lefti);
\draw[ultra thick]  (Lefti) to   (Centeri);
\draw[ultra thick]  (Centeri) to    (Righti);
\draw[ultra thick]  (Righti) to   [out=30,in=180] (Right);
\draw [ultra thick] (Lefti) circle [radius=.5];
\draw [ultra thick] (Righti) circle [radius=.5];
\end{tikzpicture}
}

\newcommand{\detbcirloop}[1]
{
\begin{tikzpicture} [scale=#1, baseline=0pt]
 \coordinate (Top) at (0,4) ;
  \coordinate (Bottom) at (0,-4) ;
  \coordinate (Left) at (-4,0) ;
  \coordinate (Right) at (4,0) ;
 \coordinate (RightBottom) at (2,-3.52) ;
    \coordinate (LeftBottom) at (-2,-3.52) ;
 \coordinate (RightTop) at (2,3.52) ;
    \coordinate (LeftTop) at (-2,3.52) ;
    \coordinate (BCCenter) at (0,0) ;
      \coordinate (SCCenter) at (0,1) ;
    \coordinate (Lefti) at (-2,-1.5) ;
    \coordinate (Centeri) at (2,-2.5) ;
     \coordinate (Righti) at (2,-1.5) ;
         \coordinate (j) at (0,-3.75) ;
         \coordinate (k) at (0,-2.7) ;
\draw [densely dashed, ultra thick] (BCCenter) circle [radius=4];
 \draw[ultra thick]  (LeftBottom) to   (LeftTop);
  \draw[ultra thick]  (RightBottom) to   (RightTop);
    \draw[ultra thick,fill=black] (SCCenter) circle [radius=2.5]; 
\draw[ultra thick]  (Left) to   [out=0,in=150] (Lefti);
\draw[ultra thick]  (Lefti) to   [out=0,in=0] (j);
\draw[ultra thick]  (Righti) to   [out=180,in=180] (j);
\draw[ultra thick]  (Centeri) to    (Righti);
\draw[ultra thick]  (Righti) to   [out=30,in=180] (Right);
\draw [ultra thick] (Lefti) circle [radius=.5];
\draw [ultra thick] (Righti) circle [radius=.5];
\draw [ultra thick] (k) circle [radius=.5];
\end{tikzpicture}
}

\newcommand{\dettcir}[1]
{
\begin{tikzpicture} [scale=#1, baseline=0pt]
 \coordinate (Top) at (0,4) ;
  \coordinate (Bottom) at (0,-4) ;
  \coordinate (Left) at (-4,0) ;
  \coordinate (Right) at (4,0) ;
 \coordinate (RightBottom) at (2,-3.52) ;
    \coordinate (LeftBottom) at (-2,-3.52) ;
 \coordinate (RightTop) at (2,3.52) ;
    \coordinate (LeftTop) at (-2,3.52) ;
    \coordinate (BCCenter) at (0,0) ;
      \coordinate (SCCenter) at (0,-1) ;
    \coordinate (Lefti) at (-2,2) ;
    \coordinate (Centeri) at (0,2) ;
     \coordinate (Righti) at (2,2) ;
\draw [densely dashed, ultra thick] (BCCenter) circle [radius=4];
 \draw[ultra thick]  (LeftBottom) to   (LeftTop);
  \draw[ultra thick]  (RightBottom) to   (RightTop);
    \draw[ultra thick,fill=black] (SCCenter) circle [radius=2.5]; 
\draw[ultra thick]  (Left) to   [out=0,in=210] (Lefti);
\draw[ultra thick]  (Lefti) to   (Centeri);
\draw[ultra thick]  (Centeri) to    (Righti);
\draw[ultra thick]  (Righti) to   [out=330,in=180] (Right);
\draw [ultra thick] (Lefti) circle [radius=.5];
\draw [ultra thick] (Righti) circle [radius=.5];
\end{tikzpicture}
}

\newcommand{\dettcirbhn}[1]
{
\begin{tikzpicture} [scale=#1, baseline=0pt]
 \coordinate (Top) at (0,4) ;
  \coordinate (Bottom) at (0,-4) ;
  \coordinate (Left) at (-4,0) ;
  \coordinate (Right) at (4,0) ;
 \coordinate (RightBottom) at (2,-3.52) ;
    \coordinate (LeftBottom) at (-2,-3.52) ;
 \coordinate (RightTop) at (2,3.52) ;
    \coordinate (LeftTop) at (-2,3.52) ;
    \coordinate (BCCenter) at (0,0) ;
      \coordinate (SCCenter) at (0,-1) ;
   \coordinate (Lefti) at (-2,1.5) ;
    \coordinate (Centeri) at (0,1.5) ;
     \coordinate (Righti) at (2,1.5) ;
     \coordinate (Rightk) at (2,-1.5) ;
    \coordinate (Leftk) at (-2,-1.5) ;
\draw [densely dashed, ultra thick] (BCCenter) circle [radius=4];
 \draw[ultra thick, final directed]  (LeftBottom) to   (LeftTop);
  \draw[ultra thick, final directed]  (RightBottom) to   (RightTop);
\draw[ultra thick]  (Left) to   [out=0,in=210] (Lefti);
\draw[ultra thick]  (Lefti) to   (Centeri);
\draw[ultra thick]  (Centeri) to    (Righti);
\draw[ultra thick, final directed]  (Righti) to   [out=330,in=180] (Right);
\draw [ultra thick] (Lefti) circle [radius=.5];
\draw [ultra thick] (Righti) circle [radius=.5];
 \draw [fill=black, ultra thick] (Leftk) circle [radius=0.15];
\draw [fill=black, ultra thick] (Rightk) circle [radius=0.15];
\draw[thin]  (Leftk) --  (Rightk) node[midway,above=.1]{\mbox{\textbf{\emph{~}}}};
\end{tikzpicture}
}

\newcommand{\dettcirwvo}[1]
{
\begin{tikzpicture} [scale=#1, baseline=0pt]
 \coordinate (Top) at (0,4) ;
  \coordinate (Bottom) at (0,-1) ;
  \coordinate (BottomB) at (0,-3) ;
  \coordinate (Left) at (-4,0) ;
  \coordinate (Right) at (4,0) ;
  \coordinate (LeftBB) at (-2,-3.52) ;
  \coordinate (RightBB) at (2,-3.52) ;
 \coordinate (RightBottom) at (2,1) ;
    \coordinate (LeftBottom) at (-2,1) ;
 \coordinate (RightTop) at (2,3.52) ;
    \coordinate (LeftTop) at (-2,3.52) ;
    \coordinate (BCCenter) at (0,0) ;
      \coordinate (SCCenter) at (0,-1) ;
    \coordinate (Lefti) at (-2,1.5) ;
    \coordinate (Centeri) at (0,1.5) ;
     \coordinate (Righti) at (2,1.5) ;
     \coordinate (Rightk) at (0,-1) ;
    \coordinate (Leftk) at (0,-3) ;
\draw [densely dashed, ultra thick] (BCCenter) circle [radius=4];
 \draw[ultra thick, final directed]  (LeftBottom) to   (LeftTop);
  \draw[ultra thick, final directed]  (RightBottom) to   (RightTop);
\draw[ultra thick]  (Left) to   [out=0,in=210] (Lefti);
\draw[ultra thick]  (Lefti) to   (Centeri);
\draw[ultra thick]  (Centeri) to    (Righti);
\draw[ultra thick, final directed]  (Righti) to   [out=330,in=180] (Right);
\draw[ultra thick, reverse directed]  (LeftBottom) to   [out=270,in=180] (Bottom);
\draw[ultra thick,  directed]  (Bottom) to   [out=0,in=270] (RightBottom);
\draw[ultra thick,  directed]  (LeftBB) to   [out=45,in=180] (BottomB);
\draw[ultra thick, reverse directed]  (BottomB) to   [out=0,in=135] (RightBB);
\draw [ultra thick] (Lefti) circle [radius=.5];
\draw [ultra thick] (Righti) circle [radius=.5];
 \draw [fill=white, ultra thick] (Leftk) circle [radius=0.15];
\draw [fill=white, ultra thick] (Rightk) circle [radius=0.15];
\draw[thin, reverse directed]  (Leftk) --  (Rightk) node[midway,above=.1]{\mbox{\textbf{\emph{~}}}};
\end{tikzpicture}
}

\newcommand{\detbcirbhn}[1]
{
\begin{tikzpicture} [scale=#1, baseline=0pt]
 \coordinate (Top) at (0,4) ;
  \coordinate (Bottom) at (0,-4) ;
  \coordinate (Left) at (-4,0) ;
  \coordinate (Right) at (4,0) ;
 \coordinate (RightBottom) at (2,-3.52) ;
    \coordinate (LeftBottom) at (-2,-3.52) ;
 \coordinate (RightTop) at (2,3.52) ;
    \coordinate (LeftTop) at (-2,3.52) ;
    \coordinate (BCCenter) at (0,0) ;
      \coordinate (SCCenter) at (0,-1) ;
   \coordinate (Lefti) at (-2,-1.5) ;
    \coordinate (Centeri) at (0,-1.5) ;
     \coordinate (Righti) at (2,-1.5) ;
     \coordinate (Rightk) at (2,1.5) ;
    \coordinate (Leftk) at (-2,1.5) ;
\draw [densely dashed, ultra thick] (BCCenter) circle [radius=4];
 \draw[ultra thick, final directed]  (LeftBottom) to   (LeftTop);
  \draw[ultra thick, final directed]  (RightBottom) to   (RightTop);
\draw[ultra thick]  (Left) to   [out=0,in=180] (Lefti);
\draw[ultra thick]  (Lefti) to   (Centeri);
\draw[ultra thick]  (Centeri) to    (Righti);
\draw[ultra thick, final directed]  (Righti) to   [out=0,in=180] (Right);
\draw [ultra thick] (Lefti) circle [radius=.5];
\draw [ultra thick] (Righti) circle [radius=.5];
 \draw [fill=black, ultra thick] (Leftk) circle [radius=0.15];
\draw [fill=black, ultra thick] (Rightk) circle [radius=0.15];
\draw[thin]  (Leftk) --  (Rightk) node[midway,above=.1]{\mbox{\textbf{\emph{~}}}};
\end{tikzpicture}
}

\newcommand{\detbcirwvo}[1]
{
\begin{tikzpicture} [scale=#1, baseline=0pt]
 \coordinate (Top) at (0,4) ;
  \coordinate (Bottom) at (0,1) ;
  \coordinate (BottomB) at (0,3) ;
  \coordinate (Left) at (-4,0) ;
  \coordinate (Right) at (4,0) ;
  \coordinate (LeftBB) at (-2,3.52) ;
  \coordinate (RightBB) at (2,3.52) ;
 \coordinate (RightBottom) at (2,-1) ;
    \coordinate (LeftBottom) at (-2,-1) ;
 \coordinate (RightTop) at (2,-3.52) ;
    \coordinate (LeftTop) at (-2,-3.52) ;
    \coordinate (BCCenter) at (0,0) ;
      \coordinate (SCCenter) at (0,-1) ;
    \coordinate (Lefti) at (-2,-1.5) ;
    \coordinate (Centeri) at (0,-1.5) ;
     \coordinate (Righti) at (2,-1.5) ;
     \coordinate (Rightk) at (0,1) ;
    \coordinate (Leftk) at (0,3) ;
\draw [densely dashed, ultra thick] (BCCenter) circle [radius=4];
 \draw[ultra thick]  (LeftBottom) to   (LeftTop);
  \draw[ultra thick]  (RightBottom) to   (RightTop);
\draw[ultra thick]  (Left) to   [out=0,in=210] (Lefti);
\draw[ultra thick]  (Lefti) to   (Centeri);
\draw[ultra thick]  (Centeri) to    (Righti);
\draw[ultra thick, final directed]  (Righti) to   [out=330,in=180] (Right);
\draw[ultra thick,  directed]  (LeftBottom) to   [out=90,in=180] (Bottom);
\draw[ultra thick, reverse directed]  (Bottom) to   [out=0,in=90] (RightBottom);
\draw[ultra thick, reverse directed]  (LeftBB) to   [out=325,in=180] (BottomB);
\draw[ultra thick,  directed]  (BottomB) to   [out=0,in=225] (RightBB);
\draw [ultra thick] (Lefti) circle [radius=.5];
\draw [ultra thick] (Righti) circle [radius=.5];
 \draw [fill=white, ultra thick] (Leftk) circle [radius=0.15];
\draw [fill=white, ultra thick] (Rightk) circle [radius=0.15];
\draw[thin,  directed]  (Leftk) --  (Rightk) node[midway,above=.1]{\mbox{\textbf{\emph{~}}}};
\end{tikzpicture}
}

\newcommand{\ROFnegv}[1]
{
\begin{tikzpicture} [scale=#1, baseline=0pt]
 \coordinate (RightTop) at (0.5,1) ;
    \coordinate (LeftTop) at (-.5,1) ;
    \coordinate (Top) at (0,1.5) ;
  \coordinate (RightT) at (1,0) ;
    \coordinate (LeftT) at (-1,0) ;
 \coordinate (RightBottom) at (1.414,-1.414) ;
    \coordinate (LeftBottom) at (-1.414,-1.414) ;
        \coordinate (BCCenter) at (0,0) ;
       \coordinate (LTCenter) at (-.25,.25) ;
         \coordinate (RBCenter) at (.25,-.25) ;
  \coordinate (i) at (.75,0) ;
\coordinate (LA) at (-1.4,0) ;
\coordinate (RA) at (1.4,0) ;
\coordinate (RB) at (0,1.4) ;
\coordinate (LB) at (0,-1.4) ;
\draw [densely dashed, ultra thick] (BCCenter) circle [radius=2];
\draw[ultra thick]  (RightTop) to [out=90,in=0]  (Top);
\draw[ultra thick]  (LeftTop) to [out=90,in=180]  (Top);
\draw[ultra thick]  (RBCenter)  to   (LTCenter) ;
\draw[ultra thick]  (LeftBottom) to  (0,0);
\draw[ultra thick]  (0,0) to [out=45,in=270] (RightTop);
\draw[ultra thick]  (LeftTop) to [out=270,in=135] (LTCenter);
\draw[ultra thick,]  (RightBottom) to   (RBCenter);
\draw [ultra thick] (0,0) circle [radius=.5];
\end{tikzpicture}
}

\newcommand{\RFemptyv}[1]
{
\begin{tikzpicture} [scale=#1, baseline=0pt]
 \coordinate (RightBottom) at (1.414,-1.414) ;
    \coordinate (LeftBottom) at (-1.414,-1.414) ;
     \coordinate (RightTop) at (0.5,0) ;
    \coordinate (LeftTop) at (-.5,0) ;
    \coordinate (Top) at (0,1.5) ;
        \coordinate (BCCenter) at (0,0) ;
        \coordinate (Label) at (0,0.5) ;
\draw [densely dashed, ultra thick] (BCCenter) circle [radius=2];
\draw[ultra thick]  (LeftBottom) to [out=45,in=270]  (LeftTop);
\draw[ultra thick]  (RightBottom) to [out=135,in=270]  (RightTop);
\draw[ultra thick]  (RightTop) to [out=90,in=0]  (Top);
\draw[ultra thick]  (LeftTop) to [out=90,in=180]  (Top);
\end{tikzpicture}
}

\newcommand{\ROSsadv}[1]
{
\begin{tikzpicture} [scale=#1, baseline=0pt]

    \coordinate (Top) at (0,.75) ;
     \coordinate (Bottom) at (0,-.75) ;
     \coordinate (RightTop) at (1.414,1.414) ;
    \coordinate (LeftTop) at (-1.414,1.414) ;
 \coordinate (RightBottom) at (1.414,-1.414) ;
    \coordinate (LeftBottom) at (-1.414,-1.414) ;
        \coordinate (BCCenter) at (0,0) ;
         \coordinate (LTCenter) at (-1,0) ;
         \coordinate (RBCenter) at (1,0) ;
         \coordinate (RTCenter) at (1,0) ;
         \coordinate (LBCenter) at (-1,0) ;
  \coordinate (i) at (-.5,0) ;
    \coordinate (j) at (1.75,0) ;
\draw [densely dashed, ultra thick] (BCCenter) circle [radius=2];
\draw[ultra thick]  (RightTop) to  (RTCenter);
\draw[ultra thick]  (LeftTop) to   (LTCenter);
\draw[ultra thick]  (RBCenter)  to [out=250,in=0] (Bottom);
\draw[ultra thick]  (LBCenter)  to [out=290,in=180] (Bottom);
\draw[ultra thick]  (LeftBottom) to [out=90,in=180] (Top)  ;
\draw[ultra thick]    (Top) to [out=0,in=90] (RightBottom);

\draw [ultra thick] (LBCenter) circle [radius=.5];
\draw [ultra thick] (RBCenter) circle [radius=.5];

\end{tikzpicture}
}

\newcommand{\RSemptyv}[1]
{
\begin{tikzpicture} [scale=#1, baseline=0pt]
 \coordinate (RightBottom) at (1.414,-1.414) ;
    \coordinate (LeftBottom) at (-1.414,-1.414) ;
     \coordinate (RightCenterBottom) at (.5,-1) ;
    \coordinate (LeftCenterBottom) at (-.5,-1) ;
 \coordinate (RightTop) at (1.414,1.414) ;
    \coordinate (LeftTop) at (-1.414,1.414) ;
     \coordinate (RightCenterTop) at (.5,1) ;
    \coordinate (LeftCenterTop) at (-.5,1) ;
    \coordinate (BCCenter) at (0,0) ;
\draw [densely dashed, ultra thick] (BCCenter) circle [radius=2];
\draw[ultra thick]  (LeftTop) to [out=315,in=235]  (RightTop);
\draw[ultra thick]  (LeftBottom) to [out=45,in=135]  (RightBottom);
\end{tikzpicture}
}

\newcommand{\Rthleft}[1]
{
\begin{tikzpicture} [scale=#1, baseline=0pt]
  \coordinate (Left) at (-4,0) ;
  \coordinate (Right) at (4,0) ;
 \coordinate (RightBottom) at (2.828,-2.828) ;
    \coordinate (LeftBottom) at (-2.828,-2.828) ;
 \coordinate (RightTop) at (2.828,2.828) ;
    \coordinate (LeftTop) at (-2.828,2.828) ;
              \coordinate (BCCenter) at (0,0) ;
\coordinate (Righti) at (2,0) ;
    \coordinate (Lefti) at (-2.5,0) ;
          \coordinate (Rightj) at (2.5,0) ;
    \coordinate (Leftj) at (1.5,0) ;
\coordinate (Rightk) at (0.25,2.5) ;
    \coordinate (Leftk) at (-.5,2) ;
  \coordinate (i) at (-1.5,0.5) ;
    \coordinate (j) at (3,-1) ;
    \coordinate (k) at (1.5,2) ;
    \coordinate (BCCenter) at (0,0) ;
\draw [densely dashed, ultra thick] (BCCenter) circle [radius=4];
\draw[ultra thick]  (RightBottom) to [out=90,in=0]  (LeftTop);
\draw[ultra thick]  (LeftBottom) to [out=90,in=200]  (Leftk);
\draw[ultra thick]  (Rightk) to [out=45,in=170]  (RightTop);
\draw[ultra thick]  (Left) to    (Right);
\draw [ultra thick] (Lefti) circle [radius=1];
\draw [ultra thick] (Righti) circle [radius=1];
\end{tikzpicture}
}

\newcommand{\Rthright}[1]
{
\begin{tikzpicture} [scale=#1, baseline=0pt]
  \coordinate (Left) at (-4,0) ;
  \coordinate (Right) at (4,0) ;
 \coordinate (RightBottom) at (2.828,-2.828) ;
    \coordinate (LeftBottom) at (-2.828,-2.828) ;
 \coordinate (RightTop) at (2.828,2.828) ;
    \coordinate (LeftTop) at (-2.828,2.828) ;
              \coordinate (BCCenter) at (0,0) ;
\coordinate (Righti) at (2,0) ;
    \coordinate (Lefti) at (-2,0) ;
          \coordinate (Rightj) at (2.5,0) ;
    \coordinate (Leftj) at (1.5,0) ;
\coordinate (Rightk) at (0.5,-2) ;
    \coordinate (Leftk) at (-0,-2.5) ;
  \coordinate (i) at (-1.5,0.5) ;
    \coordinate (k) at (2,-2) ;
    \coordinate (j) at (3,.75) ;
    \coordinate (BCCenter) at (0,0) ;
\draw [densely dashed, ultra thick] (BCCenter) circle [radius=4];
\draw[ultra thick]  (RightBottom) to [out=180,in=270]  (LeftTop);
\draw[ultra thick]  (LeftBottom) to [out=345,in=225]  (Leftk);
\draw[ultra thick]  (Rightk) to [out=45,in=270]  (RightTop);
\draw[ultra thick]  (Left) to   (Right);

\draw [ultra thick] (Lefti) circle [radius=1];
\draw [ultra thick] (Righti) circle [radius=1];
\end{tikzpicture}
}

\newcommand{\Rthleftor}[1]
{
\begin{tikzpicture} [scale=#1, baseline=0pt]
  \coordinate (Left) at (-4,0) ;
  \coordinate (Right) at (4,0) ;
 \coordinate (RightBottom) at (2.828,-2.828) ;
    \coordinate (LeftBottom) at (-2.828,-2.828) ;
 \coordinate (RightTop) at (2.828,2.828) ;
    \coordinate (LeftTop) at (-2.828,2.828) ;
              \coordinate (BCCenter) at (0,0) ;
\coordinate (Righti) at (2,0) ;
    \coordinate (Lefti) at (-2.5,0) ;
          \coordinate (Rightj) at (2.5,0) ;
    \coordinate (Leftj) at (1.5,0) ;
\coordinate (Rightk) at (0.25,2.5) ;
    \coordinate (Leftk) at (-.5,2) ;
  \coordinate (i) at (-1.5,0.5) ;
    \coordinate (j) at (3,-1) ;
    \coordinate (k) at (1.5,2) ;
    \coordinate (BCCenter) at (0,0) ;
\draw [densely dashed, ultra thick] (BCCenter) circle [radius=4];
\draw[ultra thick, final directed]  (RightBottom) to [out=90,in=0]  (LeftTop);
\draw[ultra thick]  (LeftBottom) to [out=90,in=200]  (Leftk);
\draw[ultra thick, final directed]  (Rightk) to [out=45,in=170]  (RightTop);
\draw[ultra thick, final directed]  (Left) to    (Right);
\draw [ultra thick] (Lefti) circle [radius=.5];
\draw [ultra thick] (Righti) circle [radius=.5];
\end{tikzpicture}
}

\newcommand{\Rthrightor}[1]
{
\begin{tikzpicture} [scale=#1, baseline=0pt]
  \coordinate (Left) at (-4,0) ;
  \coordinate (Right) at (4,0) ;
 \coordinate (RightBottom) at (2.828,-2.828) ;
    \coordinate (LeftBottom) at (-2.828,-2.828) ;
 \coordinate (RightTop) at (2.828,2.828) ;
    \coordinate (LeftTop) at (-2.828,2.828) ;
              \coordinate (BCCenter) at (0,0) ;
\coordinate (Righti) at (2,0) ;
    \coordinate (Lefti) at (-2,0) ;
          \coordinate (Rightj) at (2.5,0) ;
    \coordinate (Leftj) at (1.5,0) ;
\coordinate (Rightk) at (0.5,-2) ;
    \coordinate (Leftk) at (-0,-2.5) ;
  \coordinate (i) at (-1.5,0.5) ;
    \coordinate (k) at (2,-2) ;
    \coordinate (j) at (3,.75) ;
    \coordinate (BCCenter) at (0,0) ;
\draw [densely dashed, ultra thick] (BCCenter) circle [radius=4];
\draw[ultra thick, final directed]  (RightBottom) to [out=180,in=270]  (LeftTop);
\draw[ultra thick]  (LeftBottom) to [out=345,in=225]  (Leftk);
\draw[ultra thick, final directed]  (Rightk) to [out=45,in=270]  (RightTop);
\draw[ultra thick, final directed]  (Left) to   (Right);

\draw [ultra thick] (Lefti) circle [radius=.5];
\draw [ultra thick] (Righti) circle [radius=.5];
\end{tikzpicture}
}

\newcommand{\Rthleftv}[1]
{
\begin{tikzpicture} [scale=#1, baseline=0pt]
  \coordinate (Left) at (-4,0) ;
  \coordinate (Right) at (4,0) ;
 \coordinate (RightBottom) at (2.828,-2.828) ;
    \coordinate (LeftBottom) at (-2.828,-2.828) ;
 \coordinate (RightTop) at (2.828,2.828) ;
    \coordinate (LeftTop) at (-2.828,2.828) ;
              \coordinate (BCCenter) at (0,0) ;
\coordinate (Righti) at (2,0) ;
    \coordinate (Lefti) at (-2,0) ;
          \coordinate (Rightj) at (0,2) ;
    \coordinate (Leftj) at (1.5,0) ;
\coordinate (Rightk) at (0.25,2.5) ;
    \coordinate (Leftk) at (-.5,2) ;
  \coordinate (i) at (-1.5,0.5) ;
    \coordinate (j) at (3,-1) ;
    \coordinate (k) at (1.5,2) ;
    \coordinate (BCCenter) at (0,0) ;
\draw [densely dashed, ultra thick] (BCCenter) circle [radius=4];
\draw[ultra thick]  (RightBottom) to [out=90,in=0]  (LeftTop);
\draw[ultra thick]  (LeftBottom) to [out=90,in=250]  (Lefti);
\draw[ultra thick]  (Lefti) to [out=45,in=170]  (RightTop);
\draw[ultra thick]  (Left) to    (Right);
\draw [ultra thick] (Lefti) circle [radius=1];
\draw [ultra thick] (Righti) circle [radius=1];
\draw [ultra thick] (Rightj) circle [radius=1];
\end{tikzpicture}
}

\newcommand{\Rthrightv}[1]
{
\begin{tikzpicture} [scale=#1, baseline=0pt]
  \coordinate (Left) at (-4,0) ;
  \coordinate (Right) at (4,0) ;
 \coordinate (RightBottom) at (2.828,-2.828) ;
    \coordinate (LeftBottom) at (-2.828,-2.828) ;
 \coordinate (RightTop) at (2.828,2.828) ;
    \coordinate (LeftTop) at (-2.828,2.828) ;
              \coordinate (BCCenter) at (0,0) ;
\coordinate (Righti) at (2,0) ;
    \coordinate (Lefti) at (-2,0) ;
          \coordinate (Rightj) at (-.1,-2) ;
    \coordinate (Leftj) at (1.5,0) ;
\coordinate (Rightk) at (0.5,-2) ;
    \coordinate (Leftk) at (-0,-2.5) ;
  \coordinate (i) at (-1.5,0.5) ;
    \coordinate (k) at (2,-2) ;
    \coordinate (j) at (3,.75) ;
    \coordinate (BCCenter) at (0,0) ;
\draw [densely dashed, ultra thick] (BCCenter) circle [radius=4];
\draw[ultra thick]  (RightBottom) to [out=180,in=270]  (LeftTop);
\draw[ultra thick]  (LeftBottom) to [out=0,in=225]  (Righti);
\draw[ultra thick]  (Righti) to [out=45,in=270]  (RightTop);
\draw[ultra thick]  (Left) to   (Right);

\draw [ultra thick] (Lefti) circle [radius=1];
\draw [ultra thick] (Righti) circle [radius=1];
\draw [ultra thick] (Rightj) circle [radius=1];
\end{tikzpicture}
}

\newcommand{\Rthleftvhbo}[1]
{
\begin{tikzpicture} [scale=#1, baseline=0pt]
  \coordinate (Left) at (-4,0) ;
  \coordinate (Right) at (4,0) ;
 \coordinate (RightBottom) at (2.828,-2.828) ;
    \coordinate (LeftBottom) at (-2.828,-2.828) ;
 \coordinate (RightTop) at (2.828,2.828) ;
    \coordinate (LeftTop) at (-2.828,2.828) ;
              \coordinate (BCCenter) at (0,0) ;
\coordinate (Righti) at (2.828,0) ;
    \coordinate (Lefti) at (-2.828,0) ;
          \coordinate (Rightj) at (0,2) ;
    \coordinate (Leftj) at (1.5,0) ;
\coordinate (Rightk) at (2.828,1.5) ;
    \coordinate (Leftk) at (-2.828,1.5) ;
  \coordinate (i) at (-1.5,0.5) ;
    \coordinate (j) at (3,-1) ;
    \coordinate (k) at (1.5,2) ;
    \coordinate (BCCenter) at (0,0) ;
\draw [densely dashed, ultra thick] (BCCenter) circle [radius=4];
\draw[ultra thick]  (RightBottom) to   (RightTop);
\draw[ultra thick]  (LeftBottom) to   (LeftTop);
\draw[ultra thick]  (Left) to    (Right);
\draw [ultra thick] (Lefti) circle [radius=.5];
\draw [ultra thick] (Righti) circle [radius=.5];
\draw [fill=black, ultra thick] (Leftk) circle [radius=0.15];
\draw [fill=black, ultra thick] (Rightk) circle [radius=0.15];
\draw[thin, reverse directed]  (Leftk) --  (Rightk) node[midway,above=.1]{\mbox{\textbf{\emph{~}}}};
\end{tikzpicture}
}

\newcommand{\Rthrightvhbo}[1]
{
\begin{tikzpicture} [scale=#1, baseline=0pt]
  \coordinate (Left) at (-4,0) ;
  \coordinate (Right) at (4,0) ;
 \coordinate (RightBottom) at (2.828,-2.828) ;
    \coordinate (LeftBottom) at (-2.828,-2.828) ;
 \coordinate (RightTop) at (2.828,2.828) ;
    \coordinate (LeftTop) at (-2.828,2.828) ;
              \coordinate (BCCenter) at (0,0) ;
\coordinate (Righti) at (2.828,0) ;
    \coordinate (Lefti) at (-2.828,0) ;
          \coordinate (Rightj) at (0,2) ;
    \coordinate (Leftj) at (1.5,0) ;
\coordinate (Rightk) at (2.828,-1.5) ;
    \coordinate (Leftk) at (-2.828,-1.5) ;
  \coordinate (i) at (-1.5,0.5) ;
    \coordinate (j) at (3,-1) ;
    \coordinate (k) at (1.5,2) ;
    \coordinate (BCCenter) at (0,0) ;
\draw [densely dashed, ultra thick] (BCCenter) circle [radius=4];
\draw[ultra thick]  (RightBottom) to   (RightTop);
\draw[ultra thick]  (LeftBottom) to   (LeftTop);
\draw[ultra thick]  (Left) to    (Right);
\draw [ultra thick] (Lefti) circle [radius=.5];
\draw [ultra thick] (Righti) circle [radius=.5];
\draw [fill=black, ultra thick] (Leftk) circle [radius=0.15];
\draw [fill=black, ultra thick] (Rightk) circle [radius=0.15];
\draw[thin, reverse directed]  (Leftk) --  (Rightk) node[midway,above=.1]{\mbox{\textbf{\emph{~}}}};
\end{tikzpicture}
}

\newcommand{\exvktfour}[1]
{
\begin{tikzpicture} [scale=#1, baseline=0pt]
 \coordinate (Top) at (0,8) ;
  \coordinate (Bottom) at (0,-4) ;
  \coordinate (vci) at (-5,-1) ;
    \coordinate (vcj) at (-5,5) ;
      \coordinate (vck) at (5,5) ;
        \coordinate (vcl) at (5,-1) ;
          \coordinate (cci) at (-3,-1) ;
          \coordinate (ccit) at (-3,-.5) ;
          \coordinate (ccib) at (-3,-1.5) ;
    \coordinate (ccj) at (-3,5) ;
    \coordinate (ccjl) at (-3.5,5) ;
        \coordinate (ccjr) at (-2.5,5) ;
      \coordinate (cck) at (3,5) ;
      \coordinate (cckt) at (3,5.5) ;
          \coordinate (cckb) at (3,4.5) ;
        \coordinate (ccl) at (3,-1) ;
        \coordinate (ccll) at (2.5,-1) ;
        \coordinate (cclr) at (3.5,-1) ;

\foreach \position in
{(vci), (vcj), (vck), (vcl)}
\draw[ultra thick] \position circle [radius=.5];

\draw[ultra thick,final directed] (vcj) to [controls=+(180:5) and +(180:5)] (Top);
\draw[ultra thick] (vck) to [controls=+(0:5) and +(0:5)] (Top);
\draw[ultra thick] (vck) to  (ccjr);
\draw[ultra thick] (vcj) to  (ccjl);

\draw[ultra thick,final directed] (vci) to [controls=+(180:5) and +(180:5)] (Bottom);
\draw[ultra thick] (vcl) to [controls=+(0:5) and +(0:5)] (Bottom);
\draw[ultra thick] (vcl) to  (cclr);
\draw[ultra thick] (vci) to  (ccll);

\draw[ultra thick,directed] (vcl) to (vck);
\draw[ultra thick] (vcl) to [out=-90,in=-90,looseness=2] (ccl);
\draw[ultra thick] (ccl) to  (cckb);
\draw[ultra thick] (vck) to [out=90,in=90,looseness=2] (cckt);

\draw[ultra thick, reverse directed] (vci) to (vcj);
\draw[ultra thick] (vci) to [out=-90,in=-90,looseness=2] (ccib);
\draw[ultra thick] (ccit) to  (ccj);
\draw[ultra thick] (vcj) to [out=90,in=90,looseness=2] (ccj);
\end{tikzpicture}
}

\newcommand{\exvktfourhex}[1]
{
\begin{tikzpicture} [scale=#1, baseline=0pt]
 \coordinate (Top) at (0,8) ;
  \coordinate (Bottom) at (0,-4) ;
  \coordinate (vci) at (-5,-1) ;
    \coordinate (vcj) at (-5,5) ;
      \coordinate (vck) at (5,5) ;
        \coordinate (vcl) at (5,-1) ;

         \coordinate (ccib) at (-2,-1) ;
          \coordinate (ccit) at (-3,1) ;

    \coordinate (ccjt) at (-3,5) ;
        \coordinate (ccjb) at (-2,3) ;

      \coordinate (cckt) at (3,5) ;
          \coordinate (cckb) at (2,3) ;

         \coordinate (cclb) at (2,-1) ;
          \coordinate (cclt) at (3,1) ;

\foreach \position in
{(vci), (vcj), (vck), (vcl)}
\draw[ultra thick] \position circle [radius=.5];

\draw[ultra thick,final directed] (vcj) to [controls=+(180:5) and +(180:5)] (Top);
\draw[ultra thick] (vck) to [controls=+(0:5) and +(0:5)] (Top);

\draw[ultra thick,final directed] (vci) to [controls=+(180:5) and +(180:5)] (Bottom);
\draw[ultra thick] (vcl) to [controls=+(0:5) and +(0:5)] (Bottom);

\draw[ultra thick,directed] (vcl) to (vck);

\draw[ultra thick, reverse directed] (vci) to (vcj);

\draw[ultra thick,  directed] (vci) to[out=-90,in=-90,looseness=2] (ccib);
\draw[ultra thick, reverse directed] (vci) to[out=0,in=-145] (ccit);
\draw[ultra thick, reverse directed] (ccib) to (cclb);
\draw[ultra thick, reverse directed] (ccib) to (ccit);

\draw[ultra thick, reverse directed] (vcj) to[out=90,in=90,looseness=2] (ccjt);
\draw[ultra thick, reverse directed] (vcj) to (ccjt);
\draw[ultra thick, reverse directed] (ccjb) to (ccjt);
\draw[ultra thick, reverse directed] (ccjb) to (cckb);
\draw[ultra thick, reverse directed] (ccjb) to (ccit);

\draw[ultra thick,  directed] (vck) to[out=90,in=90,looseness=2] (cckt);
\draw[ultra thick,  directed] (vck) to (cckt);
\draw[ultra thick,  directed] (cckb) to (cckt);
\draw[ultra thick,  directed] (cckb) to (cclt);

\draw[ultra thick, reverse directed] (vcl) to[out=-90,in=-90,looseness=2] (cclb);
\draw[ultra thick,  directed] (vcl) to[out=180,in=-45] (cclt);
\draw[ultra thick,  directed] (cclb) to (cclt);

\end{tikzpicture}
}

\newcommand{\exknotoid}[1]
{
\begin{tikzpicture} [scale=#1, baseline=0pt]
 \coordinate (RightTop) at (2,2) ;
    \coordinate (LeftTop) at (-2,2) ;
    \coordinate (Top) at (0,-3) ;
  \coordinate (RightT) at (1,0) ;
    \coordinate (LeftT) at (-1,0) ;
 \coordinate (RightBottom) at (1.414,-1.414) ;
    \coordinate (LeftBottom) at (-1.414,-1.414) ;
        \coordinate (BCCenter) at (0,0) ;
       \coordinate (LTCenter) at (-.25,.25) ;
         \coordinate (RBCenter) at (.25,-.25) ;
  \coordinate (i) at (.75,0) ;
\coordinate (LA) at (-1.4,0) ;
\coordinate (RA) at (1.4,0) ;
\coordinate (RB) at (0,1.4) ;
\coordinate (LB) at (0,-1.4) ;
\coordinate (SegB) at (-3,-3) ;
\coordinate (SegT) at (-2.25,-2.25) ;
\draw[ultra thick]  (RightTop) to [out=0,in=0]  (Top);
\draw[ultra thick]  (LeftTop) to [out=180,in=180]  (Top);
\draw[ultra thick]  (LeftBottom) to   (0,0);
\draw[ultra thick]  (0,0) to [out=45,in=180] (RightTop);
\draw[ultra thick]  (LeftTop) to [out=0,in=135] (LTCenter);
\draw[ultra thick,  directed]  (RightBottom) to   (RBCenter);
\node at (i) {\mbox{\textbf{\emph{~}}}};
\draw[ultra thick]  (SegT) to   (SegB);
\end{tikzpicture}
}

\newcommand{\Vfivekn}[1]
{
\begin{tikzpicture}[scale=#1, baseline=0pt]
 \coordinate (Right) at (1.414,1.414) ;
    \coordinate (Left) at (-1.414,1.414) ;
    \coordinate (Bottom) at (0,-2) ;
    \coordinate (BCCenter) at (0,0) ;
\draw [densely dashed, ultra thick] (BCCenter) circle [radius=2];
\draw[ultra thick, reverse directed]  (Bottom) --  (BCCenter) node[midway,right]{\mbox{\textbf{\emph{~}}}};
\draw [fill=green, ultra thick] (BCCenter) circle [radius=0.25];
\end{tikzpicture}
}

\newcommand{\Veightkn}[1]
{
\begin{tikzpicture}[scale=#1, baseline=0pt]
 \coordinate (Right) at (1.414,1.414) ;
    \coordinate (Left) at (-1.414,1.414) ;
    \coordinate (Bottom) at (0,-2) ;
    \coordinate (BCCenter) at (0,0) ;
     \coordinate (SCCenter) at (-3,-3) ;
\draw [densely dashed, ultra thick] (BCCenter) circle [radius=2];
\draw[ultra thick, directed]  (Bottom) --  (BCCenter) node[midway,right]{\mbox{\textbf{\emph{~}}}};
\draw [fill=green, ultra thick] (BCCenter) circle [radius=0.25];
\end{tikzpicture}
}

\newcommand{\exknotoidinfm}[1]
{
\begin{tikzpicture} [scale=#1, baseline=0pt]
 \coordinate (RightTop) at (2,2) ;
    \coordinate (LeftTop) at (-1,2) ;
    \coordinate (Top) at (0,-3) ;
  \coordinate (RightT) at (1,0) ;
    \coordinate (LeftT) at (-3,3) ;
 \coordinate (RightBottom) at (1.414,-1.414) ;
    \coordinate (LeftBottom) at (-1.414,-1.414) ;
        \coordinate (BCCenter) at (0,0) ;
       \coordinate (LTCenter) at (-.25,.25) ;
         \coordinate (RBCenter) at (.25,-.25) ;
  \coordinate (i) at (.75,0) ;
\coordinate (Left) at (-5,-4) ;
\coordinate (Bottom) at (0,-5) ;
\coordinate (Right) at (5,0) ;
\coordinate (Topp) at (3,5) ;
\coordinate (Bottomp) at (-2,-2) ;
\coordinate (SCCenter) at (-4,-2) ;
\coordinate (SegB) at (-3,-3) ;
\coordinate (SegT) at (-2.25,-2.25) ;
\draw[ultra thick]  (Bottom) to [out=180,in=325] (Left) ;
\draw[ultra thick]  (Right) to [out=270,in=0]  (Bottom);
\draw[ultra thick]  (Topp) to [out=0,in=90] (Right) ;
\draw[ultra thick]  (LeftTop) to [out=90,in=180]  (Topp);
\draw[ultra thick]  (Bottomp) to [out=90,in=0]  (LeftT);
\draw[ultra thick]  (Bottomp) to [out=270,in=180]  (Top);
\draw[ultra thick]  (RightTop) to [out=0,in=0]  (Top);
\draw[ultra thick]  (LeftT) to [out=180,in=135]  (Left);
\draw[ultra thick]  (LeftBottom) to   (0,0);
\draw[ultra thick]  (0,0) to [out=45,in=180] (RightTop);
\draw[ultra thick]  (LeftTop) to [out=270,in=135] (LTCenter);
\draw[ultra thick,  directed]  (RightBottom) to   (RBCenter);
\draw [fill=yellow, ultra thick] (SCCenter) circle [radius=0.25];
\draw[ultra thick]  (SegT) to   (SegB);
\end{tikzpicture}
}

\newcommand{\exknotoidinf}[1]
{
\begin{tikzpicture} [scale=#1, baseline=0pt]
 \coordinate (SCCenter) at (-4,-2) ;
 \coordinate (RightTop) at (2,2) ;
    \coordinate (LeftTop) at (-2,2) ;
    \coordinate (Top) at (0,-3) ;
  \coordinate (RightT) at (1,0) ;
    \coordinate (LeftT) at (-1,0) ;
 \coordinate (RightBottom) at (1.414,-1.414) ;
    \coordinate (LeftBottom) at (-1.414,-1.414) ;
        \coordinate (BCCenter) at (0,0) ;
       \coordinate (LTCenter) at (-.25,.25) ;
         \coordinate (RBCenter) at (.25,-.25) ;
  \coordinate (i) at (.75,0) ;
\coordinate (LA) at (-1.4,0) ;
\coordinate (RA) at (1.4,0) ;
\coordinate (RB) at (0,1.4) ;
\coordinate (LB) at (0,-1.4) ;
\coordinate (SegB) at (-3,-3) ;
\coordinate (SegT) at (-2.25,-2.25) ;
\draw[ultra thick]  (RightTop) to [out=0,in=0]  (Top);
\draw[ultra thick]  (LeftTop) to [out=180,in=180]  (Top);
\draw[ultra thick]  (LeftBottom) to   (0,0);
\draw[ultra thick]  (0,0) to [out=45,in=180] (RightTop);
\draw[ultra thick]  (LeftTop) to [out=0,in=135] (LTCenter);
\draw[ultra thick,  directed]  (RightBottom) to   (RBCenter);
\node at (i) {\mbox{\textbf{\emph{~}}}};
\draw [fill=yellow, ultra thick] (SCCenter) circle [radius=0.25];
\draw[ultra thick]  (SegT) to   (SegB);
\end{tikzpicture}
}

\newcommand{\insfigure}[3]
{
\begin{figure}
\center{\includegraphics[width=#1]{#2}}
\caption{#3}
\label{F_#2}
\end{figure}
}
\newcommand{\insfigureasformula}[4]
{
\begin{figure}[#1]
\center{#2}
\caption{#3}
\label{#4}
\end{figure}
}
\graphicspath{{images/}}

\newtheorem{teo}{Theorem}
\usepackage{lscape}

\begin{document}


\begin{flushleft}
AMS MSC: 57M25, 57M27
\end{flushleft}

\begin{center}
Labels instead of coefficients: a label bracket $\lbr{\cdot}$ which dominates 
   the~Jones polynomial $\jbr{\cdot}$, 
 the Kuperberg bracket $\kupbr{\cdot}$, and the normalised arrow polynomial~$\narbr{\cdot}$
\vspace*{30pt}

A.A.Akimova$^{1}$, V.O.Manturov$^{2}$

$^{1}$ South Ural State University

$^{2}$ Bauman Moscow State Technical University

and Novosibirsk State University
\end{center}

\footnotesize 
\begin{center}
\emph{Abstract}
\end{center}
In the present paper, we develop a picture formalism which gives rise to an invariant that dominates several known invariants of classical and virtual knots: the Jones polynomial $\jbr{\cdot}$, 
 the Kuperberg bracket $\kupbr{\cdot}$, and the normalised arrow polynomial~$\narbr{\cdot}$.

\emph{Keywords:} knot, invariant, picture-valued, skein relation, Kauffman bracket, \linebreak Kuperberg bracket, arrow polynomial.

\normalsize

\section{Introduction}

Many knot invariants originate from skein relations and similar relations. Usually, the strategy is as follows. One
resolves a crossing in two or more possible ways with certain coefficients and then evaluates
the resulting graph (or collection of circles) according to the given set of rules. The coefficients
are fixed polynomials. In the present paper, we suggest the method of using {\em labels instead
of coefficients} which allows one to get more degrees of freedom. Suppose we want to have
a coefficient $a^{2}$ (say) for a given diagram. Instead of putting $a^{2}$ before the whole
diagram, we just put two dots on (one or two) components of the graph each meaning $a$;
certainly, we may always restore $a^{2}$ from these two dots, but rather, we can settle more
complicated relations for taking factors (say, let one multiply labels only if they belong to the
same connected component). Starting with the usual Kauffman bracket skein relation,
putting labels and arcs connecting them we are led to a formalism without any coefficients
which dominates various known brackets (the 
  Jones polynomial~$\jbr{\cdot}$\cite{Jones}, 
  the Kuperberg bracket $\kupbr{\cdot}$~\cite{1991Kuperberg}, and the normalised arrow polynomial~$\narbr{\cdot}$~\cite{Kauffman}).

The relations in their most general form are made complicated on purpose.
Formally, the obtained invariant is an equivalence of pictures modulo relations.
We do not know whether it gives rise to a picture-valued invariant of classical
knots in the sense of
\cite{ManturovIlyutko}. 


Our paper is organised as follows. In Section~\ref{S_Def}, we define our invariant called the label bracket $\lbr{\cdot}$ and prove its
invariance.

In Section~\ref{S_part_cases}, we prove that our polynomial~$\lbr{\cdot}$ dominates the 
  Jones polynomial~$\jbr{\cdot}$\cite{Jones} 
  and the Kuperberg bracket $\kupbr{\cdot}$ \cite{1991Kuperberg}.

Sections~\ref{S_VK} and~\ref{S_knotoids} are devoted to virtual knots and knotoids, respectively. We prove
that our invariant  $\lbr{\cdot}$ gives rise to pictures. 
In addition,  Section~\ref{S_VKar} shows that  the normalised arrow polynomial~$\narbr{\cdot}$~\cite{Kauffman} can be considered as a particular case of the label bracket~$\lbr{\cdot}$.


The first named author was supported by the RFBR grant No.~17-01-00690. The second named author was supported by the Laboratory of Topology and Dynamics, Novosibirsk State University (grant No. 14.Y26.31.0025 of the government of the Russian Federation).

We express sincere gratitude to D.A.Fedoseev for various remarks.

\section{The label bracket $\lbr{\cdot}$ for classical knots}\label{S_Def}

By a \emph{classical label graph}  we mean a trivalent planar connected graph having $2\cdot n$
vertices of types given in  Fig.~\ref{F_Vertices.eps}, where $n\in \mathbb{Z}^{+}$.

\insfigureasformula{b}{$\begin{array}{cccccccc}
   \Vone{0.33} &  \Vtwo{0.33} &  \Vthree{0.33} & \Vfour{0.33} &
    \Vfive{0.33} & \Vsix{0.33} &  \Vseven{0.33} &  \Veight{0.33}\\
    &&&&&&&\\
  (V.1) & (V.2) & (V.3) & (V.4) & (V.5) & (V.6) & (V.7) & (V.8)
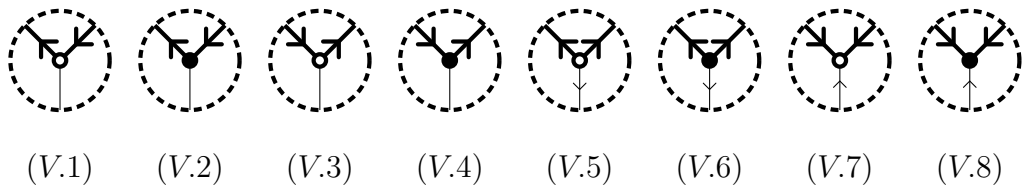
\end{array}$}{Types $(V.1)$ -- $(V.8)$ of vertices of  a classical label graph
}{F_Vertices.eps}


Namely, each vertex  is denoted by an empty or solid circle, see types $(V.1)$, $(V.3)$, $(V.5)$, $(V.7)$ and $(V.2)$, $(V.4)$, $(V.6)$, $(V.8)$, respectively. Among three edges that are incident to any vertex, there is a single edge denoted by thin line, while both others edges are denoted by thick line.  All edges are oriented except for thin edges that are adjacent to thick edges with coherent orientation, see types $(V.1)$ -- $(V.4)$.   
In-degree at each vertex of types $(V.5)$ and $(V.6)$ is   0, while in-degree at each vertex of types $(V.7)$ and $(V.8)$ is   3.

 Denote by $C(G)$ 
 a module over $\mathbb{Z}$ generated by classical label graphs  
  modulo relations
  \begin{equation}\label{Eq_lb_rel}
         R_{1.1}\mbox{,~} R_{1.2}\mbox{,~} R_{2.1} \mbox{~--~}  R_{2.4}\mbox{, and }R_{3.1}
           \end{equation}
           given in Figs.~\ref{F_Rules_R1.eps}--~\ref{F_R3givesRules.eps}.


\insfigureasformula{p}{$$
R_{1.1}: \lbr{\RFApos{0.7}+\RFBpos{0.7}} = \lbr{\RFempty{0.7}}
$$
$$
R_{1.2}: \lbr{\RFBneg{0.7}+\RFAneg{0.7}} = \lbr{\RFempty{0.7}}
$$}{Relations $R_{1.1}$ and $R_{1.2}$ of the label bracket $\lbr{\cdot}$}{F_Rules_R1.eps}

\insfigureasformula{p}{  \mbox{\hspace*{0cm}$
R_{2.1}: \lbr{\RSBAcoh{0.4}+\RSABcoh{0.4}+\RSBBcoh{0.4}+\RSAAcoh{0.4}} = \lbr{\RSemptycoh{0.4}}
$}
 \mbox{\hspace*{0cm}$
R_{2.2}: \lbr{\RSABncoh{0.4}+\RSBAncoh{0.4}+\RSAAncoh{0.4}+\RSBBncoh{0.4}} = \lbr{\RSemptycoh{0.4}}
$}
\mbox{\hspace*{0cm}$
R_{2.3}: \lbr{\RSBArever{0.4}+\RSABrever{0.4}+\RSBBrever{0.4}+\RSAArever{0.4}} = \lbr{\RSemptyrev{0.4}}
$}
\mbox{\hspace*{0cm}$
R_{2.4}: \lbr{\RSABrev{0.4}+\RSBArev{0.4}+\RSAArev{0.4}+\RSBBrev{0.4}} = \lbr{\RSemptyrev{0.4}}
$}}{Relations $R_{2.1}$ -- $R_{2.4}$ of the label bracket $\lbr{\cdot}$}{F_}

\insfigureasformula{p}{\mbox{\hspace*{0cm}$\begin{array}{c}
 R_{3.1}:\\  \lbr{\RTAABleft{0.35}+\RTBAAleft{0.35}+\RTBBBleft{0.35}+\RTBABleft{0.35}}+ \\
  \lbr{\RTABAleft{0.35}+\RTBBAleft{0.35}+\RTABBleft{0.35}+\RTAAAleft{0.35}}= \\
\end{array}$} \mbox{\hspace*{0cm}$\begin{array}{c}
   \lbr{\RTAABright{0.35}+\RTABAright{0.35}+\RTBBBright{0.35}+\RTABBright{0.35}}+ \\
  \lbr{\RTBAAright{0.35}+\RTAAAright{0.35} +\RTBABright{0.35}+\RTBBAright{0.35}} ~~~\\
\end{array}$}}{Relation $R_{3.1}$  of the label bracket $\lbr{\cdot}$}{F_R3givesRules.eps}


 Let $D$ be an oriented classical knot diagram having $n$ crossings.

Throughout the paper, by \emph{smoothing of  classical crossings }\label{D_sccr} of the diagram~$D$ we mean the following.
  Replace each classical crossing of $D$ with a sum of two fragments as follows.  Provide each angle of the crossing  with a marker $A$ or $B$ according to the relations given in Fig.~\ref{F_Rules_s12.eps} on the left, and smooth the crossing in the following two ways. First,  join together the two angles endowed with the marker $A$, and connect the pair of obtained arcs by a thin segment with endpoints marked as empty circles. Second,  join together two angles endowed with the marker $B$, and connect the pair of obtained arcs by a thin segment  with endpoints marked as solid circles.  Finally,  orient the segment connecting the arcs with not coherent orientation in a way such that the in-degree at each vertex is either 0 or 3.

 \insfigureasformula{t}{$$
R_{S.1}:\lbr{\RScrpos{0.6}}=\lbr{\RSApos{0.6}+\RSBpos{0.6}}
$$
$$
R_{S.2}: \lbr{\RScrneg{0.6}}=\lbr{\RSAneg{0.6}+\RSBneg{0.6}}
$$}{Smoothing relations $R_{S.1}$ and $R_{S.2}$  of the label bracket $\lbr{\cdot}$}{F_Rules_s12.eps}


Throughout the paper, each \emph{state}~\label{D_state} $s$ of the diagram $D$ is defined by a combination of ways to smooth
    classical crossings of $D$ such as to join together either  two angles endowed with a marker $A$, or two angles endowed with a marker $B$.


      The \emph{label bracket~$\lbr{D}$ of  an oriented classical knot  diagram $D$} is the sum

      \begin{equation}\label{Eq_lb_def}
     \sum\limits_{s=1}^{2^n} G_s (D)
        \end{equation}
        modulo relations~\eqref{Eq_lb_rel},  i.e.  the label bracket $\lbr{D}$ takes
values in $C(G)$. 

The sum~\eqref{Eq_lb_def} is taken over all possible states of the~diagram~$D$. Here $s$ is a state of the diagram $D$, and $G_s(D)$ is a classical label graph, which is obtained as a result of smoothing each crossing of $D$ according to the state $s$ by the  smoothing relations $R_{S.1}$ and $R_{S.2}$ given in Fig.~\ref{F_Rules_s12.eps}.



         \begin{teo}\label{Th_inv}
     The label bracket $\lbr{\cdot}$ is  invariant under isotopy of  classical links.
     \end{teo}
     \begin{proof}
     Figs.~\ref{F_R1.eps}~--~\ref{F_R3.eps} show that the label bracket $\lbr{\cdot}$ is  invariant under all three Reidemeister moves. It is well known that the considered variants of orientations are sufficient. 
      \end{proof}

      \insfigureasformula{p}{\mbox{\hspace*{0cm}$
      \begin{array}{c}
        \lbr{\ROFneg{0.7}}=  \lbr{\RFempty{0.7}} \\
        \lbr{\RFAneg{0.7}+\RFBneg{0.7}}=  \lbr{\RFempty{0.7}}\\
         \lbr{\ROFpos{0.7}} = \lbr{\RFempty{0.7}} \\
         \lbr{\RFApos{0.7}+\RFBpos{0.7}} =\lbr{\RFempty{0.7}}\\
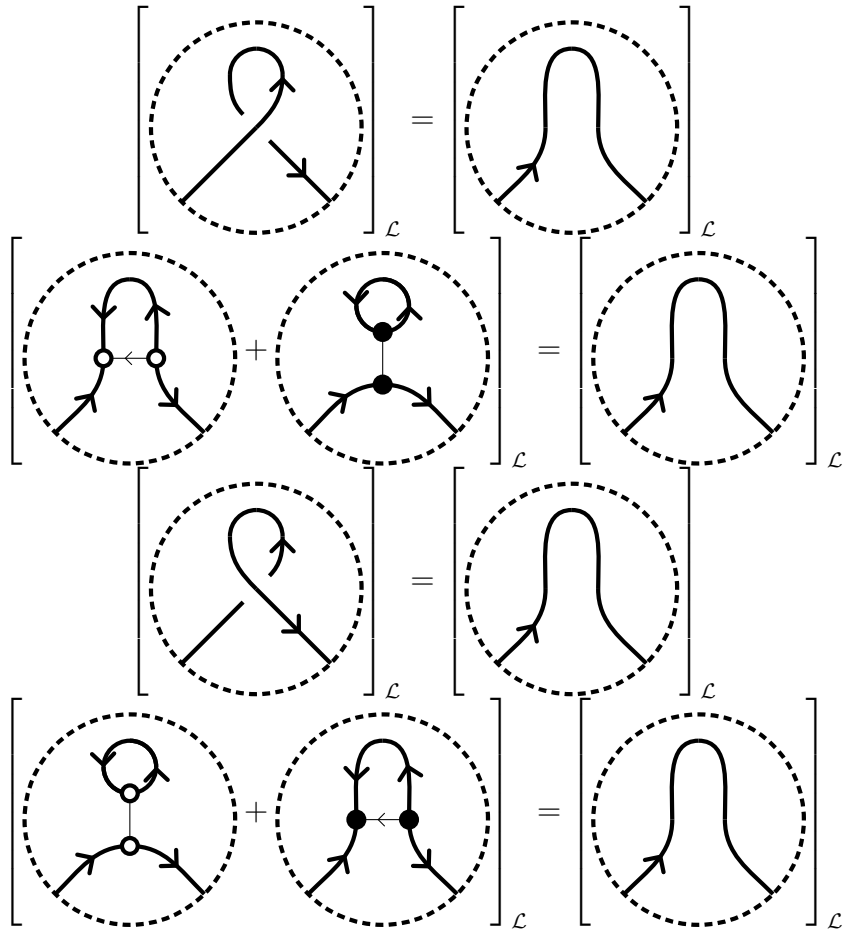
      \end{array}$}}{The label bracket $\lbr{\cdot}$ is  invariant under the first Reidemeister move}{F_R1.eps}


      \insfigureasformula{p}{\mbox{\hspace*{-0cm}$
      \begin{array}{c}
        \lbr{\ROSsadcoh{0.5}}=\lbr{\RSemptycoh{0.5}}=\lbr{\ROShapcoh{0.5}} \\
        \lbr{\RSAAcoh{0.5}+\RSABcoh{0.5}+\RSBAcoh{0.5}+\RSBBcoh{0.5}} = \lbr{\RSemptycoh{0.5}}\\
        \lbr{\RSAAncoh{0.5}+\RSABncoh{0.5}+\RSBAncoh{0.5}+\RSBBncoh{0.5}} = \lbr{\RSemptycoh{0.5}}\\
         \lbr{\ROSsadrev{0.5}}=\lbr{\RSemptyrev{0.5}}=\lbr{\ROShaprev{0.5}} \\
        \lbr{\RSAArever{0.5}+\RSABrever{0.5}+\RSBArever{0.5}+\RSBBrever{0.5}} = \lbr{\RSemptyrev{0.5}}\\
        \lbr{\RSAArev{0.5}+\RSABrev{0.5}+\RSBArev{0.5}+\RSBBrev{0.5}} = \lbr{\RSemptyrev{0.5}}\\
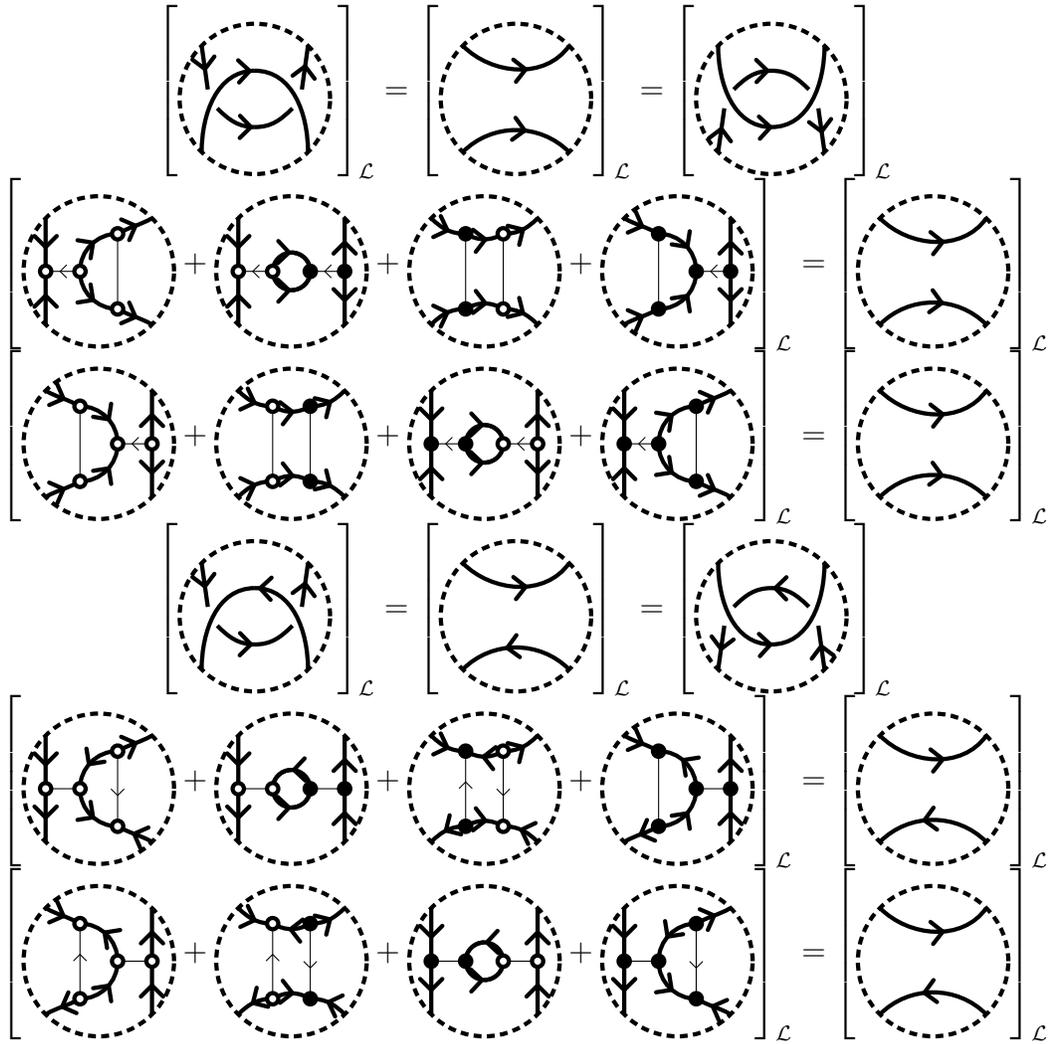
      \end{array}$}}{The label bracket $\lbr{\cdot}$ is  invariant under the second Reidemeister move}{F_R2.eps}

      \insfigureasformula{p}{\mbox{\vspace*{0cm}\hspace*{0cm}$
      \begin{array}{c}
        \lbr{\ROTleft{0.35}}=\lbr{\ROTright{0.35}}, i.e.\\
         \lbr{\RTAAAleft{0.35}+\RTAABleft{0.35}+\RTABAleft{0.35}+\RTABBleft{0.35}}+ \\
  \lbr{\RTBAAleft{0.35}+\RTBABleft{0.35}+\RTBBAleft{0.35}+\RTBBBleft{0.35}}= \\
\end{array}$} \mbox{\hspace*{0cm}$\begin{array}{c}
   \lbr{\RTAAAright{0.35}+\RTAABright{0.35}+\RTABAright{0.35}+\RTABBright{0.35}}+ \\
  \lbr{\RTBAAright{0.35}+\RTBABright{0.35}+\RTBBAright{0.35}+\RTBBBright{0.35}}~~~ \\
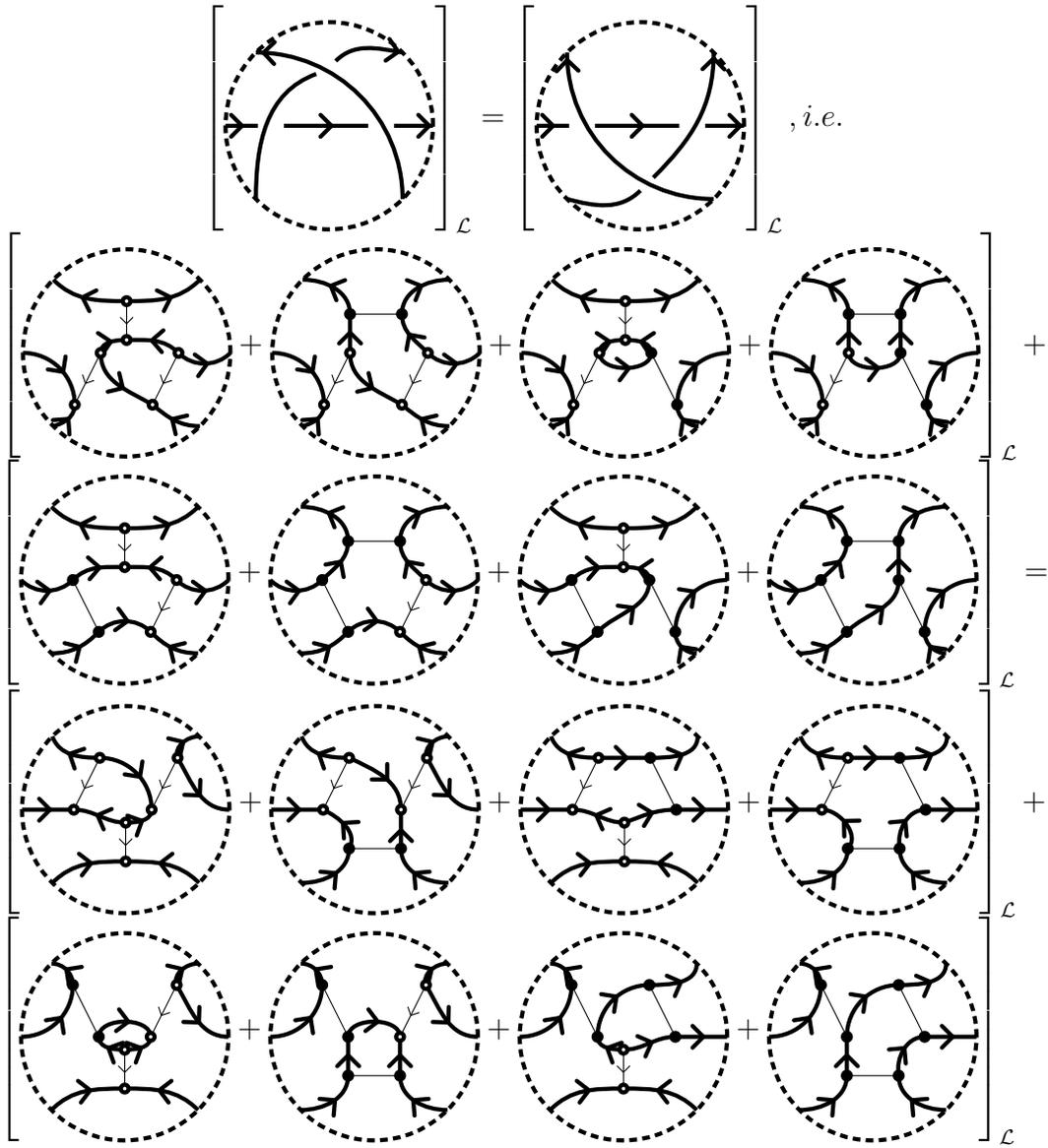
      \end{array}$}}{The label bracket $\lbr{\cdot}$ is  invariant under the third Reidemeister move}{F_R3.eps}


      \section{Some known invariants of classical knots as particular cases of the label bracket $\lbr{\cdot}$}\label{S_part_cases}

Let us show that 
the 
  Jones polynomial~$\jbr{\cdot}$ \cite{Jones}  
  and the Kuperberg bracket $\kupbr{\cdot}$~\cite{1991Kuperberg}
  can be considered as  particular cases of the label bracket $\lbr{\cdot}$.
To this end, we  reduce the relations of the label bracket $\lbr{\cdot}$ given in Figs.~\ref{F_Rules_R1.eps}--~\ref{F_Rules_s12.eps}  to relations given in Fig.~\ref{F_KP.jpg}    and Fig.~\ref{F_A2.jpg}, respectively.


\subsection{The Jones polynomial $\jbr{\cdot}$ as a particular case of the label bracket $\lbr{\cdot}$}

     Recall that the Kauffman bracket $\kbr{\cdot}$~\cite{Jones}, \cite{1987Kauffman} is defined by the  relations given in Fig.~\ref{F_KB.jpg}.
      \insfigureasformula{}{$\begin{array}{l}
       \kbr{\unknot}=1\\
       ~\\
       \kbr{\Emcr{0.35}}= A\cdot\kbr{\RSAKauf{0.35}}+A^{-1}\cdot\kbr{\RSBKauf{0.35}}\\
       ~\\
        \kbr{\unknot\cup D}=(-A^2-A^{-2})\cdot \kbr{D}\\
      \end{array}$}{Relations of the Kauffman bracket $\kbr{\cdot}$, where $\unknot$ is a diagram of unknot without crossings, and $D$ is a  classical knot  diagram}{F_KB.jpg}

 By the  writhe of an oriented classical knot  diagram $D$ with $n$ crossings we mean  the sum over all  crossings of  $D$
$$w(D)=\sum\limits_{i=1}^{n}\varepsilon(i),$$
where $\varepsilon(i)$ is a sign of the $i$-th crossing of $D$ defined by the rules given in Fig.~\ref{F_signs.eps}.

  \insfigureasformula{}{$\begin{array}{ccc}
       \crpos{0.4}&~~~&\crneg{0.4}\\[1cm]
       \varepsilon(i)=1&~~~&\varepsilon(i)=-1\\
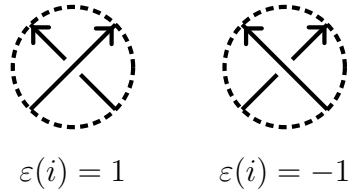
      \end{array}$}{Rules to define the sign $\varepsilon(i)$ of the $i$-th crossing}{F_signs.eps}


      \begin{teo}\label{Th_Kp}
     Let $D$ be an oriented classical knot  diagram, $w(D)$ be the  writhe of $D$, and $\kbr{D}$ be the Kauffman bracket of $D$. The Jones polynomial \cite{Jones}, \cite{1987Kauffman}
     \begin{equation}\label{Eq_XD}
    \jbr{D}=(-A)^{-3w(D)}\kbr{D}
       \end{equation}
       is a particular case of the label bracket $\lbr{D}$.
     \end{teo}
     \begin{proof}

 Taking into account the relations of the  Kauffman bracket~$\kbr{\cdot}$ given in Fig.~\ref{F_KB.jpg}, we represent formula~\eqref{Eq_XD} as the relations, see Fig.~\ref{F_KP.jpg}.

 \insfigureasformula{}{$\begin{array}{l}
      R_{JP.1}: \jbr{\unknot}=1\\
       ~\\
       R_{JP.2}: \jbr{\crposem{0.3}}= (-A)^{-3}\cdot\left(A\cdot\jbr{\RSAKauf{0.3}}+A^{-1}\cdot\jbr{\RSBKauf{0.3}}\right)\\
        ~\\
         R_{JP.3}: \jbr{\crnegem{0.3}}= (-A)^{3}\cdot\left(A\cdot\jbr{\RSAKauf{0.3}}+A^{-1}\cdot\jbr{\RSBKauf{0.3}}\right)\\
       ~\\
        R_{JP.4}: \jbr{\unknot\cup D}=(-A^2-A^{-2})\cdot \jbr{D}\\
      \end{array}$}{Relations of the Jones polynomial $\jbr{\cdot}$, where $\unknot$ is a diagram of unknot without crossings, and $D$ is a  classical knot  diagram}{F_KP.jpg}


In view of the  relations given in Fig.~\ref{F_KP.jpg}, the Jones polynomial~$\jbr{\cdot}$ can be obtained from the label bracket $\lbr{\cdot}$ as follows:
\begin{enumerate}
  \item[1)] each unoriented thin edge  incident to vertices denoted by empty circles (together with both these circles) is replaced with multiplication by $-A^{-2}$;
  \item[2)] each oriented thin edge  incident to vertices denoted by solid circles (together with both these circles) is replaced with multiplication by $-A^{-4}$;
  \item[3)] each oriented thin edge  incident to vertices denoted by empty circles (together with both these circles) is replaced with multiplication by $-A^{4}$;
  \item[4)]  each unoriented thin edge  incident to vertices denoted by solid circles (together with both these circles) is replaced with multiplication by $-A^{2}$;
\item[5)] any information on orientation of the graph edges  is removed.
\end{enumerate}

 Therefore, we obtain that relation $R_{S.1}$   gives relation  $R_{JP.2}$, see Fig.~\ref{F_Rules_s12toRK2.eps}. Relation $R_{S.2}$ gives relation  $R_{JP.3}$ in the same way.

  \insfigureasformula{}{$\begin{array}{l}
R_{S.1}:\lbr{\RScrpos{0.5}}=\lbr{\RSApos{0.5}+\RSBpos{0.5}}\\
~\\
Steps~1~-~4~give\\
~\\
{\crposem{0.5}}={-A^{-2}\cdot\RSAposem{0.5}-A^{-4}\cdot\RSBposem{0.5}}~.\\
~\\
Step~5~gives\\
~\\
 R_{JP.2}: \jbr{\crposem{0.3}}= (-A)^{-3}\cdot\left(A\cdot\jbr{\RSAKauf{0.3}}+A^{-1}\cdot\jbr{\RSBKauf{0.3}}\right).\end{array}$}{Relation $R_{JP.2}$  of the Jones polynomial~$\jbr{\cdot}$ is obtained from relation $R_{S.1}$ of the label bracket $\lbr{\cdot}$}{F_Rules_s12toRK2.eps}



Relation $R_{2.1}$  gives relation  $R_{JP.4}$, see Fig.~\ref{F_R2givesRKP4.eps}. Relations $R_{2.2}$ -- $R_{2.4}$ give relation  $R_{JP.4}$ in the same way.

\insfigureasformula{}{\mbox{\hspace*{0cm}$\begin{array}{l}
R_{2.1}: \lbr{\RSBAcoh{0.45}+\RSABcoh{0.45}+\RSBBcoh{0.45}+\RSAAcoh{0.45}} = \lbr{\RSemptycoh{0.45}}
\\
~\\
Steps~1~-~4~give\\
~\\
\RSBAcohem{0.45}+\RSABcohem{0.45}+A^{-2}\cdot\RSBBcohem{0.45}+A^{2}\cdot\RSAAcohem{0.45} = {\RSemptycoh{0.45}}~.
\\
~\\
Step~5~gives~\\
~\\
\RSABcohempty{0.4}=-A^{-2}\cdot\RSAKauf{0.4}-A^{2}\cdot\RSAKauf{0.4}~.\\
~\\
Therefore,~we~have\\
~\\
R_{JP.4}: \jbr{\unknot\cup D}=(-A^2-A^{-2})\cdot \jbr{D}.\\
\end{array}$}}{Relation $R_{JP.4}$  of the  Jones polynomial~$\jbr{\cdot}$  is obtained from relation $R_{2.1}$  of the label bracket $\lbr{\cdot}$}{F_R2givesRKP4.eps}


Taking into account relation $R_{JP.4}$ obtained above, we have that relation  $R_{1.1}$ is reduced to a trivial identity, see Fig.~\ref{F_Rules_R1toid.eps}. Relation $R_{1.2}$ is reduced to a trivial identity in the same way.

\insfigureasformula{}{$\begin{array}{l}
R_{1.1}: \lbr{\RFApos{0.7}+\RFBpos{0.7}} = \lbr{\RFempty{0.7}}\\
        ~\\
       Steps~1~-~4~give\\
        ~\\
      {-A^{2}\cdot\RFAposi{0.7}+A^{4}\cdot\RFBposi{0.7}} = {\RFempty{0.7}}~.\\
        ~\\
       Taking~into~account~relation~R_{JP.4}~and~Step~5,~we~have\\
        ~\\
          {-A^{2}\cdot(-A^{2}-A^{-2})\cdot\RFemptyunor{0.6}+A^{4}\cdot\RFemptyunor{0.6}} = {\RFemptyunor{0.6}}~.\\
\end{array}$}{Relation $R_{1.1}$ of the label bracket $\lbr{\cdot}$ is reduced to a trivial identity}{F_Rules_R1toid.eps}


Taking into account relation $R_{JP.4}$ obtained above, we have that relation  $R_{3.1}$ is reduced to a trivial identity.

Finally, it is necessary to add  normalisation relation $R_{JP.1}$. 

     \end{proof}

      \subsection{The Kuperberg bracket~$\kupbr{\cdot}$ as a particular case of the label bracket $\lbr{\cdot}$}

     \begin{teo}\label{Th_A2}
     The Kuperberg bracket~$\kupbr{\cdot}$~\cite{1991Kuperberg} defined by the  relations given in Fig.~\ref{F_A2.jpg}  is a particular case of the label bracket~$\lbr{\cdot}$.

     \insfigureasformula{}{$\begin{array}{l}
     R_{A2.1}: \kupbr{\emptyset} = 1
\\
~\\
R_{A2.2}: \kupbr{\orcircle{0.5}} =\left(q+1+q^{-1}\right)\cdot \kupbr{\emptyd{0.5}}
\\
~\\
R_{A2.3}: \kupbr{\kupscirc{0.5}} =\left(q^{\frac{1}{2}}+q^{-\frac{1}{2}}\right)\cdot \kupbr{\kupltor{0.5}}\\
~\\
R_{A2.4}: \kupbr{\kupsqv{0.5}} = \kupbr{\RSemptyrev{0.5}}+\kupbr{\emptyrevvert{0.5}}\\
~\\
R_{A2.5}: \kupbr{\crposem{0.5}} = -q^{\frac{1}{6}}\cdot\kupbr{\kupline{0.5}}+q^{-\frac{1}{3}}\cdot\kupbr{\emptycohvert{0.5}}\\
~\\
R_{A2.6}: \kupbr{\crnegem{0.5}} = -q^{-\frac{1}{6}}\cdot \kupbr{\kupline{0.5}}+q^{\frac{1}{3}}\cdot\kupbr{\emptycohvert{0.5}}\\
\end{array}$}{Relations of the Kuperberg bracket~$\kupbr{\cdot}$}{F_A2.jpg}

     \end{teo}

     \begin{proof}

 The Kuperberg bracket~$\kupbr{\cdot}$ can be obtained from the label bracket $\lbr{\cdot}$ as follows:

\begin{enumerate}
  \item[1)] each unoriented thin edge  incident to vertices denoted by empty circles is replaced with multiplication by $q^{-\frac{1}{3}}$;
  \item[2)]    both vertices denoted by solid circles and incident to oriented thin edge  are replaced with multiplication by $-q^{\frac{1}{6}}$;
       \item[3)]    both vertices denoted by empty circles and incident to oriented thin edge  are replaced with multiplication by $-q^{-\frac{1}{6}}$;
  \item[4)] each unoriented thin edge  incident to vertices denoted by solid circles is replaced with multiplication by $q^{\frac{1}{3}}$;
      \item[5)] all edges are denoted by thick lines.

\end{enumerate}

 Therefore, we  immediately obtain that relation $R_{S.1}$   gives relation  $R_{A2.5}$, while relation $R_{S.2}$ gives relation  $R_{A2.6}$.


Relations $R_{1.1}$ and $R_{1.2}$ are removed at all, since these relations correspond to the first  Reidemeister move, while the Kuperberg bracket~$\kupbr{\cdot}$ is  invariant under ambient isotopy of classical knot diagrams.

Relation $R_{2.1}$   gives relation  $R_{A2.3}$, see Fig.~\ref{F_R21givesA2.eps}, while relation $R_{2.3}$ gives relations $R_{A2.2}$ and $R_{A2.4}$, see Fig.~\ref{F_R23givesA2.eps}. 
Relations $R_{2.2}$ and $R_{2.4}$ give relations $R_{A2.2}$ -- $R_{A2.4}$ in the same way.

\insfigureasformula{}{\mbox{\hspace*{0cm}$\begin{array}{l}
R_{2.1}: \lbr{\RSBAcoh{0.45}+\RSABcoh{0.45}+\RSBBcoh{0.45}+\RSAAcoh{0.45}} = \lbr{\RSemptycoh{0.45}}
\\
~\\
Steps~1~-~5~give\\
~\\
\RSBAcohem{0.45}+\RSABcohemi{0.45}-q^{\frac{1}{2}}\cdot\RSBBcohemi{0.45}-q^{-\frac{1}{2}}\cdot\RSAAcohemi{0.45} = {\RSemptycoh{0.45}}~.
\\
~\\
Therefore,~we~have\\
~\\
R_{A2.3}: \kupbr{\kupscirc{0.45}} =\left(q^{\frac{1}{2}}+q^{-\frac{1}{2}}\right)\cdot \kupbr{\kupltor{0.45}}.\\
\end{array}$}}{Relation $R_{A2.3}$ of the Kuperberg bracket~$\kupbr{\cdot}$ is obtained from relation $R_{2.1}$  of the label bracket $\lbr{\cdot}$}{F_R21givesA2.eps}


\insfigureasformula{}{\mbox{\hspace*{0cm}$\begin{array}{l}
R_{2.3}: \lbr{\RSBArever{0.45}+\RSABrever{0.45}+\RSBBrever{0.45}+\RSAArever{0.45}} = \lbr{\RSemptyrev{0.45}}
\\
~\\
Steps~1~-~5~give\\
~\\
\RSBAreveri{0.45}+\RSABreveri{0.45}-q^{\frac{1}{2}}\cdot\RSBBreveri{0.45}-q^{-\frac{1}{2}}\cdot\RSAAreveri{0.45} = {\RSemptyrev{0.45}}~.
\\
~\\
We~use~R_{A2.3}~for~3-th~and~4-th~terms,then\\
~\\
\RSBAreveri{0.45}+\RSABreveri{0.45}={\RSemptyrev{0.45}}+\emptyrevvert{0.45}
+\left(q+1+q^{-1}\right)\cdot\emptyrevvert{0.45}~.
 \\
 ~\\
Therefore,~we~have~\\
R_{A2.4}: \kupbr{\kupsqv{0.45}} = \kupbr{\RSemptyrev{0.45}}+\kupbr{\emptyrevvert{0.45}},\\
~\\
R_{A2.2}: \kupbr{\orcircle{0.45}} =\left(q+1+q^{-1}\right)\cdot \kupbr{\emptyd{0.45}}.
\\
\end{array}$}}{Relations $R_{A2.2}$ and $R_{A2.4}$ of the Kuperberg bracket~$\kupbr{\cdot}$ are obtained from relation $R_{2.3}$  of the label bracket $\lbr{\cdot}$}{F_R23givesA2.eps}


Taking into account relations $R_{A2.3}$ and $R_{A2.4}$  of the Kuperberg bracket~$\kupbr{\cdot}$  obtained above, we have that relation $R_{3.1}$  of the label bracket $\lbr{\cdot}$ is reduced to a trivial identity.

Finally, it is necessary to add  normalisation relation $R_{A2.1}$. 

     \end{proof}

     \section{The label bracket~$\lbr{\cdot}$ for virtual knots }\label{S_VK}

     In the previous section, 
we constructed various realisations of the label bracket $\lbr{\cdot}$
as polynomial invariants of classical knots.

       In all the above cases, after reducing with respect to the  relations given in Figs.~\ref{F_Rules_R1.eps}--~\ref{F_Rules_s12.eps},
we get a polynomial by calculating values of the simplest pictures (circles).

       Nevertheless, from the reduction relations it does not follow immediately that the values of the invariant are
polynomials. Moreover, there is no evidence of having variables instead of dots.

      In the present section, we construct an invariant of {\em virtual links} valued in {\em graphs}.
Namely, we use the same  relations given in Figs.~\ref{F_Rules_R1.eps}--~\ref{F_Rules_s12.eps}
as above, but the pictures are not reduced any more to polynomials.

      The argument is similar to that from~\cite{ManturovKauffman}. More precisely, the invariant constructed in~\cite{ManturovKauffman} is valued in linear combinations of graphs and is a specification of our polynomial.

The advantage of the invariant from~\cite{ManturovKauffman} is that it is more concrete: it gives linear combinations of specific graphs and not elements of some module generated by graphs modulo relations.

       In  Section~\ref{S_knotoids}, we do the same for the case of knotoids.



    A {\em virtual knot diagram} is a planar regular 4-graph where each
crossing is either {\em classical} or {\em virtual}. In the former case, the crossing is
endowed with over/under information as usual, i.e. $\crnegat{.2}$ and $\crposit{.2}$, while  in the latter case the crossing is encircled
as $\vcr{.2}$. Besides planar graphs, we also allow circular components
having no crossings.

       Virtual knot diagrams can be oriented as well as classical knot diagrams. 

      Our strategy will be to ignore virtual crossings and consider virtual knot diagrams
as collections of classical crossings with an information how they are connected to
each other. To this end, we shall define the {\em virtual picture} as an equivalence
class of virtual knot diagrams modulo the detour move, see an example in Fig.~\ref{F_dmove}. This detour move admits local versions called the {\em virtual }
and  {\em  semivirtual Reidemeister moves}, see Fig.~\ref{F_vmove}.

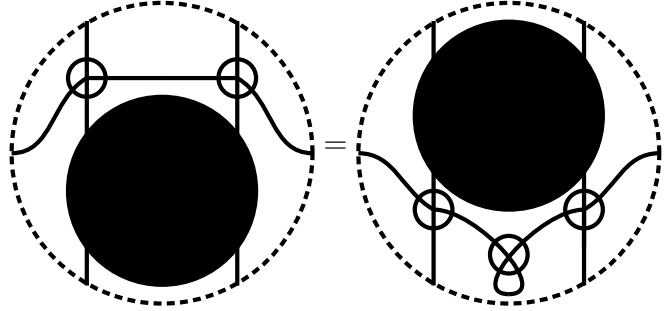
        \insfigureasformula{t}{$$\dettcir{.5}=\detbcirloop{.5}$$}{An example: the  detour move on two strands}{F_dmove}

        \insfigureasformula{b}{$\begin{array}{ccc}
                               \ROFnegv{.4}=\RFemptyv{.4}~ & \ROSsadv{.4}=\RSemptyv{.4}~ & \Rthleftv{.2}=\Rthrightv{.2}~ \\
                               &&\\
                                & \Rthleft{.2}=\Rthright{.2} &
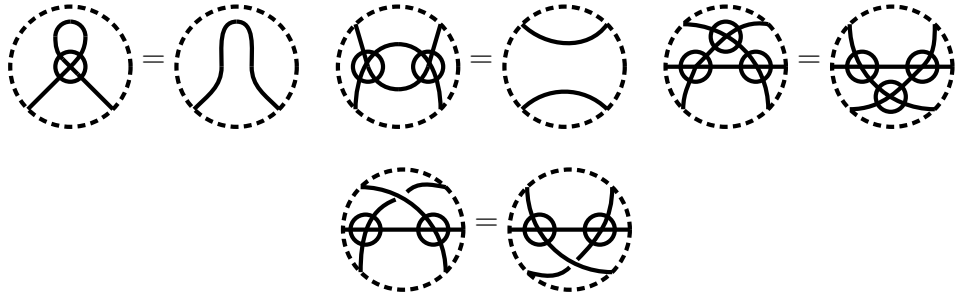
                             \end{array}$}{The  virtual (top) and semivirtual (bottom) Reidemeister moves}{F_vmove}

           In other words, we can say that a virtual picture is an equivalence class
of virtual knot diagrams modulo virtual and semivirtual Reidemeister moves.

           A {\em virtual link} is an equivalence class of virtual pictures modulo
classical Reidemeister moves, which deal with classical crossings only.
A {\em virtual knot} is an one-component virtual link.

           Now, we can generalise the definitions of the above invariants to virtual
knots and links.


           The only difference with the above text will be that the graphs of  the
label bracket~$\lbr{\cdot}$ 
are not trivalent any more; they are allowed to have 4-valent vertices  denoted by $\vcr{.2}$. Alternatively, these graphs can be treated as immersed trivalent graphs with 4-valent crossings being artifacts of the immersion.

However, let us give formal definitions.

By a \emph{virtual label graph}  we mean a planar connected graph having $2\cdot n$
3-valent vertices of types given in  Fig.~\ref{F_Vertices.eps}, where $n\in \mathbb{Z}^{+}$, and some 4-valent vertices  denoted by $\vcr{.2}$. 


 Denote by $V(G)$ 
 a module over $\mathbb{Z}$ generated by virtual label graphs  
  modulo relations~\eqref{Eq_lb_rel}, the relations given in Fig.~\ref{F_vmove} on the top and in Fig.~\ref{F_Vlg.eps}.

  \insfigureasformula{t}{$\begin{array}{c}
                               \dettcir{.3}=\detbcir{.3}~, \\
                               ~\\
                              where~each~solid~circle~is~replaced~by~the~same~fragment~of~the~form  \\
                               ~\\
                                \RSApos{0.5}~,~\RSBpos{0.5}~,~\RSAneg{0.5}~,~or~\RSBneg{0.5}~.\\
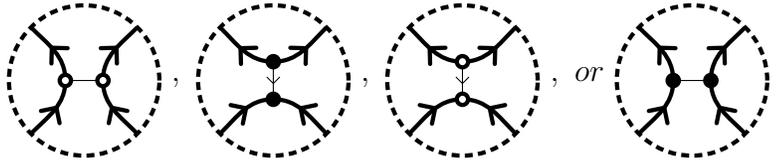
                             \end{array}$}{A version of the  semivirtual  Reidemeister move for virtual label graphs}{F_Vlg.eps}

 Let $D$ be an oriented virtual knot diagram having $n$ classical crossings.

In order to smooth classical crossings of the diagram $D$, we follow the approach given on page~\pageref{D_sccr}. States of the diagram $D$ are defined in a way given on page~\pageref{D_state}.


      The \emph{label bracket~$\lbr{D}$ of  an oriented virtual knot  diagram $D$} is the  sum~\eqref{Eq_lb_def}
        modulo relations~\eqref{Eq_lb_rel},  the relations given in Fig.~\ref{F_vmove} on the top and in Fig.~\ref{F_Vlg.eps}. Therefore, if $D$ is  an oriented virtual knot  diagram, then the label bracket $\lbr{D}$ takes
values in $V(G)$. 

The sum~\eqref{Eq_lb_def}  is taken over all possible states of the~diagram~$D$. Here $s$ is a state of the diagram $D$, and $G_s(D)$ is a virtual label graph, which is obtained as a result of smoothing each classical crossing of $D$ according to the state $s$ by the  smoothing relations $R_{S.1}$ and $R_{S.2}$ given in Fig.~\ref{F_Rules_s12.eps}.



         \begin{teo}\label{Th_inv_virt}
     The label bracket $\lbr{\cdot}$ is  invariant under isotopy of  virtual links.
     \end{teo}
     \begin{proof}
     Figs.~\ref{F_R1.eps}~--~\ref{F_R3.eps} show that the label bracket $\lbr{\cdot}$ is  invariant under all three classical Reidemeister moves. It is well known that the considered variants of orientations are sufficient. In addition, the label bracket $\lbr{\cdot}$ is  invariant under   the semivirtual Reidemeister move, see Fig.~\ref{F_Rsvirt.eps} for the proof in one of the possible cases (the proof for the remaining cases is the same). An invariance under all three virtual Reidemeister moves is obvious. 
      \end{proof}

      \insfigureasformula{t}{$\begin{array}{c}                             \lbr{\Rthleftor{.3}}=\lbr{\Rthrightor{.3}} \\
                               ~\\                                    \lbr{\detbcirwvo{0.3}+\detbcirbhn{0.3}}=\lbr{\dettcirwvo{0.3}+\dettcirbhn{0.3}}\\ \end{array}$}{The label bracket $\lbr{\cdot}$ is  invariant under the semivirtual \mbox{Reidemeister} move}{F_Rsvirt.eps}

                             \begin{teo}

                             The picture-valued Kuperberg bracket~$\kupbr{\cdot}$~\cite{ManturovKauffman}  
     defined for virtual knots by the  relations given in Fig.~\ref{F_A2.jpg}, where $q=A^{-6}$,  is a particular case of the label bracket~$\lbr{\cdot}$.
     \end{teo}

     \begin{proof} The theorem is proved by analogy with Theorem~\ref{Th_A2}.
\end{proof}

 Let us give an example of a virtual link~$L$ such that the Kuperberg bracket~$\kupbr{L}$ is a sum of graphs, i.e. takes values in pictures and can not be reduced to a polynomial in variables~$q$. Consider a virtual link~$L$, which can be embedded in the thickened torus, see Fig.~\ref{F_hex.eps}~(a). The Kuperberg bracket~$\kupbr{L}$ contains a  graph given in Fig.~\ref{F_hex.eps}~(b), which, according to the relations given in Fig.~\ref{F_A2.jpg}, admits no simplification. Indeed, all four faces of this  graph are hexagons, and all other graphs in the sum have the less number of vertices. Here by a face of a graph $G$ embedded in the 2-dimensional torus~$T^2$ we mean a connected component of the set $T^2\setminus G$.  
This example shows how to construct  any number of  virtual links embedded in the thickened torus such that  the Kuperberg bracket~$\kupbr{L}$ is a sum of graphs.
\insfigureasformula{t}{\mbox{\hspace*{0cm}$\begin{array}{cc}
   \exvktfour{.35} & \exvktfourhex{.35} \\
   &\\
   (a) & (b) \\
 \end{array}
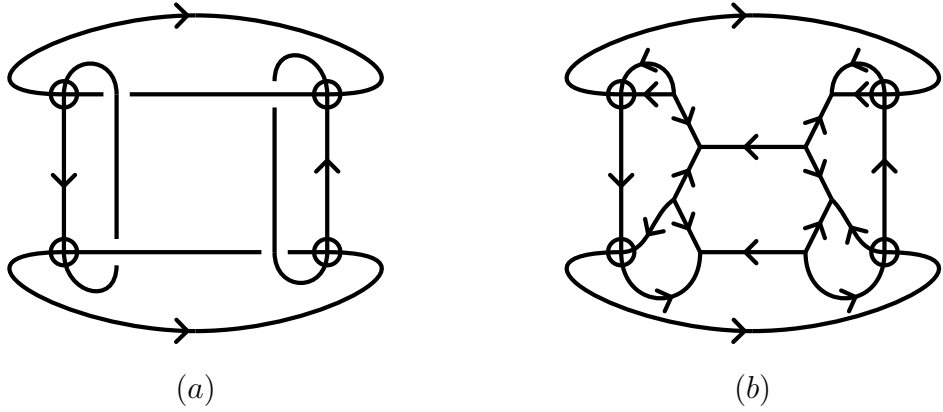
$}}{An example: a virtual link~$L$ (a) and its graph in the Kuperberg bracket~$\kupbr{L}$  having 4 hexagons (b)}{F_hex.eps}

     \section{The normalised arrow polynomial $\narbr{\cdot}$ as a particular case of the label bracket $\lbr{\cdot}$ }\label{S_VKar}

      The arrow polynomial $\arbr{\cdot}$~\cite{Kauffman} is defined by the relations given in Fig.~\ref{F_AP.jpg}. In the last relation, following~\cite{Kauffman}, we illustrate how the disoriented smoothing
is a local disjoint union of two vertices. Each vertex is denoted by an angle with
arrows either both entering the vertex or both leaving the vertex. Furthermore, the angle
locally divides the plane into two parts: one part is the span of an acute angle (of size
less than $\pi$); the other part is the span of an obtuse angle. We refer to the span of the
acute angle as the inside of the vertex, and label the insides of the
vertices with the symbol $\#$.

      In the case of classical knots,  the arrow polynomial~$\arbr{\cdot}$ can be reduced to the Kauffman bracket~$\kbr{\cdot}$.  However, in the case of virtual knots, the arrow polynomial~$\arbr{\cdot}$ takes values in linear combinations of circles with zigzags and is of particular interest.

      \insfigureasformula{}{$\begin{array}{l}
      \arbr{\crposem{0.4}}= A\cdot\arbr{\RSAKauf{0.4}}+A^{-1}\cdot\arbr{\RSBposi{0.4}}\\
        ~\\
         \arbr{\crnegem{0.4}}= A^{-1}\cdot\arbr{\RSAKauf{0.4}}+A\cdot\arbr{\RSBposi{0.4}}\\
        ~\\
               \arbr{\unknot\cup D}=(-A^2-A^{-2})\cdot \arbr{D}\\
                ~\\
               \arbr{\kupltori{0.4}}= \arbr{\kupltor{0.4}}\\
      \end{array}$}{Relations of the arrow polynomial~$\arbr{\cdot}$}{F_AP.jpg}


  \begin{teo}\label{Th_Ap}
   Let $D$ be an oriented virtual knot diagram, $w(D)$ be the  writhe of $D$, and $\arbr{D}$  be the the arrow polynomial of $D$. The normalised arrow polynomial~\cite{Kauffman}
     \begin{equation}\label{Eq_WAP}
     \narbr{D}=(-A)^{-3w(D)}\arbr{D}
       \end{equation}
      is a particular case of the label bracket $\lbr{D}$.
         \end{teo}
     \begin{proof}

     Taking into account the  relations of the   arrow polynomial $\arbr{D}$ given in Fig.~\ref{F_AP.jpg}, we represent formula~\eqref{Eq_WAP} as the  relations, see Fig.~\ref{F_WAP.jpg}.

 \insfigureasformula{}{$\begin{array}{l}
R_{W.1}: \narbr{\crposem{0.3}}= (-A)^{-3}\cdot\left(A\cdot\narbr{\RSAKauf{0.3}}+A^{-1}\cdot\narbr{\RSBposi{0.3}}\right)\\
        ~\\
        R_{W.2}: \narbr{\crnegem{0.3}}= (-A)^{3}\cdot\left(A^{-1}\cdot\narbr{\RSAKauf{0.3}}+A\cdot\narbr{\RSBposi{0.3}}\right)\\
        ~\\
          R_{W.3}~(the~loop~value~relation):     \narbr{\unknot\cup D}=(-A^2-A^{-2})\cdot \narbr{D}\\
                ~\\
            R_{W.4}~(the~reduction~relation):   \narbr{\kupltori{0.4}}= \narbr{\kupltor{0.4}}\\
\end{array}$}{Relations of the normalised arrow polynomial~$\narbr{\cdot}$}{F_WAP.jpg}


In view of the  relations given in Fig.~\ref{F_WAP.jpg}, we have to show that the normalised arrow polynomial~$\narbr{\cdot}$ is the partial case of the label bracket $\lbr{\cdot}$.

To this end, we perform the following simplifying procedures.

First, we remove all thin edges.

More precisely,

\begin{enumerate}
  \item[1)] each unoriented thin edge  incident to vertices denoted by empty circles (together with both these circles) is replaced with multiplication by $- A^{-2}$;
  \item[2)] each oriented thin edge  incident to vertices denoted by solid circles is replaced with multiplication by $- A^{-4}$, both small areas bounded by edges adjacent to the removed edge  are labeled with the symbol~$\#$, while  both endpoints of the removed edge are remained the same;
  \item[3)] each oriented thin edge  incident to vertices denoted by empty circles is replaced with multiplication by $-A^4$,  both small areas bounded by edges adjacent to the removed edge  are labeled with the symbol~$\#$, while  each endpoint of the removed edge is replaced with a solid circle;
  \item[4)]  each unoriented thin edge  incident to vertices denoted by solid circles (together with both these circles) is replaced with multiplication by $-A^{2}$.
\end{enumerate}

This gives rise to  relation $R_{W.1}$ in its form shown in Fig.~\ref{F_Rules_s12toROBE2.eps}. By analogy, relation $R_{W.2}$ can be obtained.


\insfigureasformula{h}{$\begin{array}{l}
R_{S.1}:\lbr{\RScrpos{0.5}}=\lbr{\RSApos{0.5}+\RSBpos{0.5}}\\
        ~\\
       Steps~1~-~4~give\\
        ~\\
          R_{W.1}: \narbr{\crposem{0.4}}= -A^{-2}\cdot\narbr{\RSAKauf{0.4}}-A^{-4}\cdot\narbr{\RSBposi{0.4}}.\\
\end{array}$}{Relation $R_{W.1}$ of the normalised arrow polynomial~$\narbr{\cdot}$ is obtained from relation $R_{S.1}$  of the label bracket $\lbr{\cdot}$}{F_Rules_s12toROBE2.eps}


Now, we are left to show that the loop value relation $R_{W.3}$ and the
reduction relation $R_{W.4}$  cancelling two consecutive cusps are indeed
{\em stronger} than the relations we are left after the simplifications made.

Indeed,  Fig.~\ref{F_Rules_fromR23.eps} shows that  relation $R_{2.3}$ gives the reduction relation  $R_{W.4}$ and the loop value relation $R_{W.3}$. 



Note that the value of a single circle can be imposed arbitrarily,
so we can choose it to be 1 as in the case of the arrow bracket.

Hence,  the normalised arrow polynomial~$\narbr{\cdot}$ is a specification of the label bracket $\lbr{\cdot}$.

\insfigureasformula{}{\mbox{\hspace*{0cm}$\begin{array}{l}
R_{2.3}: \lbr{\RSBArever{0.5}+\RSABrever{0.5}+\RSBBrever{0.5}+\RSAArever{0.5}} = \lbr{\RSemptyrev{0.5}}\\
        ~\\
       Steps~1~-~4~give\\
        ~\\
      {\RSBAreverar{0.5}+\RSABreverar{0.5}+A^{-2}\cdot\RSBBreverar{0.5}+A^{2}\cdot\RSAAreverar{0.5}} = {\RSemptyrev{0.5}}~.\\
        ~\\
       Therefore,~we~have~R_{W.4}~in~the~form \\
        ~\\
         \narbr{\RSBAreverar{0.5}} = \narbr{\RSemptyrev{0.5}},\\
        ~\\
        and,~taking~into~account~R_{W.4},~we~obtain~R_{W.3}~in~the~form\\
        ~\\
         \narbr{\RSABreverar{0.5}} = (-A^{2}-A^{-2})\cdot\narbr{\RSBBreverarem{0.5}}.\\
\end{array}$}}{The reduction  relation $R_{W.4}$ and the loop value relation $R_{W.3}$ of the normalised arrow polynomial~$\narbr{\cdot}$ are obtained from relation $R_{2.3}$  of the label bracket $\lbr{\cdot}$}{F_Rules_fromR23.eps}

     \end{proof}

      \section{The label bracket $\lbr{\cdot}$ for knotoids}\label{S_knotoids}

       In the present section, we shall show that the label bracket $\lbr{\cdot}$ defined in
Section~\ref{S_Def} for the case of classical knots can be defined literally in the same way for the case of knotoids. The reason is that knotoids,
as well as classical knots and virtual knots, are equivalence classes of some diagrams modulo
 Reidemeister moves.

        Knotoids were first introduced by Turaev~\cite{Turaev}. We shall distinguish between
two types of knotoids, the {\em planar ones} and the {\em spherical ones}.

        A {\em knotoid diagram} on the plane is an image of a generic immersion of
$[0,1]$ in $\mathbb{R}^{2}$ with all intersections being endowed with over/under crossing structure.

        For example, a knotoid with two crossings is shown in Fig.~\ref{F_exknotoid}.

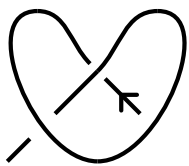
        \insfigureasformula{h}{$$\exknotoid{.4}$$}{An example: a knotoid with two crossings}{F_exknotoid}

        Any knotoid diagram on the plane can be equivalently treated as a diagram in
$S^{2}$. Hence, we can perform Reidemeister moves either in $\mathbb{R}^{2}$
or in $S^{2}$.

        A {\em planar knotoid} is an equivalence class of knotoid diagrams modulo
Reidemeister moves in $\mathbb{R}^{2}$.

        A {\em spherical knotoid} is an equivalence class of knotoid diagrams modulo
Reidemeister moves in $S^{2}$. The main difference between these two classes
is the possibility of pulling the strand over the infinite point, see an example in Fig.~\ref{F_exknotoidinf}.  Therefore,  spherical knotoids can be considered as an equivalence class of planar knotoids modulo  the "pulling over infinity" move.

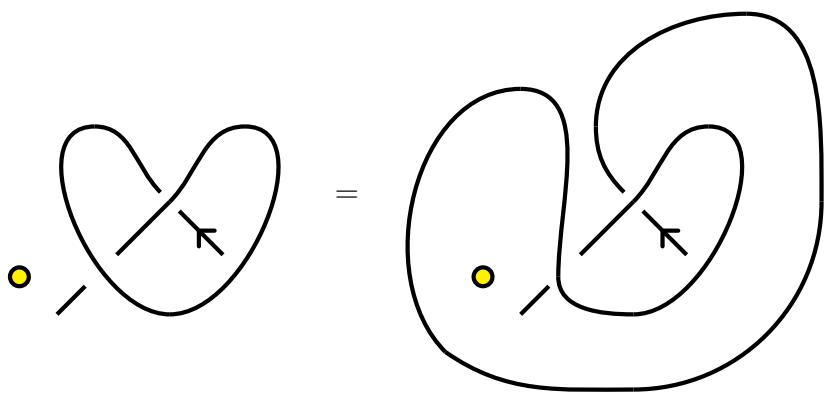
      \insfigureasformula{t}{$$\exknotoidinf{0.5}=\exknotoidinfm{0.5}$$}{An example: pulling the strand of a knotoid over the infinite point, which is marked by yellow}{F_exknotoidinf}

          Now, we are ready to define the label bracket~$\lbr{\cdot}$  for both planar and
spherical knotoids.



           The only difference with the above text will be that the graphs of the
label bracket~$\lbr{\cdot}$ are planar connected graphs having vertices of types given in  Fig.~\ref{F_Vertices.eps} and, in addition, two vertices of both types  given in  Fig.~\ref{F_Verticeskn.eps}.

\insfigureasformula{}{$\begin{array}{cc}
   \Veightkn{.33} &  \Vfivekn{0.33}\\
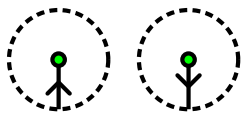
\end{array}$}{Additional types of vertices of knotoid  label graphs}{F_Verticeskn.eps}

However, let us give formal definitions.

By \emph{knotoid label graphs}  we mean  planar connected graphs having $2\cdot n$
3-valent vertices of types given in  Fig.~\ref{F_Vertices.eps}, where $n\in \mathbb{Z}^{+}$, and, in addition, two vertices of both types  given in  Fig.~\ref{F_Verticeskn.eps}. 


 Let $P(G)$ be 
 a module over $\mathbb{Z}$ generated by knotoid label graphs  
  modulo relations~\eqref{Eq_lb_rel}, and $S(G)$ be 
 a module over $\mathbb{Z}$ generated by knotoid label graphs  
  modulo relations~\eqref{Eq_lb_rel} and the "pulling over infinity" move.

  Let $D$ be an oriented knotoid diagram having $n$ classical crossings.
In order to smooth classical crossings of the diagram $D$, we follow the approach given on page~\pageref{D_sccr}. States of the diagram $D$ are defined in a way given on page~\pageref{D_state}.


      The \emph{label bracket~$\lpkbr{D}$ of  an oriented planar knotoid diagram $D$} is the sum~\eqref{Eq_lb_def}
        modulo relations~\eqref{Eq_lb_rel}. Therefore, if $D$ is  an oriented planar knotoid diagram, then the label bracket $\lskbr{D}$ takes
values in $P(G)$.

 The \emph{label bracket~$\lskbr{D}$ of  an oriented spherical knotoid diagram $D$} is the sum~\eqref{Eq_lb_def}
        modulo relations~\eqref{Eq_lb_rel} and the "pulling over infinity" move. Therefore, if $D$ is  an oriented spherical knotoid diagram, then the label bracket $\lskbr{D}$ takes
values in $S(G)$. 

The sum~\eqref{Eq_lb_def} of knotoid label graphs is taken over all possible states of the~diagram~$D$. Here $s$ is a state of the diagram $D$, and $G_s(D)$ is a knotoid label graph, which is obtained as a result of smoothing each classical crossing of $D$ according to the state $s$ by the  smoothing relations $R_{S.1}$ and $R_{S.2}$ given in Fig.~\ref{F_Rules_s12.eps}.



         \begin{teo}\label{Th_inv_plkn}
     The label bracket $\lpkbr{\cdot}$ is  invariant under isotopy of  planar knotoids.
     \end{teo}
     \begin{proof}
     Figs.~\ref{F_R1.eps}~--~\ref{F_R3.eps} show that the label bracket $\lpkbr{\cdot}$ is  invariant under all three classical Reidemeister moves. It is well known that the considered variants of orientations are sufficient.
      \end{proof}

       \begin{teo}\label{Th_inv_skn}
     The label bracket $\lskbr{\cdot}$ is  invariant under isotopy of  spherical knotoids.
     \end{teo}
     \begin{proof}
     Figs.~\ref{F_R1.eps}~--~\ref{F_R3.eps} show that the label bracket $\lskbr{\cdot}$ is  invariant under all three classical Reidemeister moves. It is well known that the considered variants of orientations are sufficient. The invariance under  the "pulling over infinity" move is obvious.
      \end{proof}

\end{document}